\documentclass[12pt,a4paper]{amsart}
\usepackage[latin1]{inputenc}
\usepackage[T1]{fontenc}
\usepackage{amssymb,amsthm,amsmath,amsfonts,mathrsfs}
\usepackage{centernot}
\usepackage{graphicx}
\usepackage{empheq}
\usepackage{comment}
\usepackage{eqnarray}

\usepackage[colorlinks=true]{hyperref}
\usepackage{cleveref}
\usepackage{geometry}
\usepackage{pdfpages}
\DeclareMathOperator\dig{div}

\newcommand{\ee}{\varepsilon}
\newcommand{\ro}{\rho}
\newcommand{\n}{\nabla}
\newcommand{\pt}{\partial_t}
\newcommand{\pal}{\partial_\alpha}
\newcommand{\pbe}{\partial_\beta}
\newcommand{\sumk}{\sum_{k=1}^{n-1} }

\newcommand{\sumab}{\sum_{\alpha,\beta=1}^d}
\newcommand{\suma}{\sum_{\alpha=1}^d}
\newcommand{\Rd}{\mathbb{R}^d}
\newcommand{\cd}{\frac{d}{2}}
\newcommand{\cdp}[1]{\frac{d}{{#1}}}

\newcommand{\tV}{\widetilde{V}}

\newcommand{\tz}{\widetilde{z}}

\newcommand{\Lj}{\mathcal{L}_j}
\newcommand{\Ij}{\mathcal{I}_j}

\newcommand{\esp}[1]{\quad \text{#1} \quad}
\newcommand{\defn}{\overset{\rm def}{=}}
\newcommand{\eq}{(\textbf{E})\;}

\newcommand{\barS}{\overline{S}}

\newcommand{\Ss}{\mathcal{S}}

\newcommand{\U}{\mathcal{U}}

\newcommand{\Dj}{\Dot{\Delta}_j}
\newcommand{\DDj}{{\Delta}_j}

\newcommand{\intd}{\int_{\mathbb{R}^d}}

\newcommand{\normede}[1]{\left\Vert #1\right\Vert_{L^2(\Rd)}}
\newcommand{\normed}[1]{\left\Vert #1\right\Vert_{L^2}}
\newcommand{\normep}[1]{\left\Vert #1\right\Vert_{L^p(\Rd)}}
\newcommand{\normepp}[1]{\left\Vert #1\right\Vert_{L^p}}
\newcommand{\normespi}[1]{\left\Vert #1\right\Vert_{L^{p_i}}}
\newcommand{\normes}[2]{\left\Vert #2\right\Vert_{L^{#1}(\Rd)}}
\newcommand{\normess}[2]{\left\Vert #2\right\Vert_{L^{#1}}}
\newcommand{\normeinf}[1]{\left\Vert #1\right\Vert_{L^\infty}}
\newcommand{\BH}[3]{\dot{B}^{#1}_{#2,#3}}
\newcommand{\NBH}[4]{\left\Vert {#4} \right\Vert_{\dot{B}^{#1}_{#2,#3}}}

\newcommand{\NB}[4]{\left\Vert {#4} \right\Vert_{B^{#1}_{#2,#3}}}

\newcommand{\LpNBH}[5]{\left\Vert {#4} \right\Vert_{L^{#5}_T(\Dot{B}^{#1}_{#2,#3})}}
\newcommand{\ttLpNBH}[5]{\left\Vert {#4} \right\Vert_{L^{#5}_t(\Dot{B}^{#1}_{#2,#3})}}

\newcommand{\TLpNBH}[5]{\left\Vert {#4} \right\Vert_{\widetilde{L}^{#5}_T(\dot{B}^{#1}_{#2,#3})}}
\newcommand{\tLpNBH}[5]{\left\Vert {#4} \right\Vert_{\widetilde{L}^{#5}_t(\dot{B}^{#1}_{#2,#3})}}

\newcommand{\B}[3]{B^{#1}_{#2,#3}}

\newcommand{\N}{\mathbb{N}}

\newtheorem{thm}{Theorem}[section]
\newtheorem{dfn}{Definition}[section]
\newtheorem{propo}{Proposition}

\newtheorem{rmq}{Remark}[section]

\newtheorem{lem}{Lemma}

\begin{document}
\title{Global existence for viscous systems of conservation laws in  \texorpdfstring{$L^p\:$}:  framework}
\author{Jean-Paul Adogbo }

\begin{abstract}
In this paper, which originates from Chapter~4 of the author's PhD thesis \cite[Chap.~4]{Adogbo25}, we investigate the global well-posedness of a class of partially diffusive hyperbolic systems in critical homogeneous Besov spaces. Our main objective is to extend the analysis developed in our previous work \cite{DanADOglob} to a functional framework in which the high-frequency components of the solution are controlled in \(L^p\)-based spaces with \(p>2\). This broader framework allows us to establish global existence under weaker smallness assumptions on the initial data and provides more precise information on the qualitative behavior of the resulting solutions. In particular, our analysis yields refined regularity. As an application, our global existence theorem applies to the multidimensional compressible Fourier-Navier-Stokes system.

\end{abstract}
\maketitle

\label{chap:glo:Lp:SHPD}
\section{Introduction}
We are concerned with multi-dimensional first order $n$-components systems in $\Rd$ for $d\ge 2,$ of the type
\begin{equation}
   \tag{\textbf{E}}
   \label{systemE}
       S^0(U)\pt U+\suma  S^\alpha(U) \pal U-\sumab \pal(Y^{\alpha\beta}(U)\pbe U)=f(U, \n U),
\end{equation}
where the unknown $U = U (t, x) \in \mathbb{R}^n$ depends on the time variable $t\in [0,\infty)$ and on the space variable $x \in \Rd$. We denote by $\pt\defn \frac{\partial}{\partial_t}$ and for all $\alpha=1,\cdots, d$,  $\pal\defn \frac{\partial}{\partial x_\alpha} $. The matrices-valued maps $S^0$, $Y^{\alpha\beta}$, $S^\alpha $ with $ \alpha,\beta =1,\cdots, d$ are assumed to be smooth, defined in some open subset $\mathcal{U}$ in $\mathbb{R}^n$ and satisfy the following assumptions (similar to \cite{DanADOglob})~:
\begin{itemize}
    \item[(\textbf{H1})] The matrices $ S^\alpha(U)$ are real symmetric, for all $U\in \U$ and $   \alpha=1,\cdots, d$.
    \item[(\textbf{H2})] For all $\alpha,\beta\in \{ 1,\cdots, d\}$, the real matrices $ Y^{\alpha\beta} $ have the following form
    \begin{equation}
        \label{form:Y}
        Y^{\alpha\beta}(U)=
        \begin{pmatrix}
0_{n_1} & 0\\
0 &  Z^{\alpha\beta}(U)
\end{pmatrix},
    \end{equation}
    where $1\le n_1\le  n$  and  $Z^{\alpha\beta}(U)\in \mathcal{M}_{n_2\times n_2}(\mathbb{R})$ with $n_2\defn n-n_1$.\\
    Moreover there exists a constant solution $ \overline{U}\in \mathcal{U}$ of \eqref{systemE} (that is $f(\overline{U},0)=0$)  such that the symbol $ Y(\overline{U},\xi)\defn \sumab Y^{\alpha\beta}(\overline{U})\xi_\alpha\xi_\beta  $  is  symmetric, has rank $n_2$ and 
    $Z(\overline{U}) \nabla_x$ is strongly elliptic in the sense
    \begin{equation}
	    \label{strong_elli}
	    \sum_{\alpha,\beta=1}^{d} \sum_{i,j= n_1+1}^n{\xi_\alpha\lambda_i\xi_\beta\lambda_jZ_{ij}^{\alpha\beta}  (\overline{U}) }  \ge c_1 \|\xi\|^2\|\lambda\|^2, \forall \xi \in \mathbb{R}^d, \lambda\in \mathbb{R}^{n-n_1}.
	    \end{equation}
     \item[(\textbf{H3})]  For all $U\in \U$, the matrix $S^0(U)$ is real symmetric  positive definite, and has the form
     \begin{align}
         \label{form:S0}
          S^0= \begin{pmatrix}
        S^0_{11} & 0\\
        0 & S^0_{22}
     \end{pmatrix},
     \esp{with}  S^0_{11} \in \mathcal{M}_{n_1\times n_1} \esp{and} S^0_{22}\in \mathcal{M}_{n_2\times n_2}.
     \end{align}
     \end{itemize}
The source term $f: (U,\eta) \in \mathbb{R}^n\times \mathbb{R}^{nd}\mapsto \mathbb{R}^n $ is assumed to be smooth, with additional conditions on
$f$ specified later on.
 The form \eqref{form:Y} was first introduced
in  \cite{MatsmuraNishida80,Kawashima83,Serre09} and it is, after a change of variables, the general form for the entropy-dissipative viscous system of conservation laws, in the sense of \cite{Serre09}. 

Condition (\textbf{H2}) reveals that the symbol $ Y(\overline{U},\xi) $ is not necessary one to one, that is,  only \textit{partial diffusion} occurs.  This is motivated by the fact that physically relevant models need not have  a stricly dissipative structure~: they may contain first-order conservation laws
as is the case for e.g. the Navier-Stokes equation (see Section \ref{sec:NSF}), the MHD equation (see 
\cite{Kawashima83}) and a lot of models coming from electromagnetism, viscoelasticity, 
and so on, see \cite{Serre97}. 

We supplement System (\textbf{E}) with an initial data $U_0\in \U$ at time $t = 0 $ and are concerned with the existence of global strong solutions in the case where $U_0$ is close to an equilibrium state~$\overline{U}$. 
\smallbreak
In \cite{DanADOlocal}, under assumptions more restrictive than (\textbf{H1})-(\textbf{H3}), we establish the existence of a local-in-time solution to \eqref{systemE} in a critical framework, with initial data $U_0-\overline{U}$ 
  belonging to spaces that do not require embedding in the space of Lipschitz functions. Furthermore, we provide a continuation criterion. Our results significantly improve upon those of D. Serre in \cite{Serr10} and Kawashima in \cite{Kawashima83}, where the initial data are required to have $C^{0,1}$ regularity.
\smallbreak
The search for sufficient conditions that ensure global existence for small perturbations of a constant solution $\overline{U}\in \mathcal{U}$ of \eqref{systemE} dates back to the 1980s with the work of Kawashima.
In his PhD thesis \cite{Kawashima83}, he exhibited a fairly simple sufficient condition that is nowadays called the (SK) (meaning Shizuta-Kawashima) condition for the global existence of strong solutions in the neighborhood of $\overline{U}$. Condition (SK) exactly says that for the linearized system, for all $\xi\in \mathbb{R}^d$, the intersection between the kernel of $Y(\overline{U},\xi)$ and the set of all eigenvectors of the symmetric first order term $ \suma S^\alpha(\overline{U})\xi_\alpha$ is reduced to $\{0\}$. This criterion has been revisited in a number of research works \cite{Yong04,Beau11}, and some works 
focused on systems violating the (SK) condition (see, for instance, \cite{BurCrinTan23,DanMucha24,Beau11} and the papers therein).

In \cite{Beau11}, K. Beauchard and E. Zuazua took advantage of techniques that originate from Kalman control theory for linear ODEs so as to construct explicit Lyapunov functionals for general partially dissipative systems (i.e. the system \eqref{systemE} where instead of second order term, we have zero order term). They also pointed out the connection between Condition (SK) and the Kalman criterion for observability in the theory of linear ODEs. The strength of Beauchard-Zuazua's method is that, at linear level, it supplies a simple way for constructing an explicit Lyapunov functional (that encodes enough information about the behavior of the solution) and a dissipation term altogether without computing the eigenvalues.
It has been noticed by R. Danchin \cite{Danc22} that the method that has been originally proposed by K. Beauchard and E. Zuazua in \cite{Beau11} works exactly the same for partial differential operators of any order (and, more generally, for homogeneous Fourier multipliers) provided one of them is skew-symmetric and the other one, nonnegative.

Although Beauchard-Zuazua's method is quite appropriate for describing the large-time behavior of solutions to the linear system of \eqref{systemE}, it does not address how to achieve maximal regularity for the components affected by diffusion (or dissipation, in the case of partially dissipative systems). To address this, in \cite{DanADOglob}, we introduced the concept of the \textit{parabolic mode}, which has better 
regularity properties than the overall solution in high frequencies. This allowed us to improve the analysis of high frequencies and to establish the existence of global strong solutions for perturbations in critical $L^2$-type spaces. The \textit{parabolic mode} serves a similar purpose as the `effective velocity' in the works of Hoff \cite{hoff2006uniqueness} and Haspot \cite{Haspot11} on compressible Navier-Stokes equations, and more recently, the `damped mode' introduced in \cite{BaratDan22M,BaratDan22D1,BaratDan23,lemarie2023parabolic}.
\smallbreak
Our goal here is to extend these results to a broader functional framework, where the high-frequency components of the solution are bounded in some $L^p$-type space with $p>2,$ rather than in 
a Sobolev space or Besov space built on $L^2.$ This extension allows for initial data with slower decay at infinity and lower regularity, compared to those previously considered.

At low frequencies, the first-order terms dominate, confining us to spaces constructed on $L^2$.
Indeed, the resulting system (i.e. \eqref{systemE}, where the second-order term is treated as a source term) does not meet the criteria established in \cite{Gues06} or those outlined by Brenner in \cite{Bren66}. Even the compressible Navier-Stokes equations fail to satisfy the aforementioned criteria (see Section \ref{sub:sec:low}).
Next, similar to the approach in \cite{DanADOglob}, we will construct a Lyapunov functional in the style of Beauchard-Zuazua. This functional will be equivalent to the norm we aim to control. To do this, we will work out new commutator and composition lemmas to manage the high-frequency components of the nonlinearities that are not contained in $L^2$.
\smallbreak

At high frequencies, the parabolic mode introduced in \cite{DanADOglob}  will allow us to diagonalize and decouple the system into a hyperbolic damped equation and a purely parabolic one. This approach, along with additional conditions ensuring the ellipticity of a specific operator, will be crucial in obtaining a priori estimates in the $L^p$ framework for high frequencies. These tools ultimately lead to the global well-posedness result for \eqref{systemE}, as stated in Theorems \ref{thm:p} and \ref{thm:p:s}.

\section{Hypotheses, methods and mains results.}
\subsection{Reduction of the problem.}
Before presenting our main results, motivating the structure assumptions that will be made on System \eqref{systemE} is in order. To this end,  if  $\overline{U}\in \mathcal{U}$ is an equilibrium state, and if we set $V\overset{def}{=}U-\overline{U}$, we can rewrite  \eq as follows.
\begin{equation}
  \label{eq:base:V}
       S^0(U)\pt V+\suma S^\alpha (U)\pal V-\sumab \pal(Y^{\alpha\beta}(U)\pbe V)=f(U,\n U) .
\end{equation}
Note that, using the decomposition 
 \begin{equation}
    \label{def:decompo:}
    \begin{split}
    V= \begin{pmatrix}
        V^1\\
        V^2
     \end{pmatrix},\quad
     S^\alpha= \begin{pmatrix}
         S^\alpha_{11} & S^\alpha_{12}\\
       S^\alpha_{21} &   S^\alpha_{22}
     \end{pmatrix}\esp{and}
   f=\begin{pmatrix}
f^1\\
f^2
\end{pmatrix},\esp{with} V^1\in \mathbb{R}^{n_1}, V^2\in \mathbb{R}^{n_2}
 \end{split}
\end{equation}
the system \eqref{eq:base:V} can be rewritten 
\begin{equation}
  \label{syste_part_hper_diff}
	 \begin{cases}
	    S^0_{11}(U)\pt V^1 +\suma \left(S^\alpha_{11}(U) \pal V^1 +  S^\alpha_{12}(U) \pal V^2\right)=f^1(U,\n U),  \\[0,5cm]
	     S^0_{22}(U)\pt V^2 +\suma \left(S^\alpha_{21}(U) \pal V^1 + S^\alpha_{22}(U) \pal V^2\right) \\
        \hspace{5cm} - \sumab\pal(Z^{\alpha \beta}(U)\pbe V^2)=f^2(U,\n U),
	 \end{cases}
	 \end{equation}
  and  the corresponding linearized system to \eqref{eq:base:V} is thus
  \begin{equation}
	\label{eq:lindiff}
 \overline{S^0}\pt V+  \sum_{\alpha=1}^d\overline{S}^\alpha \pal V-\sum_{\alpha,\beta=1}^d\overline{Y}^{\alpha\beta} \partial^2_{\alpha\beta}V=F
	\end{equation}
with \begin{align*}
     \overline{S^0}\defn S^{0}(\overline{U}),\;\; \Bar{S}^\alpha\defn S^\alpha(\overline{U}),\;\; \overline{Y}^{\alpha\beta}\defn {Y}^{\alpha\beta}(\overline{U}).
 \end{align*}
Clearly, System \eqref{syste_part_hper_diff} corresponds to 
\eqref{eq:lindiff} with $F=F^t+F^V+F^{V^2}+f$, where
 \begin{align}
      \label{F_t}
      \begin{split}
          &F^t\defn(\overline{S^0}-S(U))\pt V,\esp{} F^V\defn \suma\left(\Bar{S}^\alpha-S^\alpha(U)\right)\pal V,\\
          & F^{V^2}\defn\sumab \pal(r^{\alpha\beta}(U)\pbe V) \esp{and} r^{\alpha\beta}(U)\defn(\overline{Y}^{\alpha\beta}-Y^{\alpha\beta}(U)).
     \end{split}
 \end{align}
  In the Fourier space, System \eqref{eq:lindiff} recasts in
 \begin{align}
     \label{eq:in:Fourier:space}
     \pt \widehat{V} +\sum_{\alpha=1}^d(\overline{S}^0)^{-1}\overline{S}^\alpha \xi_\alpha\widehat{V} + \sum_{\alpha,\beta=1}^d(\overline{S}^0)^{-1}\overline{Y}^{\alpha\beta} \partial^2_{\alpha\beta} \widehat{V}=(\overline{S}^0)^{-1}\left( \widehat{F}\right).
 \end{align}
Setting $\xi=\ro \omega$ with $\omega\in \mathbb{S}^{d-1}$ and $\ro=|\xi|$, the above system rewrites
\begin{align}
     \label{eq:rew:in:Fourier:space}
     \pt \widehat{V} +(i\ro N_\omega + \ro^2 M_\omega )\widehat{V}=(\overline{S}^0)^{-1}\left( \widehat{F}\right),
 \end{align}
with $i$ a complex number satisfying $i^2=-1$ and
\begin{equation}
     \label{A:def}
N_\omega=(\overline{S}^0)^{-1}\sum_{\alpha=1}^d\overline{S}^\alpha \omega_\alpha,
\end{equation}
 \begin{equation}
     \label{B_ome:def}
 M_\omega=(\overline{S}^0)^{-1}\sum_{\alpha,\beta=1}^d\overline{Y}^{\alpha\beta}\omega_\alpha\omega_\beta= (\overline{S}^0)^{-1}\sum_{\alpha,\beta=1}^d
\begin{pmatrix}
0_{n_1} & 0\\
0 &  \overline{Z}^{\alpha\beta}\omega_\alpha\omega_\beta
\end{pmatrix}.
\end{equation}

\subsection{Shizuta-Kawashima condition.}
\label{sub:sec:SK}
For expository purpose, we assume in this paragraph that $F=0$. Taking the scalar product of \eqref{eq:lindiff} with $V$ combining with an integration by parts in the second term on the right-hand side  and using \eqref{strong_elli}, one has
\begin{align}
    \label{:sys:ene:me:1}
    \begin{split}
       \frac{d}{dt} \intd  \overline{S}^0 V\cdot V + \overline{c_1} \normed{\n V^2}^2\le 0.
    \end{split}
\end{align}
Hence, it looks that $V^1$ does not experience dissipation.   To recover it, one can take advantage of Beauchard-Zuazua method (\cite{Beau11}) and (SK) condition in \cite{Kawashima83,SK85}. To this end, fix $n-1$ positive parameters $\ee_1,\cdots,\ee_{n-1}$ (destined to be small) and set:
\begin{equation}
    \label{df:corect:0}
    I \overset{def}{=}  \sumk \ee_k {\rm  Re} (M_\omega N_\omega^{k-1} \widehat{V} \cdot M_\omega N_\omega^{k} \widehat{V} ).
\end{equation}
where $\cdot$ designates the Hermitian scalar product in $\mathbb{C}^n$. We have the following result:
\begin{lem}
\label{lem:derI}
There exists $(\ee_0,\ee_1,\cdots, \ee_{n-1})$ that can be taken arbitrarily small and $C>0$,  such that:
\begin{equation}
\label{lya:lo:in}
    \frac{d}{dt}I +\rho \sumk \ee_k |M_\omega N_\omega^{k} \widehat{V}|^2 \le C\ee_0\max(\rho,\rho^3)| M_\omega  \widehat{V}|^2.
\end{equation}
\end{lem}
 Now, taking the Hermitian product of \eqref{eq:rew:in:Fourier:space} with $\widehat{V}$ (recall that $F=0$) and using \eqref{strong_elli}  gives:
\begin{align*}
    \frac{d}{dt}\intd \overline{S^0}\widehat{V}\cdot \widehat{V}+c \ro^2 \normede{\widehat{V}^2}^2\le 0.
\end{align*}
As for all $\omega\in \mathbb{S}^{d-1}$,  $M_\omega $ has rank $n_2$ (see assumption ($\mathbf{H2}$)) one can replace $|\widehat{V}^2|$ by $|M_\omega \widehat{V}|$ on the right-hand side, up to a harmless change of C.
In the end, taking $\ee_0$ small enough, adding up \eqref{lya:lo:in}, we get 
\begin{align}
    \label{def:lya:lf}
    \frac{d}{dt} L_{\ro,\omega} (\widehat{V}) +\min(1,\ro^2)  \sum_{k=0}^{n-1} \ee_k |M_\omega N_\omega^{k} \widehat{V}|^2\le 0  \!\esp{with}\!  L_{\ro,\omega}\defn \intd \overline{S^0}\widehat{V}\cdot \widehat{V}+ \min(\ro,\ro^{-1}) I.
\end{align}
 
It is clear that $ L_{\ro,\omega}(\widehat{V})\simeq |\widehat{V}|^2$  for sufficiently small 
$\ee_1,\cdots,\ee_{n-1}.$
The question now is whether the rate of dissipation in \eqref{def:lya:lf} may be compared to $ |\widehat{V}|^2 $. The answer depends on the possible cancellation of the following quantity: 
\begin{align}
    \label{def:N:omega}
    \mathcal{N}\defn \inf \left\{ \sum_{k=0}^{n-1} \ee_k |M_\omega N_\omega^{k} x|^2, x\in \mathbb{S}^{d-1}, \omega \in \mathbb{S}^{d-1}   \right\}\cdotp
\end{align}
At this very point, the (SK)  condition comes into play:
\begin{dfn}
\label{dfn:SK}
Let $n\in \mathbb{N}\backslash\{0\}$, $N, M$  be $n\times n$ matrices with complex coefficients.
The pair $(N, M)$ is said to satisfy  the (\text{SK}) condition  if we have at the same time $M \phi =0 $ and $\lambda \phi+N \phi =0$ for some $\lambda\in \mathbb{R}$, if and only if $\phi=0_{\mathbb{R}^n}$.\\
 We say that the linear system \eqref{eq:lindiff} verifies Condition (SK) if the pair $(N_\omega, M_\omega)$ satisfies this condition for all $\omega \in  \mathbb{S}^{d-1}$.
\end{dfn}

   Roughly speaking,   (SK) condition for System \eqref{eq:lindiff} means that \textit{the intersection between the kernel of the (matrix) product of the $2$-order term and $S^{-1}$, and the set of all eigenvectors of the (matrix) product of the skew-symmetric $1-$order term and $S^{-1}$ is reduced to}  ${0}$. 
  \medbreak
  Below, we recall the connection between Condition (SK) and the Kalman criterion for observability in the theory of linear ODEs (see \cite{Beau11}).
\begin{lem}
    \label{lem:SK_diss}
    Let $N$ and $M$ be two $n\times n$ complex-valued matrices. 
        Then, the following properties are equivalent:
        \begin{enumerate}
            \item For all $\ee_0,\dots \ee_{n-1}$, we have $\sum_{l=0}^{n-1}\ee_l|MN^l\eta|^2>0$ for all $\eta\in  \mathbb{S}^{n-1}$.
            \item $(N, M)$ satisfies the Kalman rank condition: the rank of the $n^2\times n$ matrix $
            \begin{pmatrix}
                &M\\
                &MN\\
                &\dots\\
                &MN^{n-1}
            \end{pmatrix}
            $
            is equal to $n$.
            \item $(N, M)$ satisfies the (SK) condition: the intersection between ${\rm ker} M$ and the linear space of all eigenvectors of $N $ is reduced to $\{0\}$.
        \end{enumerate}
        \end{lem}
        Thanks to Lemma \ref{lem:SK_diss} and observing that the unit sphere $\mathbb{S}^{d-1}$ is compact,  one can deduce from  \eqref{def:lya:lf} and  $ L_{\ro,\omega}(\widehat{V})\simeq |\widehat{V}|^2$  that there exists some constant  $C_{\overline{V}} >0$ such that
        such that
        \begin{equation}
    \label{1:dt:L_ro_om}
    \frac{d}{dt}L_{r,\omega}(\widehat{V})+C_{\overline{V}} \min(\ro^{2}, 1) L_{r,\omega}(\widehat{V}) \le 0.
\end{equation}
 Then, integrating on $\Rd$ and using Fourier-Plancherel theorem, we conclude that there exists a Lyapunov functional $ \mathcal{L}_{r,\omega}(V) $ that is equivalent to $ \normede{V}^2 $, and an explicit rate of dissipation. More precisely there exists a positive constant $c>0$ such that 
 \begin{align}
     \label{V(xi)=eV_0}
     \widehat{V}(t,\xi)\le e^{-c\min(1,|\xi|^{2})t }  \widehat{V_0}(\xi).
 \end{align}
 The exponent $\min(1, |\xi|^2 )$ may be seen as a gain of two derivatives in low frequencies after time integration (like for the heat equation) whereas it corresponds to exponential decay for high frequencies.
 Thus, at the linear level, we observe that, provided (SK) is satisfied:
 \begin{itemize}
     \item in high frequencies, the solution decays exponentially to 0;
     \item in low frequencies, the solution behaves as the solutions of the heat equation.
 \end{itemize}
According to the above analysis, in order to get optimal dissipative estimates, it is suitable to split the solution into low and high frequencies parts. This will be achieved by means of a Littlewood-Paley decomposition (introduced in the next section). Then, an important part of our work will consist to bound separately the low and high frequencies of the solution to the nonlinear system \eqref{eq:base:V}.

\subsection{Parabolic mode.}
\label{sub:sec:Para:mode}
In order to explain the supplementary conditions that we need in order to get our main results, let us introduce, as in \cite{DanADOglob}, a special combination of the components of the solution, that is expected
to have  better properties in high frequencies than the overall solution. In fact, although estimate \eqref{V(xi)=eV_0} gives the overall behavior of the solution of the linearized system, it does not fully reflect that a part of the solution gains some regularity through the time evolution. To see this, we can note from \eqref{syste_part_hper_diff}
that $V^2$ satisfies a fully parabolic equation with some lower order source term depending on $V^1.$
 Leveraging  a suitable `parabolic mode' will help us to recover this information. For that purpose, introducing some notation is in order. For $\xi=(\xi_1,\cdots, \xi_d)\in \mathbb{R}^{d}\backslash\{0\}$, we set
\begin{align}
    \label{def:S:21:22:Z}
    \mathcal{Z}(\xi)\defn \frac{1}{|\xi|^2}\sumab\overline{Z}^{\alpha \beta}\xi_\alpha\xi_\beta \esp{and, for}
    1\leq j,k\leq 2,\quad\!\!
     S_{jk}(\xi)\defn \frac{1}{|\xi|} \sum_{\alpha}\overline{S}^\alpha_{jk}\xi_\alpha.
     \end{align}
From now on, we agree that for all tempered distribution $z$,
\begin{align}    \label{def:S:21:22:ZD}
 \mathcal{Z}(D) z \defn \mathcal{F}^{-1}(Z(\xi) \widehat{z}) 
  \esp{and, for}    1\leq j,k\leq 2,\quad\!\!
    \Ss_{jk}(D) z \defn \mathcal{F}^{-1}(S_{jk}(\xi) \widehat{z}).\end{align}
 As $Z(\omega)$, for all $\omega\in \mathbb{S}^{d-1}$ is  an invertible matrix (in fact from \eqref{strong_elli} the Kernel of $ Z(\omega)\in \mathcal{M}_{n_2} $ is $\{0\}$), we define the so-called \textit{parabolic mode}  $W$ 
 on the Fourier side  as follows for all $\xi\neq 0$:
 \begin{equation}
    \label{W:def}
    \begin{split}
    \widehat{W}(\xi)&\overset{def}{=} i|\xi|^{-1}(\mathcal{Z})^{-1}(\xi)\left( \Ss_{21}(\xi)\widehat{V}^1(\xi)+  \Ss_{22}(\xi)\widehat{V}^2\right) +\widehat{V}^2(\xi).
     \end{split}
\end{equation}
Then, System  \eqref{syste_part_hper_diff}$_2$ may be rewritten
\begin{multline}    \label{rel:W:V2}
   \overline{S}^0_{22}\pt V^2 -\Delta\mathcal{Z}(D)W=\suma\left((\barS_{21}-S^\alpha_{21}(U))\pal V^1+(\barS_{22}-S^\alpha_{22}(U)\pal V^2)\right)\\
        + (\overline{S}^0_{22}-S^0_{22}(U))\pt V+ \sumab \pal\left((\overline{Z}^{\alpha\beta}-Z^{\alpha\beta}(U))\pbe V^2\right)+f^2.
\end{multline}
Let us set $g \defn g^t+g^{V^2}+g^{V}+g^{f}+\overline{g}^V$ with
\begin{align}
    \label{def:g}
    \begin{split}
        &g^t=(\overline{S}^0_{22}-S^0_{22}(U))\pt V^2,\;\; g^{V^2} = \sumab \pal\left((\overline{Z}^{\alpha\beta}-Z^{\alpha\beta}(U))\pbe V^2\right),\\
        &g^V= \suma\left((\barS_{21}-S^\alpha_{21}(U))\pal V^1+(\barS_{22}-S^\alpha_{22}(U)\pal V^2)\right),\;\;   g^f= f^2,\\
        & \Bar{g}^{V}=i |D|^{-1}\overline{S}^0_{22}\mathcal{Z}^{-1}(D)\pt\left( \Ss_{21}(D) V^1+  \Ss_{22}(D) V^2 \right).
    \end{split}
\end{align}

Then,  starting from \eqref{rel:W:V2}, we see that $W$ satisfies the following parabolic equation
\begin{equation}
    \label{eq:W}
    \overline{S}^0_{22}\pt W -\Delta\mathcal{Z}(D) W= g.
\end{equation}

From \eqref{W:def},  we have the following equivalent definitions of $W$:
\begin{align}
    \label{W:def:2}
    W &=i |D|^{-1}(\mathcal{Z}^{-1}(D)\left( \Ss_{21}(D) V^1+  \Ss_{22}(D)V^2\right)  +V^2\\
    \label{W:def:22}
    &= (-\Delta)^{-1} \mathcal{Z}^{-1}(D) \overline{S}^0_{22}\pt V^2+(-\Delta)^{-1}\mathcal{Z}^{-1}(D)\left(g^t+g^V+g^{V^2}+f^2\right)\cdotp
\end{align}
The parabolic equation \eqref{eq:W} reveals that $W$ should gain two derivatives in space 
compared to the data, after time integration if working in a suitable space and making  further assumptions on $  \overline{S}^0_{22} $ and $\mathcal{Z}(D)$ (see Assumption (\textbf{H6}) below).

\subsection{Transfert of integrability from the parabolic mode to the hyperbolic mode. }
\label{subsec:hyper:mode}
Taking advantage of the parabolic mode, we can obtain an $L^p$ type control on the high frequencies of $V^1.$  To see that, we set for all $\alpha=1\cdots d$
\begin{align}
\label{def:A:B}
    \begin{split}
      A^\alpha_{11}(U)  \defn (S^0_{11}(U))^{-1}S^\alpha_{11}(U).
    \end{split}
\end{align}
The equation satisfied by $V^1$ reads
\begin{align}
   \label{e:V1:f(A)}
   \pt V^1+\suma A^\alpha_{11}(U)\pal V^1 +\suma (S^0_{11}(U))^{-1}S^\alpha_{12}(U)\pal V^2=(S^0_{11}(U))^{-1}f^1.
\end{align}

Assume now 
\begin{itemize}
\item [(\textbf{H4})] For all $U\in \U$,  the matrices 
$A^{\alpha}_{11}$ with $\alpha=1\cdots,d$ {are commuting diagonalizable with real spectrum}.
\end{itemize}
Assumption $(\mathbf{H4})$ is motivated by the well-posedness of  \eqref{e:V1:f(A)} 
\emph{seen as a system with unknown $V^1$},  in the $L^p$  framework for $p\neq 2$. In fact, as
observed by O. Gu\`es and J. Rauch in \cite{Gues06},  a necessary and sufficient condition for the linear system of \eqref{e:V1:f(A)} to be well-posed in general $L^p$ spaces is that the diagonalizable matrices $ A^{\alpha}_{11}(\overline{U})$ commute and have real spectrum (see also \cite{Bren66}).  
Of course, this commutation property allows us to diagonalize the matrices $ A^{\alpha}_{11}(\overline{U}) $ in the same basis.  Here we assume that there exist smooth functions $P:\U\to {\rm GL}_n(\mathbb R)$
and $H_\alpha:\U\to {\rm diag}_n(\mathbb R)$ for $\alpha=1\cdots,d$  such that for all $U\in \U $
\begin{align}
    \label{chang:basis}
   A^\alpha_{11}(U)=P^{-1}(U)H^\alpha(U)P(U).
\end{align}
Hence, setting $z=P(U)V^1,$  System \eqref{e:V1:f(A)}  becomes 
\begin{multline}    \label{eq:V1=f(H)}
\pt z  + \suma H^\alpha(U)\pal z= P(S^0_{11}(U))^{-1}f^1- \suma P(U) ( S^0_{11}(U))^{-1} S^\alpha_{12}(U)\pal V^2\\
+\pt(P(U))\,V^1+\suma H^\alpha(U)\pal (P(U))V^1.
\end{multline}
The last term will cause a loss of one derivative if $P$ is not a constant with respect to $V^1$. Moreover, even if $ P$ depends only on the component $U^2$, we are in trouble concerning how to control the last two terms on the left-hand side of \eqref{eq:V1=f(H)}.
Hence, in addition, we assume that~:
\smallbreak
\begin{itemize}
\item [(\textbf{H5})]  The function  $P$ is constant, that is,  independent of $U^1$ and $U^2$. 
\end{itemize}
\vspace{0.3cm}
Then, using   \eqref{W:def:2} to express $V^2$  in terms of $W$ in the equation \eqref{eq:V1=f(H)} yields 
\begin{align}
 \label{eq:V1=f(W)}
      \pt z  + \suma H^\alpha(U)\pal z+B(D)z=h
\end{align}
with
\begin{align}
\label{defin:B:(D)}
\begin{split}
    B(D) &\defn -i P\suma (\overline{S}^0_{11})^{-1}\overline{S}^\alpha_{12}|D|^{-1}\pal Z^{-1}(D)\mathcal{S}_{12}(D)P^{-1}\\
    &= P(\Bar{S}^0_{11})^{-1} \Ss_{12}(D)\mathcal{Z}^{-1}(D)\Ss_{21}(D)P^{-1}.
    \end{split}
\end{align}

The nonlinear term $h$ is defined as $h=h^{W}+h^{V^2}_0+ h^{V^2}_1+h^f$ with
\begin{align}
    \label{def:ht:V:V2}
    \begin{split}
   &h^{V^2}_1\overset{def}{=} \suma P \biggl((\overline{S}^0_{11})^{-1}\overline{S}^\alpha_{12} - (S^0_{11}(U))^{-1}S^\alpha_{12}(U)\biggr) \pal V^2 ,\;\;  h^f\defn P (S^0_{11}(U))^{-1} f^1,\\
   & h^{V^2}_0\defn -P (\overline{S}^0_{11})^{-1}  \Ss_{12}(D)\mathcal{Z}^{-1}(D)\Ss_{22}(D)V^2, \;\; h^{{W}}\defn -\suma P(\overline{S}^0_{11})^{-1}\overline{S}^\alpha_{12}\pal W .
   \end{split}
\end{align}

In order to pursue our analysis we make the following assumption~:
\smallbreak
\begin{itemize}
    \item[(\textbf{H6})] 
   The matrix $B(\omega)$ is a (constant) diagonal matrix with positive entries. 
\end{itemize}
\vspace{0.3cm}
Remark that it implies that the rank of the matrix $S_{21}(\xi)$ is independent of $\xi$ and is equal to $n_1$. Unlike in \cite{DanADOglob}, where we only required the pair $$\bigl((\overline{S}^0_{11})^{-1} S_{11}(\omega), ( \overline{S}^0_{11})^{-1} S_{12}(\omega)Z(\omega)^{-1}S_{21}(\omega)\bigr)$$ to satisfy the (SK) condition (equivalently that
$(N_\omega, M_\omega)$ satisfies the (SK) condition), it is unclear how to obtain global estimates on $V^1$ in the $L^p$ framework without  (\textbf{H6}). 

 We note that $z$ verifies a damped hyperbolic system. Therefore  (\textbf{H6}) will help us to derive exponential decay on $z$
 in $L^p$ type spaces. Keep in mind that in the homogeneous
Besov spaces setting (or Lebesgue spaces), bounding $z$ or $V^1$ is equivalent as one can
go from $V^1$ to $z$ or from $z$ to $V^1$ by means of an invertible constant matrix.
\begin{rmq} 
\label{rmq:cnd:Z:B:H}
A few remarks are in order:
\begin{itemize}
   
   \item Assume that $S^0_{11}$ is the identity.  Then, Condition (\textbf{H4})  is reduced to 
   the fact that   the matrices
$ S^\alpha_{11}(U)$  commute, in accordance with the work by  Ph. Brenner in \cite{Bren66}.
 \item 
 As pointed out in \cite{BaratDan23} (see also Lemma \ref{lem:equi:SK:positi:V1}), 
 assuming the matrix $ S_{12}(\omega)Z(\omega)^{-1}S_{21}(\omega) $ to be positive is equivalent to the (SK) condition when $ \overline{S}^0$ is identity, $P$ is an orthogonal matrix and $ \overline{S}^\alpha_{11}=0$, for all $\alpha=1,\cdots, d$. 
 \item 
 Assume that the matrices $\overline{S}^{0}_{11}$, $S^\alpha_{11}(U)$ ($\alpha=1,\cdots, d$) and $ S_{12}(\xi)Z(\xi)^{-1}S_{21}(\xi) $ 
 commute. This implies assumption (\textbf{H4}) and that $P$ is an orthogonal matrix. If we additionally assume (\textbf{H5}), then assumption (\textbf{H6}) is met provided that the rank of the matrix $S_{21}(\xi)$ equals $n_1$  and $\xi\mapsto   S_{12}(\xi)Z(\xi)^{-1}S_{21}(\xi) $ is a constant matrix.

In particular, when $n_1=1$, the
 assumptions (\textbf{H4}),  (\textbf{H5}) and (\textbf{H6}) are satisfied if the matrix $S_{21}(\xi)\neq 0$
and $\xi\mapsto   S_{12}(\xi)Z(\xi)^{-1}S_{21}(\xi) $ is constant. These conditions are, for example, satisfied by the compressible Navier-Stokes system (see Section \ref{sec:NSF}).
   \end{itemize}
\end{rmq}

Finally, as in \cite{DanADOglob}, we assume for simplicity that the external force $f$ has the following structure: 
\begin{itemize}
\item [(\textbf{H7})]   $f$ has components 
      \begin{equation}
\label{globa:f=f(f1,f2)}
    f^1(0,\n U)\defn 0_{n_1} \esp{and} f^2(U,\n U)\defn 
Q(U,\n U,\n U),
  \end{equation}
  with  $Q$ is quadratic with respect to $\n U$ (that is,  $Q$  is a combination of terms of type  $\upsilon(U)\n U\bigotimes\n U$).
\end{itemize}

\subsection{Notations and main results}
Our preliminary analysis of System \eqref{systemE} suggests working with analysis tools that would facilitate the decomposition of the frequencies. Clearly, as shown in \cite{DanADOglob} and explained above, achieving optimal results requires using different types of regularity for the low and high frequency components of the solution.
In this regard, the Littlewood-Paley decomposition and homogeneous Besov spaces emerge as natural tools for analyzing our model.
Here, following \cite[Chapter 2]{HajDanChe11}, 
we define a homogeneous Littlewood-Paley decomposition $ (\Dj)_{j\in \mathbb{Z}}$ setting 
\begin{align}
\label{def:varphi:chap3}
    \Dj \defn \varphi(2^{-j}D) \esp{with} \varphi(\xi)=\chi(\xi/2)-\chi(\xi)
\end{align}
where $\chi$ stands for a smooth function with range in $[0,1]$, supported in the open ball $B(0,4/3)$
and such that $\chi=1$ on the closed ball $\overline{B}(0,3/4)$. From now on, we set 
for all tempered distribution $z$: 
\begin{equation}
    \label{def:z:lf:hf}  z_j\defn \Dj z;\;\;\;
    z^{l,N_0}\overset{def}{=} \sum_{j\le N_0}z_j;\;\;\;  z^{h,N_0} \overset{def}{=} \sum_{j>N_0} z_j \esp{and}  z^{H,N_0} \overset{def}{=} \sum_{j>N_0+2} z_j,
\end{equation}
for some large enough nonnegative integer $N_0$ that will depend   on the matrices of System \eqref{systemE},  to be determined later\footnote{The values of $N_0$ can be explicitly determined in our computations. It is essential to choose it suitably large to ensure that some linear source terms can be absorbed.}. 
\medbreak
$z^{l,N_0}$ and $z^{h,N_0}$ correspond respectively to the low and high frequency parts of $z$. Note also that $z^{H,N_0}$ refers to the high-frequency part of $z$ while $ z-z^{H,N_0}$ corresponds to the  associated low-frequency part.
The reader is referred to the section \ref{appendix:LP:loc} for the definition of homogeneous Besov spaces
$\dot B^s_{p,r}$ and to some of their properties. 
We recall the "truncated" Besov semi-norms:
\begin{equation}
    \label{def:norm:lf:hf}
    \NBH{s}{p}{r}{z}^{l,N_0} = \Vert\left(  2^{js} \Vert z_j \Vert_{L^p}\right)_{j\le N_0 +1} \Vert_{l^r} \esp{and}  \NBH{s}{p}{r}{z}^{h,N_0} = \Vert\left(  2^{js} \Vert z_j \Vert_{L^p}\right)_{j> N_0 +1} \Vert_{l^r} 
\end{equation}
so that for all $s^{\prime}, s\in \mathbb{R}$ and $ p^{\prime}\le p$
\begin{equation}
    \label{link:lf:hf}
     \NB{s}{p}{r}{z^{l,N_0}}\lesssim  \NB{s}{p}{r}{z}^{l,N_0} \esp{and}  \NB{s}{p}{r}{z^{h,N_0}}\lesssim  \NB{s}{p}{r}{z}^{h,N_0} +  \NB{s^{\prime}}{p^{\prime}}{\infty}{z}
\end{equation}
and 
\begin{equation}
\label{zh=f(zh)}
     \NB{s}{p}{r}{z^{H,N_0}}\lesssim  \NB{s}{p}{r}{z}^{h,N_0}.
\end{equation}

Since we must use an $L^2$ type framework for the low frequencies of the solution of the system \eqref{systemE} (see section \ref{sub:sec:low}), we should keep the same Besov space as in \cite{DanADOglob} for this regime: $ V^{l,N_0}\in \BH{\cd-1}{2}{1}$. As regards the space for the high frequencies,  we shall take $ V^{2,H,N_0}_0\in  \BH{\frac{d}{p}}{p}{1} $ (or, equivalently  $W^{H,N_0}_0\in  \BH{\frac{d}{p}}{p}{1} $) and we expect to control the high frequencies of $W^{H,N_0}$ in $L^\infty(\mathbb{R}_+,\BH{\frac{d}{p}}{p}{1})\cap L^1(\mathbb{R}_+,\BH{\frac{d}{p}+2}{p}{1})$,  provided we succeed in bounding $g^{H,N_0}$ in $ L^1(\mathbb{R}_+,\BH{\frac{d}{p}}{p}{1})$. 
In order to bound in $L^1(\mathbb{R}_+,\BH{\frac{d}{p}}{p}{1})$ the nonlinear terms appearing in  \eqref{eq:W} (for instance the first term of $g^V$), it is important that $V^{1,H,N_0}$ has one more degree of regularity than $V^{2,H,N_0}$. Hence, $V^{1,H,N_0}_0$ must be in $\BH{\frac{d}{p}+1}{p}{1}$.
\medbreak
Our first global existence result for System (3) exactly corresponds to this framework.
\begin{thm}
    \label{thm:p}
    Let $d\ge 2$, $ 2\le  p_1\le p_2< 2d $ and $p_2\le \min( 4,\frac{2d}{d-2})$.  Let the (SK) condition for the pairs ($N_\omega,M_\omega$)$_{\omega \in \mathbb{S}^{d-1}}$, defined in \eqref{A:def}-\eqref{B_ome:def} and the hypotheses (\textbf{H1})-(\textbf{H7}) be in force.
        There exists an integer $N_0$ (depending on the matrices involved in system \eqref{systemE}) and a positive constant $c_0$ such that for any data $V_0=(V_0^1,V_0^2)$ satisfying
    $V^1_0\in \BH{\frac{d}{p_1}+1}{p_1}{1}$ and $ V^2_0\in \BH{\frac{d}{p_1}}{p_1}{1},$
    if, in addition,    
    \[   \NBH{\frac{d}{2}-1}{2}{1}{V_0}^{l,N_0} +\NBH{\frac{d}{p_2}+1}{p_2}{1}{V^1_0}^{h,N_0}+ \NBH{\frac{d}{p_2}}{p_2}{1}{V^2_0}^{h,N_0} \le c_0, \]
    then \eqref{eq:base:V} has a unique global-in-time solution in the space 
    $\mathbb{E}_{p_1}\cap \mathbb{E}_{p_2}$ with $\mathbb{E}_{p_i}$ defined by
    \begin{align}
        \label{def:E}
        \begin{split}
            V^{l,N_0} \in & \mathcal{C}_b(\mathbb{R}^+,\BH{\frac{d}{2}-1}{2}{1})\cap L^1(\mathbb{R}^+,\BH{\frac{d}{2}+1}{2}{1}); \;\;  V^{1,H,N_0}\in \mathcal{C}_b(\mathbb{R}^+,\BH{\frac{d}{p_i}+1}{p_i}{1})\cap L^1(\mathbb{R}^+,\BH{\frac{d}{p_i}+1}{p_i}{1}) \\
            & V^{2,H,N_0}, W^{H,N_0}\in \mathcal{C}_b(\mathbb{R}^+,\BH{\frac{d}{p_i}}{p_i}{1})\cap L^1(\mathbb{R}^+,\BH{\frac{d}{p_i}+2}{p_i}{1});\;\; \pt V\in L^1(\mathbb{R}^+,\BH{\frac{d}{p_i}}{p_i}{1}) 
        \end{split}
    \end{align}
with $W$ defined in \eqref{W:def:2}.
\medbreak
Moreover, we have the following bound for $1\le i\le 2$:
\begin{align*}
    \mathbb{V}_{p_i}&\lesssim  \mathbb{V}_{{p_i},0}\defn \NBH{\frac{d}{2}-1}{2}{1}{V_0}^{l,N_0} +\NBH{\frac{d}{p_i}+1}{p_i}{1}{V^1_0}^h+ \NBH{\frac{d}{p_i}}{p_i}{1}{V^2_0}^h
\end{align*}
with
\begin{multline*}
      \mathbb{V}_{p_i}\defn \LpNBH{\frac{d}{2}-1}{2}{1}{V}{\infty}^{l,N_0}+\LpNBH{\frac{d}{2}+1}{2}{1}{V}{1}^{l,N_0} + \Vert V^1\Vert_{L^\infty\cap L^1(\BH{\frac{d}{p_i}+1}{p_i}{1})}^{h,N_0}\\
    + 
    \LpNBH{\frac{d}{p_i}}{p_i}{1}{(W,V^2)}{\infty}^{h,N_0} + \LpNBH{\frac{d}{p_i}+2}{p_i}{1}{(W,V^2)}{1}^{h,N_0}
    +\LpNBH{\frac{d}{p_i}}{p_i}{1}{\pt V}{1}.
\end{multline*}
    
\end{thm}
\begin{rmq}
    \label{rmq:srong:cnd:p}
    This theorem improves our result obtained in \cite{DanADOglob}, as
here one may take the high frequencies of the data belonging to Besov spaces related to $L^p$ (with $p\ge 2$), which are larger. However, our structural assumptions (\textbf{H1})-(\textbf{H7}) are stronger than those in \cite{DanADOglob}. As in \cite{BaratDan23}, we have some restrictions on $p_1$ and $p_2$, that
is, $ 2\le p_1\le p_2\le \max(2d,2d/(d-2))$. This is due to the interactions between low and high frequencies
in the nonlinear terms. Let us also underline that we need smallness only in space related to $L^{p_2}$, so $V_0$ has to be taken so that $\NBH{\frac{d}{p_1}+1}{p_1}{1}{V^1_0}^{h,N_0}+ \NBH{\frac{d}{p_1}}{p_1}{1}{V^2_0}^{h,N_0}$ is large.
\end{rmq}
\begin{rmq}
    \label{uniq:p2}
    Uniqueness holds in a space larger than $\mathbb{E}_{p_i}$: as we will see in Section \ref{sec:uni:p}, uniqueness holds true in the space $\mathbb{F}_T$ of functions $(V^1,V^2)$ such that $V^1\in \mathcal C([0,T];\BH{\cdp{p_2}}{p_2}{1})$ and $ V^2\in \mathcal C([0,T];\BH{\cdp{p_2}-1}{p_2}{1})\cap L^1_T(\BH{\cdp{p_2}+1}{p_2}{1})$ with $\TLpNBH{\cdp{p_2}}{p_2}{1}{V^1}{\infty}\le c$ (for some positive small constant $c$ depending only on $d$, $p_2$ and on the matrices of System \eqref{eq:base:V}) and $2\le p_2<2d$.
\end{rmq}

In the previous results, we had some  restrictions on
$p_i$ (namely, $p_i \le  \max(4, \frac{d}{d-2})$). We cannot choose $p_i$ to be arbitrarily large, because of the behavior in low frequencies. This  comes
from some restrictions in  the use of the
paraproduct laws (see the section \ref{appendix:LP:loc}). To overcome them, one can add some assumptions on
the high frequencies of the data,  which leads to  the following statement:
\begin{thm}
    \label{thm:p:s}
    Let $d\ge 2$ and assume that the Condition (SK) for the pairs ($N_\omega,M_\omega$)$_{\omega \in \mathbb{S}^{d-1}}$, defined in \eqref{A:def}-\eqref{B_ome:def},  and the hypotheses  (\textbf{H1})-(\textbf{H7})  are satisfied. Assume that $V^1_0\in \BH{\frac{d}{p_1}+1}{p_1}{1}$, $ V^2_0\in \BH{\frac{d}{p_1}}{p_1}{1} $ for some $p_1\ge 2$. Let $p_2\ge p_1$. Then there exists an integer $N_0$ such that if we make the following additional  (lower order) assumptions 
    \[  V^2, V^{1,l,N_0}\in \BH{s}{2}{r} ;\;\; V^{1,h,N_0}\in   \BH{s+1}{2}{r} \]
    for some $r\in [1,\infty]$ and $s$ such that
    \begin{align}
        \label{def:s:r:p2}
       -\min(1,\frac{d}{p_2})<  s<\frac{d}{2}-1 \;\; \text{and}\;\; r>1 \esp{or} s=\frac{d}{2}-1 \;\; \text{and}\;\; r=1.
    \end{align}
    then there exist some positive constants $c$ and $C$ such that if 
   \begin{align*}
       \NBH{s}{2}{r}{(V^2_0,V^{1,l,N_0}_0)}+  \NBH{s+1}{2}{r}{V^{1,h,N_0}_0} +\NBH{\frac{d}{p_2}}{p_2}{1}{V^{2,h,N_0}_0}+ \NBH{\frac{d}{p_2}+1}{p_2}{1}{V^{1,h,N_0}_0}\le  c
   \end{align*}
     then \eqref{eq:base:V} has  a global-in-time solution in the space
     $\mathfrak{E}_{p_1}^s\cap \mathfrak{E}_{p_2}^s$ with   $\mathfrak{E}_{p_i}^s$  defined by
    \begin{align}
        \label{def:E:s}
        \begin{split}
          &  V^2, V^{1,l,N_0}\in \widetilde{\mathcal{C}}_b(\mathbb{R}^+,\BH{s}{2}{r})\cap \widetilde{L}^1(\mathbb{R}^+,\BH{s+2}{2}{r}); \;\;  V^{1,h,N_0}\in \widetilde{\mathcal{C}}_b(\mathbb{R}^+,\BH{s+1}{2}{r})\cap \widetilde{L}^1(\mathbb{R}^+,\BH{s+1}{2}{r})\\
            &  V^{1,H,N_0}\in \mathcal{C}_b(\mathbb{R}^+,\BH{\frac{d}{p_i}+1}{p_i}{1})\cap L^1(\mathbb{R}^+,\BH{\frac{d}{p_i}+1}{p_i}{1}) ;\;\; \pt V\in L^1(\mathbb{R}^+,\BH{\frac{d}{p_i}}{p_i}{1})  \\
            & V^{2,H,N_0}, W^{H,N_0}\in \mathcal{C}_b(\mathbb{R}^+,\BH{\frac{d}{p_i}}{p_i}{1})\cap L^1(\mathbb{R}^+,\BH{\frac{d}{p_i}+2}{p_i}{1}), 
        \end{split}
    \end{align}
with $W$ defined in \eqref{W:def:2}.
\medbreak
Moreover, we have the following bound for all $t\ge 0$ and $i=1,2$:
\begin{align*}
   X_{p_i}(t)\lesssim  \NBH{s}{2}{r}{(V^2_0,V^{1,l,N_0}_0)}+  \NBH{s+1}{2}{r}{V^{1,h,N_0}_0} +\NBH{\frac{d}{p_i}}{p_i}{1}{V^{2,h,N_0}_0}+ \NBH{\frac{d}{p_i}+1}{p_i}{1}{V^{1,h,N_0}_0}
\end{align*}
with 
\begin{multline*}
 X_{p_i}(t) \defn \tLpNBH{s}{2}{r}{(V^2,V^{1,l,N_0})}{\infty}+\tLpNBH{s+2}{2}{r}{(V^2,V^{1,l,N_0})}{1}+\tLpNBH{s+1}{2}{r}{V^{1,h,N_0}}{\infty}\\+\tLpNBH{s+1}{2}{r}{V^{1,h,N_0}}{1}
    + \Vert V^1\Vert_{L^\infty_t\cap L^1_t(\BH{\frac{d}{p_i}+1}{p_i}{1})}+ 
    \ttLpNBH{\frac{d}{p_i}}{p_i}{1}{(W,V^2)}{\infty}^{h,N_0} + \ttLpNBH{\frac{d}{p_i}+2}{p_i}{1}{(W,V^2)}{1}^{h,N_0}+\\
    \ttLpNBH{\frac{d}{p_i}}{p_i}{1}{\pt V}{1}.
\end{multline*}

In addition, uniqueness holds for $p_2<2d$.
    \end{thm}

\begin{rmq}
    \label{rmq:s}
    Let us notice that in our result, contrary with the results in \cite{DanCha10,Haspot11,ChenMia10}, we don't have any restriction on $p_i$ in the existence part. We can take $p_2$ and then  $p_1$ as large as we want. This is due to the fact
    that here we ask one more degree of regularity for the high frequencies part of $V$. 
    
    \end{rmq}
    \begin{rmq}
        If $r=\infty$, then in the above theorem we must replace the strong continuity in $\BH{s}{2}{r}$ by the weak continuity.
    \end{rmq}

\section{A priori estimates for Theorem \ref{thm:p}.}
\label{sec:priorEstThm1}

We assume that we are given a smooth (and decaying) solution $V$ of \eqref{eq:base:V} on $[0, T ] \times \Rd$ with $V_0$ as initial data, satisfying
\begin{align}
     \label{smaal:cond:V}
     \LpNBH{\frac{d}{2}-1}{2}{1}{V}{\infty}^{l}+ \LpNBH{\frac{d}{p_2}+1}{p_2}{1}{V^1}{\infty}^h+\LpNBH{\frac{d}{p_2}}{p_2}{1}{V^2}{\infty}^h\ll 1.
 \end{align}

 We shall use repeatedly that
for all $\sigma\in \mathbb{R}$, $q,r\in [1,\infty]$ and $\alpha\ge 0$
\begin{align}
 \label{est:lh:hf}
    \NBH{s}{q}{r}{z}^h\lesssim \NBH{s+\alpha}{q}{r}{z}^h \esp{and} \NBH{s-\alpha}{q}{r}{z}^{l}\lesssim \NBH{s}{q}{r}{z}^{l}.
\end{align}
Owing to the embedding $\BH{\frac{d}{p}}{p}{1}\hookrightarrow L^\infty$, Inequalities \eqref{est:lh:hf} and the smallness condition \eqref{smaal:cond:V}, we have also the following inequalities which will be used repeatedly~: 
 \begin{align}
    \label{small:data:inf}
\sup_{t\in [0,T]} \NBH{\frac{d}{p_2}}{p_2}{1}{V(t)}\ll 1 \esp{and} \sup_{t\in [0,T]} \normeinf{V(t)} \ll 1 .
\end{align}

From now on, $C >0$  designates a generic harmless constant, the value of which depends on the context and we denote by $(c_k)_{k\in \mathbb{Z}}$ various nonnegative sequences such that: $\sum_{k\in \mathbb{Z} } c_k \le 1$.
\smallbreak
The following control of the time derivative will be very useful in the sequel.

 \begin{lem}
	 \label{lem_esti_partial_t_V_2}
     Let hypotheses (\textbf{H1})-(\textbf{H7}) be in force, and \eqref{smaal:cond:V} be satisfied.
     \begin{enumerate}
         \item Let $q\in [2,+\infty] $ and $\kappa\in [\frac{d}{q}-1,\frac{d}{q}]$.  There exists a constant $C=C_{\NBH{\frac{d}{q}}{q}{1}{V}}>0$ 
	  such that 
\begin{equation} \label{pt:V:est}
        \NBH{\kappa}{q}{1}{ \pt V } \le C(\NBH{\frac{d}{q}}{q}{1}{\n V} +1)\NBH{\kappa}{q}{1}{\nabla V}+
      C\left(\NBH{\kappa+2}{q}{1}{V^2}+\NBH{\kappa+1}{q}{1}{V} \NBH{\frac{d}{q}+1}{q}{1}{V^2}\right)\cdotp
\end{equation}
    
    \item Let $2\le p_1\le p_2$. There exists a constant $C>0$ depending only on $p_1,p_2, d$ and the matrices involved in the system \eqref{systemE}, such that 

\begin{multline}
     \label{pt:V:est:p1}
    \NBH{\cdp{p_1}}{p_1}{1}{\pt V}\le C\NBH{\cdp{p_1}}{p_1}{1}{\n V}+ C\NBH{\cdp{p_1}}{p_1}{1}{\n V}\NBH{\cdp{p_2}}{p_2}{1}{\n V}+C\NBH{\cdp{p_1}}{p_1}{1}{\n^2 V^2}\\
   + C(1+\NBH{\cdp{p_1}}{p_1}{1}{V})\biggl(\NBH{\cdp{p_2}}{p_2}{1}{\n V}^2+\NBH{\cdp{p_2}}{p_2}{1}{\n V}+\NBH{\cdp{p_2}}{p_2}{1}{\n^2 V^2}\biggr)\cdotp
\end{multline}
 \end{enumerate}

	 \end{lem}
	 \begin{proof}
	    The proof relies on the product estimates of Proposition \ref{propo_produc_BH} and on the composition estimate of Proposition \ref{propo:comp}. We recall  from\eqref{eq:base:V} that
\begin{equation}
	\label{sysdif::modif}
	 \pt V = (S^0(U))^{-1}f-\sum_{\alpha}{(S^0(U))^{-1}S^\alpha (U) \pal U}  +\sum_{\alpha,\beta}{ (S^0(U))^{-1}\pal (Y(U)^{\alpha\beta} \pbe V)  }.
	\end{equation}
 So, as $-d/p <\kappa \le d/p $, using product estimates \eqref{product_propo2} combined with Proposition  \ref{propo:comp}  yields (keep in mind the structure of $f$ defined in (\textbf{H7})): 
 \begin{equation*}
  \NBH{\kappa}{q}{1}{  (S^0(U))^{-1}f }\lesssim \Bigl( 1+\NBH{\frac{d}{q}}{q}{1}{V}\Bigr) \NBH{\frac{d}{q}}{q}{1}{\n V} \NBH{\kappa}{q}{1}{\nabla V},\end{equation*}
  and thus 
  \begin{multline*}
     \NBH{\kappa}{q}{1}{ \pt V } \lesssim \Bigl( 1+\NBH{\frac{d}{q}}{q}{1}{V}\Bigr) \NBH{\frac{d}{q}}{q}{1}{\n V} \NBH{\kappa}{q}{1}{\nabla V}+ \Bigl( 1+\NBH{\frac{d}{q}}{q}{1}{V}\Bigr) \NBH{\kappa}{q}{1}{\nabla V}\\
    +     \Bigl( 1+\NBH{\frac{d}{q}}{q}{1}{V}\Bigr) \sumab \NBH{\kappa}{q}{1}{\pal((Y(u)^{\alpha\beta} \pbe V)}.
 \end{multline*}
 Hence, using \eqref{product_propo3:chap2} for the last term, and remembering the form of each $ Y^{\alpha\beta} $ yields 
 \begin{align*}
     \NBH{\kappa}{q}{1}{ \pt V } \lesssim  \Bigl( 1+\NBH{\frac{d}{q}}{q}{1}{V}\Bigr) \NBH{\frac{d}{q}}{q}{1}{\n V} \NBH{\kappa}{q}{1}{\nabla V}+\Bigl( 1+\NBH{\frac{d}{q}}{q}{1}{V}\Bigr) \NBH{\kappa}{q}{1}{\nabla V}\\
     +
     ( 1+\NBH{\frac{d}{q}}{q}{1}{V})  \Bigl(( 1+\NBH{\frac{d}{q}}{q}{1}{V}) \NBH{\kappa+2}{q}{1}{V^2}+\NBH{\kappa+1}{q}{1}{V} \NBH{\frac{d}{q}+1}{q}{1}{V^2}\Bigr)\cdotp
 \end{align*}

 Regarding the proof of Inequality \eqref{pt:V:est:p1}, it suffices to adapt the previous proof by utilizing the smallness condition \eqref{smaal:cond:V} and the product estimate \eqref{product_propo:p1:p2} instead of \eqref{product_propo3:chap2}.
	 \end{proof}
     
\subsection{High frequencies analysis in \texorpdfstring{$L^{p_i}\:$}:}  We aim at proving the following result: 

\begin{propo}
    \label{propo:est:hf}
    Assume that $2\le p_1\le p_2$ and let $V$ be a smooth solution of \eqref{eq:base:V} under Assumptions (\textbf{H1})-(\textbf{H7}). Then,

     \begin{equation}
        \label{est:V:W:hf}
         \LpNBH{\frac{d}{p_i}}{p_i}{1}{(W,V^2)}{\infty}^h+ \LpNBH{\frac{d}{p_i}+2}{p_i}{1}{(W,V^2)}{1}^h+\Vert V^1\Vert_{(L^\infty\cap L^1)(\BH{\frac{d}{p_i}+1}{p_i}{1})}^h\lesssim  \mathbb{V}_{{p_i},0}+(2^{-N_{0}}+ \mathbb{V}_{p_2}+\mathbb{V}_{p_2}^2)\mathbb{V}_{p_i},
    \end{equation}
    provided $N_{0}$ is sufficiently  large.

\end{propo}

The proof of Proposition \ref{propo:est:hf} will be divided into three steps.
\paragraph{\textbf{Step 1: Estimate for $W$.}} We have the following statement.
\begin{lem}
    \label{est:M:lem}
It holds that
  \begin{multline}\label{2:est:p:W}
  \LpNBH{\frac{d}{p_i}}{p_i}{1}{W}{\infty}^h+ \LpNBH{\frac{d}{p_i}+2}{p_i}{1}{W}{1}^h \lesssim  \NBH{\frac{d}{p_i}}{p_1}{1}{W_0}^h+\int^T_0\biggl( 2^{-N_{0}}\NBH{\frac{d}{p_i}}{p_i}{1}{\pt V}^h+\NBH{\frac{d}{p_2}+1}{p_2}{1}{V}^2\NBH{\frac{d}{p_i}}{p_i}{1}{ V}\\
+\NBH{\frac{d}{p_2}}{p_2}{1}{V}\NBH{\frac{d}{p_i}}{p_i}{1}{\pt V}+\NBH{\frac{d}{p_i}}{p_i}{1}{V}\NBH{\frac{d}{p_2}}{p_2}{1}{\pt V}+ \NBH{\frac{d}{p_i},\frac{d}{p_i}+1}{p_i}{1}{V}\NBH{\frac{d}{p_2},\frac{d}{p_2}+1}{p_2}{1}{ V}
    + \NBH{\frac{d}{p_2}}{p_2}{1}{V}\NBH{\frac{d}{p_i}+2}{p_i}{1}{ V^2}\biggr)
\end{multline}   
where hereafter $ \BH{s,s^{\prime}}{p}{r}\defn \BH{s}{p}{r}\cap \BH{s^{\prime}}{p}{r}$ and $\lVert \cdot \rVert_{\BH{s,s^{\prime}}{p}{r}}$ designates the  norm associated to $\BH{s,s^{\prime}}{p}{r}$.
 \end{lem}
\begin{proof}
Applying  Lemma \ref{lem:maxi:regula}  to the parabolic equation \eqref{eq:W} 
yields: for all $t\in [0,T]$ and $j\in \mathbb{Z},$
\begin{align*}
    \normepp{W_j(t)}\le C\biggl(e^{-c_02^{2j }t} \normepp{W_{0,j}}+\int^t_0 e^{-c_02^{2 j}(t-\tau)} \normepp{\Dj g(\tau)}d\tau \biggr),
\end{align*}
where $c_0$ and $C$ are positive constant and we have used the notation  $W_j=\Dj W$. Integrating the previous inequality with respect to time on $ [0,T] $ implies
\begin{align}
\label{est:p:W:1}
    \normepp{W_j(t)}+ 2^{2j}c\int^t_0  \normepp{W_j}\le C\biggl( \normepp{W_{0,j}}+\int^t_0 \normepp{\Dj g(\tau)}d\tau \biggr),
\end{align}
where $c$ is a positive constant.
\smallbreak
In order to bound the right-hand side, we decompose $g$ into  $g=g^t+g^{V^2}+f^f+g^V+\Bar{g}^V$
as in \eqref{def:g}.
The first term can be bounded by means of \eqref{product_propo1} and of Proposition \ref{propo:comp}. We get
\begin{align*}
    \NBH{\frac{d}{p_i}}{p_i}{1}{g^t}\lesssim  \NBH{\frac{d}{p_i}}{p_i}{1}{V}\normeinf{\pt V}+ \normeinf{V}\NBH{\frac{d}{p_i}}{p_i}{1}{\pt V}.
\end{align*}
Using the embedding critical embedding
\begin{equation}\label{embed:crit}
\BH{\frac{d}{p_2}}{p_2}{1}\hookrightarrow L^\infty,\end{equation} we arrive at
\begin{equation}    \label{est:gt}
    \NBH{\frac{d}{p_i}}{p_i}{1}{g^t}\lesssim  \NBH{\frac{d}{p_i}}{p_i}{1}{V}\NBH{\frac{d}{p_2}}{p_2}{1}{\pt V}+  \NBH{\frac{d}{p_2}}{p_2}{1}{V}\NBH{\frac{d}{p_i}}{p_i}{1}{\pt V}.
    \end{equation}
To handle $ g^{V^2}$ we use \eqref{product_propo1} combined with Proposition \ref{propo:comp} and Bernstein inequality \eqref{bern:ine}. It holds
\begin{align*}
    \NBH{\frac{d}{p_i}}{p_i}{1}{g^{V^2}}&\lesssim \NBH{\frac{d}{p_i}+1}{p_i}{1}{\sumab (\overline{Z}^{\alpha\beta}-Z^{\alpha\beta}(U))\pbe V^2 }\\
    &\lesssim \NBH{\frac{d}{p_i}+1}{p_i}{1}{V}\normeinf{\n V^2}+ \normeinf{V}\NBH{\frac{d}{p_i}+1}{p_i}{1}{\n V^2},
\end{align*}
which becomes, using again \eqref{embed:crit}, 
\begin{equation}     \label{est:gV2}
     \NBH{\frac{d}{p_i}}{p_i}{1}{g^{V^2}}\lesssim 
     \NBH{\frac{d}{p_i}+1}{p_i}{1}{V}\NBH{\frac{d}{p_2}+1}{p_2}{1}{ V^2}+ \NBH{\frac{d}{p_2}}{p_2}{1}{V}\NBH{\frac{d}{p_i}+2}{p_i}{1}{ V^2}.
    \end{equation}
Similarly taking into account the form of $f$ and \eqref{small:data:inf}, we have 
\begin{align}
    \label{est:gf}
    \begin{split}
     \NBH{\frac{d}{p_i}}{p_i}{1}{g^{f}}&\lesssim  \NBH{\frac{d}{p_i}}{p_i}{1}{\Upsilon(U)\n U\otimes \n U}\\
     &\lesssim  (1+\normeinf{V})\normeinf{\n V}\NBH{\frac{d}{p_i}}{p_i}{1}{\n V}+ \normeinf{\n V}^2\NBH{\frac{d}{p_i}}{p_i}{1}{ V} )\\
     &\lesssim \NBH{\frac{d}{p_2}+1}{p_2}{1}{ V}\NBH{\frac{d}{p_i}+1}{p_i}{1}{V} +\NBH{\frac{d}{p_2}+1}{p_2}{1}{V}^2\NBH{\frac{d}{p_i}}{p_i}{1}{ V} .
      \end{split}
\end{align}

The term $g^V$ can be bounded directly using \eqref{product_propo1} and \eqref{embed:crit}. We have 
\begin{align}
    \label{est:gV}\nonumber
     \NBH{\frac{d}{p_i}}{p_i}{1}{g^{V}} &\lesssim \normeinf{V}\NBH{\frac{d}{p_i}}{p_i}{1}{\n V} + \normeinf{\n V}\NBH{\frac{d}{p_i}}{p_i}{1}{ V}\\
      &\lesssim \NBH{\frac{d}{p_2}}{p_2}{1}{V}\NBH{\frac{d}{p_i}+1}{p_i}{1}{ V} +  \NBH{\frac{d}{p_2}+1}{p_2}{1}{V}\NBH{\frac{d}{p_i}}{p_i}{1}{ V}.
      \end{align}
As the operators  $\Ss_{12}(D)$ , $\Ss_{21}(D)$ and  $\mathcal{Z}(D)$ are homogeneous of degree $0$, applying Bernstein inequality \eqref{bern:ine:M} allows us to bound the last term $\Bar{g}^{V}$. It holds
\begin{align}
    \label{est:bar{g}V}
    \begin{split}
  \normespi{\Dj(\Bar{g}^{V})}\lesssim 2^{-j}  \normespi{\pt V_j}.
      \end{split}
\end{align}
Inserting \eqref{est:gV}, \eqref{est:bar{g}V},  \eqref{est:gt} and \eqref{est:gV2} into \eqref{est:p:W:1}, then multiplying the obtained inequality by $2^{j\frac{d}{p_i}}$
and summing up on $j\ge N_0+2$ completes the proof of \eqref{2:est:p:W}. \end{proof}

\paragraph{\textbf{Step 2: Estimates for $V^1$}} 
\begin{lem}
\label{est:V1:hf:lem}
Under assumptions (\textbf{H1})-(\textbf{H7}), It holds:
\begin{multline}\label{est:z:1}
\LpNBH{\frac{d}{p_i}+1}{p_i}{1}{V^1}{\infty}^{h}+  \LpNBH{\frac{d}{p_i}+1}{p_i}{1}{V^1}{1}^{h}\lesssim \NBH{\frac{d}{p_i}+1}{p_i}{1}{V^1_0}^{h}+\int^T_0 \NBH{\frac{d}{p_2}+1}{p_2}{1}{V}\NBH{\frac{d}{p_i}+1}{p_i}{1}{V}\\
     + \int^T_0 \left(\NBH{\frac{d}{p_i}+2}{p_i}{1}{{W}}^{h}+\NBH{\frac{d}{p_i}+1}{p_i}{1}{{V^2}}^{h}+\NBH{\frac{d}{p_2}}{p_2}{1}{V}\NBH{\frac{d}{p_i}+2}{p_i}{1}{ V^2}  \right)\cdotp
     \end{multline}  
\end{lem}

\begin{proof}
Throughout this proof, we will use repeatedly the fact that the matrix $B$ is a constant. 
Applying $\Dj$ to the equation \eqref{eq:V1=f(W)}  and denoting $ z_j= \Dj z$ leads to 
\begin{align}
    \label{eq:zj}
     \pt z_j+\suma H^\alpha\pal  z_j+ B(D) z_j=R_j+\Dj h
\end{align}
 with $h$ given by \eqref{def:ht:V:V2} and
 \begin{align}
     \label{def:Rj}
    R_j\defn  -\suma \left[\Dj,H^\alpha(U)- H^\alpha(\overline{U})\right]\pal z.
 \end{align}
  Let us denote $(z^i_j)_{i}$, $(B(D)z_j)_i$ (with $i=1,\cdots,n_1$) respectively the components of the vector $z_j$ and $ B(D)z_j $ and $(H^\alpha_i)_i$ the diagonal components of the matrix $ H^\alpha$. \\
Multiplying \eqref{eq:zj} by the vector $(\overline{z}^i_j|z^i_j|^{p_i-2})_i$, using Leibniz formula on the convective term, taking the real part and remembering that the matrices $H^\alpha$ are diagonal yields for $i=1,2$
\begin{multline*}
     \frac{1}{p_i}\frac{d}{dt}|z_j|^{p_i}+\frac{1}{p_i}\suma \sum_{k=1}^{n_1}\pal\left(H^\alpha_k|z_j^k|^{p_i}_{}\right) -\frac{1}{p_i}\sum_{k=1}^{n_1}\suma\pal (H^\alpha_i(U))|z^k_j|^{p_i}+\sum_{k=1}^{n_1} (B(D)z_j)_k z_j^k |z_j^k|^{p_i-2}\\
     \le |z_j|^{p_i-1}(|R_j|+|\Dj h|).
 \end{multline*}
 Then, integrating on $\Rd$ and remembering that $B$ is a constant diagonal matrix with positive entries yields for some  positive constants $c_0$ and $C,$ 
 \begin{multline}     \label{est:V1:ini}
         \frac{1}{p_i}\frac{d}{dt}\normespi{z_j}^{p_i}+ c_0\normespi{z_j}^{p_i} \le C\normeinf{\n H(U)}\normespi{z_j}^{p_i}+\normespi{R_j} \normespi{z_j}^{p_i-1}\\
     +C \normespi{\Dj h} \normespi{z_j}^{p_i-1},
      \end{multline}

Combining the results of Proposition \ref{propo:comp} with the embedding \eqref{embed:crit} yields
\begin{align}
    \label{est: H}
    \normeinf{\n H(U)}\lesssim \NBH{\frac{d}{p_2}+1}{p_2}{1}{V}.
\end{align}
 Applying directly \eqref{comm:est:p} (with $\sigma=\frac{d}{p_i}+1$, $p=p_i$, $q=\widetilde{p}=p_2$)  and using the embeddings $ \BH{\frac{d}{p_2}}{p_2}{1}\hookrightarrow L^\infty  \cap \BH{\frac{d}{p_2}}{p_2}{\infty} $ and \eqref{embed:crit} leads to
 \begin{align}     \label{est:Rj}
     \normespi{R_j} &\le Cc_j 2^{-j(\frac{d}{p_i}+1)}\left( \lVert \n V \rVert_{\BH{\cdp{p_2}}{p_2}{1}} \NBH{\frac{d}{p_i}+1}{p_i}{1}{z}+ \normeinf{\n z}\NBH{\frac{d}{p_i}+1}{p_i}{1}{H(U)-H(\overline{U})}\right)\nonumber\\
     &\le Cc_j 2^{-j(\frac{d}{p_i}+1)}\NBH{\frac{d}{p_i}+1}{p_i}{1}{V}\NBH{\frac{d}{p_2}+1}{p_2}{1}{V}.
 \end{align}
Owing to the product property \eqref{product_propo1} and the embedding  $ \BH{\frac{d}{p_2}}{p_2}{1}\hookrightarrow L^\infty$ we get the following bound on $h^{V^2}_1$: for all $j\ge N_{0}+2$
\begin{align}
    \label{est: h^v^2}
    \NBH{\frac{d}{p_i}+1}{p_i}{1}{h^{V^2}_1}^h\lesssim \NBH{\frac{d}{p_2}}{p_2}{1}{V}\NBH{\frac{d}{p_i}+1}{p_i}{1}{\n V^2}+ \NBH{\frac{d}{p_i}+1}{p_i}{1}{V}\NBH{\frac{d}{p_2}}{p_i}{1}{\n V^2}.
\end{align}
 Since the homogeneous operators $P(\overline{S}^0_{11})^{-1}\overline{S}^\alpha_{12}\pal$ and $ P (\overline{S}^0_{11})^{-1}  \Ss_{12}(D)\mathcal{Z}^{-1}(D)\Ss_{22}(D) $ are respectively of degree 1 and $0$,  we have from Bernstein inequality
 \begin{align}
   \label{est: h^W}
    \normespi{\Dj h^{W}}\lesssim   \normespi{\Dj \n W}\esp{and}  \normespi{\Dj h^{V^2}_0}\lesssim \normespi{\Dj V^2} .
 \end{align}

Hence, plugging \eqref{est: H}, \eqref{est:Rj}, \eqref{est: h^v^2} and \eqref{est: h^W}  into \eqref{est:V1:ini} yields for all $j\in\mathbb Z,$
\begin{multline}
\label{est:V1:pi:0}
    \frac{1}{p_i}\frac{d}{dt}\normespi{z_j}^{p_i}+ \normespi{z_j}^{p_i} \lesssim \normespi{z_j}^{p_i}\NBH{\frac{d}{p_2}+1}{p_2}{1}{V}+c_j2^{-j(\frac{d}{p_i}+1)}\left(\NBH{\frac{d}{p_2}+1}{p_2}{1}{V}\NBH{\frac{d}{p_i}+1}{p_i}{1}{V}\right.\\
       \left.   + \NBH{\frac{d}{p_2}}{p_2}{1}{V}\NBH{\frac{d}{p_i}+2}{p_i}{1}{ V^2}+ \NBH{\frac{d}{p_i}+1}{p_i}{1}{V^2}^{h}+\NBH{\frac{d}{p_i}+2}{p_i}{1}{W}^{h} \right) \normespi{z_j}^{p_i-1}.
\end{multline}
In the sequel we will choose $N$ large enough  so that the last term in the previous inequality  be absorbed by $ \mathbb{V}_{p_i} $.\\ 
Finally, using lemma \ref{lem_der_int} (with $p=p_i$, $X=\normespi{z_j}^{p_i}$), then multiplying the obtained  inequality by $2^{j\frac{d}{p_i}+1}$, summing up $j\ge N_{0}+2$, gives us
\begin{multline*}
   \LpNBH{\frac{d}{p_i}+1}{p_i}{1}{z}{\infty}^{h}+  \LpNBH{\frac{d}{p_i}+1}{p_i}{1}{z}{1}^{h}\lesssim \NBH{\frac{d}{p_i}+1}{p_i}{1}{z(0)}^{h}+\int^T_0 \NBH{\frac{d}{p_2}+1}{p_2}{1}{V}\NBH{\frac{d}{p_i}+1}{p_i}{1}{z}
   +\int^T_0 \NBH{\frac{d}{p_2}+1}{p_2}{1}{V}\NBH{\frac{d}{p_i}+1}{p_i}{1}{V}\\
     + \int^T_0 \left(\NBH{\frac{d}{p_i}+2}{p_i}{1}{{W}}^{h}+\NBH{\frac{d}{p_i}+1}{p_i}{1}{{V^2}}^{h}+\NBH{\frac{d}{p_2}}{p_2}{1}{V}\NBH{\frac{d}{p_i}+2}{p_i}{1}{ V^2}  \right)
     \cdotp\end{multline*}
Remembering that $V^1=P^{-1}z$ with $P$ a constant matrix (independent of $U$) completes the proof. 

\end{proof}

\paragraph{\textbf{Step 3: Closing the estimates}} The goal of this step is to conclude the proof of estimate \eqref{est:V:W:hf}.

Putting together Inequalities \eqref{2:est:p:W} and \eqref{est:z:1} (note that the 
term with $W$ on the right-hand side of  \eqref{est:z:1} may be absorbed by the left-hand side of \eqref{2:est:p:W})
and using the fact that 
$$ \NBH{\frac{d}{p_i}+1}{p_i}{1}{V^2}^{h}\lesssim 2^{-N_{0}} \NBH{\frac{d}{p_i}+2}{p_i}{1}{V^2}^{h},$$ we get 
\begin{multline*}
\LpNBH{\frac{d}{p_i}+1}{p_i}{1}{V^1}{\infty}^{h}+  \LpNBH{\frac{d}{p_i}+1}{p_i}{1}{V^1}{1}^{h} +  \LpNBH{\frac{d}{p_i}}{p_i}{1}{W}{\infty}^{h}+  \LpNBH{\frac{d}{p_i}+2}{p_i}{1}{W}{1}^{h}\lesssim \mathbb{V}_0 +\int^T_0 \NBH{\frac{d}{p_2}+1}{p_2}{1}{V}^2\NBH{\frac{d}{p_i}}{p_i}{1}{ V}\\
+  \int^T_0(\NBH{\frac{d}{p_2}}{p_2}{1}{\pt V}\NBH{\frac{d}{p_i}}{p_i}{1}{V}+ (2^{-N_{0}}+ \NBH{\frac{d}{p_2}}{p_2}{1}{V} )\NBH{\frac{d}{p_i}}{p_i}{1}{\pt V})\\
+ \int^T_0\NBH{\frac{d}{p_2},\frac{d}{p_2}+1}{p_2}{1}{V}\NBH{\frac{d}{p_i},\frac{d}{p_i}+1}{p_i}{1}{V}+ \int^T_0\Bigl( \NBH{\frac{d}{p_i}+2}{p_i}{1}{V^2}\NBH{\frac{d}{p_2}}{p_2}{1}{V}+ 2^{-N_{0}} \NBH{\frac{d}{p_i}+2}{p_i}{1}{V^2}^{h}\Bigr)\cdotp
\end{multline*}
Keeping \eqref{W:def:2} in mind, 
we have for all $s\in \mathbb{R}$ and $1\le p,r\le \infty$
\begin{align}
    \label{est:w=f(v1,V2)}
    \NBH{s}{p}{r}{W}^{l}\lesssim  \NBH{s-1}{p}{r}{(V^1,V^2)}^{l}\esp{and}  \NBH{s}{p}{r}{W}^{h}\lesssim  \NBH{s-1}{p}{r}{V^1}^{h}+\NBH{s}{p}{r}{V^2}^{h}
\end{align}
and
\begin{align*}
  \NBH{s}{p}{r}{V^2}^{h}\lesssim  \NBH{s-1}{p}{r}{V^1}^{h}+\NBH{s}{p}{r}{W}^{h}+ 2^{-N_{0}} \NBH{s}{p}{r}{V^2}^{h}
\end{align*}
from which  choosing $N_0$ large enough, we get
\begin{align}
\label{est:V2=f(v1,W)}
  \NBH{s}{p}{r}{V^2}^{h}\lesssim  \NBH{s-1}{p}{r}{V^1}^{h}+\NBH{s}{p}{r}{W}^{h}.
\end{align}
Taking \eqref{est:V2=f(v1,W)} into account, we finally get
\begin{multline}\label{est:W:V^1}
\LpNBH{\frac{d}{p_i}+1}{p_i}{1}{V^1}{\infty}^{h}+  \LpNBH{\frac{d}{p_i}+1}{p_i}{1}{V^1}{1}^{h} +  \LpNBH{\frac{d}{p_i}}{p_i}{1}{(W,V^2)}{\infty}^{h}+  \LpNBH{\frac{d}{p_i}+2}{p_i}{1}{(W,V^2)}{1}^{h}\lesssim \mathbb{V}_{{p_i},0}\\ +\int^T_0 \NBH{\frac{d}{p_2}+1}{p_2}{1}{V}^2\NBH{\frac{d}{p_i}}{p_i}{1}{ V}
+  \int^T_0(\NBH{\frac{d}{p_2}}{p_2}{1}{\pt V}\NBH{\frac{d}{p_i}}{p_i}{1}{V}+ (2^{-N_{0}}+ \NBH{\frac{d}{p_2}}{p_2}{1}{V} )\NBH{\frac{d}{p_i}}{p_i}{1}{\pt V})\\
+ \int^T_0\NBH{\frac{d}{p_2},\frac{d}{p_2}+1}{p_2}{1}{V}\NBH{\frac{d}{p_i},\frac{d}{p_i}+1}{p_i}{1}{V}+ \int^T_0\NBH{\frac{d}{p_i}+2}{p_i}{1}{V^2}\NBH{\frac{d}{p_2}}{p_2}{1}{V}.
\end{multline}
Let us now estimate each term of the right-hand side of the previous inequality in terms of $ \mathbb{V}_{p_i}$ and $\mathbb{V}_{{p_i},0}$. The embeddings \eqref{est:lf:u:hf} and 
\begin{align}
    \label{emmb:p:2}
    \BH{s+\frac{d}{p}}{2}{1}\hookrightarrow \BH{s}{p}{1} \esp{ for all }s\in \mathbb{R}, \quad p\geq2
\end{align}
will come into play. Since we have \eqref{smaal:cond:V}, using Estimate \eqref{pt:V:est} with $\kappa=\frac{d}{p_2}$ and $q=p_2$  yields
\begin{equation}
    \label{est:ptV:w:v}
    \LpNBH{\frac{d}{p_2}}{p_2}{1}{\pt V}{1}\lesssim \LpNBH{\frac{d}{p_2}+1}{p_2}{1}{ V}{1} + \LpNBH{\frac{d}{p_2}+2}{p_2}{1}{ V^2}{1} +\int^T_0 \NBH{\frac{d}{p_2}+1}{p_2}{1}{V}^2.    
\end{equation}
We claim that, each term on the right-hand side of \eqref{est:ptV:w:v} is less than $C (\mathbb{V}_{p_i}+\mathbb{V}_{p_2}+ \mathbb{V}_{p_i}\mathbb{V}_{p_2})$. In fact, from \eqref{est:lh:hf} it holds
\begin{align}
\label{est:V:Bs:p:Vi}\nonumber
        \LpNBH{\frac{d}{p_i}+1}{p_i}{1}{ V}{1} + \LpNBH{\frac{d}{p_i}+2}{p_i}{1}{ V^2}{1} &\lesssim \LpNBH{\frac{d}{2}+1}{2}{1}{ V}{1}^{l} +\LpNBH{\frac{d}{p_i}+1}{p_i}{1}{ V^1}{1}^h+\LpNBH{\frac{d}{p_i}+2}{p_i}{1}{ V^2}{1}^h\\\
        & \lesssim  \mathbb{V}_{p_i}.
\end{align}
Therefore,  using  interpolation inequality in Besov spaces and H\"older inequality, we obtain that
\begin{align*}
        \int^T_0 \NBH{\frac{d}{p_2}+1}{p_2}{1}{V}^2 &\lesssim
        \LpNBH{\frac{d}{p_2}+1}{p_2}{1}{V^1}{1}  \LpNBH{\frac{d}{p_2}+1}{p_2}{1}{V^1}{\infty}+  \LpNBH{\frac{d}{p_2}}{p_2}{1}{V^2}{\infty}\LpNBH{\frac{d}{p_2}+2}{p_2}{1}{V^2}{1}\lesssim 
         \mathbb{V}_{p_2}^2.
\end{align*}
 In the end, we get
 \begin{align}
     \label{est:pt:V:W:v:1}
      \LpNBH{\frac{d}{p_2}}{p_2}{1}{\pt V}{1} \le C (\mathbb{V}_{p_2}+ \mathbb{V}_{p_2}^2).
 \end{align}

 Next, from \eqref{est:lh:hf} and embedding \eqref{emmb:p:2}, we have
\begin{align*}
    \begin{split}
\int^T_0 \NBH{\frac{d}{p_i},\frac{d}{p_i}+1}{p_i}{1}{V}\NBH{\frac{d}{p_2},\frac{d}{p_2}+1}{p_i}{1}{V} &\lesssim \LpNBH{\frac{d}{2}}{2}{1}{V}{2}^{l}\LpNBH{\frac{d}{2}}{2}{1}{V}{2}^{l}+\LpNBH{\frac{d}{p_2}+1}{2}{1}{V}{1}^h\LpNBH{\frac{d}{2}-1}{2}{1}{V}{\infty}^{l}\\
&+  \LpNBH{\frac{d}{p_i}+1}{p_i}{1}{V}{1}^h\LpNBH{\frac{d}{p_2}+1}{p_2}{1}{V^1}{\infty}^h\\
&+ \LpNBH{\frac{d}{p_i}+1}{p_i}{1}{V^1}{\infty}^h\LpNBH{\frac{d}{p_2}+2}{p_2}{1}{V^2}{1}^h+ \LpNBH{\frac{d}{p_i}+1}{p_i}{1}{V^2}{2}^h\LpNBH{\frac{d}{p_2}+1}{p_2}{1}{V^2}{2}^h.
    \end{split}
\end{align*}
Then, interpolation inequality in Besov space allow us to get
\begin{align}
    \label{pr:int:V:V:}
    \int^T_0 \NBH{\frac{d}{p_i},\frac{d}{p_i}+1}{p_i}{1}{V}\NBH{\frac{d}{p_2},\frac{d}{p_2}+1}{p_i}{1}{V} &\lesssim \mathbb{V}_{p_i} \mathbb{V}_{p_2}.
\end{align}
 Similarly one can prove that
    \begin{align*}
   \LpNBH{\frac{d}{p_2}}{p_2}{1}{V}{\infty}\LpNBH{\frac{d}{p_i}+2}{p_i}{1}{V^2}{1} &\lesssim \mathbb{V}_{p_i} \mathbb{V}_{p_2},\\
   \int^T_0 \NBH{\frac{d}{p_2}+1}{p_2}{1}{V}^2\NBH{\frac{d}{p_i}}{p_i}{1}{ V}&\lesssim \LpNBH{\cdp{p_2}+1}{p_2}{1}{V}{2}^2\LpNBH{\cdp{p_i}}{p_i}{1}{V}{\infty}\lesssim \mathbb{V}_{p_2}^2 \mathbb{V}_{p_i} .
  \end{align*}
Also, taking advantage of   \eqref{est:lh:hf},  \eqref{pt:V:est:p1} and \eqref{emmb:p:2},  we have
  \begin{align}
     \label{est:pt:V:W:v:p:1}
      \LpNBH{\frac{d}{p_1}}{p_1}{1}{\pt V}{1} \le C \Bigl(\mathbb{V}_{p_1}+ \mathbb{V}_{p_1}\mathbb{V}_{p_2} 
      \Bigr)\cdotp
 \end{align}

Plugging all these information into \eqref{est:W:V^1} and using \eqref{smaal:cond:V} once more yields 
\begin{multline}\label{est:W:V^1:1}
   \LpNBH{\frac{d}{p_i}+1}{p_i}{1}{V^1}{\infty}^h+  \LpNBH{\frac{d}{p_i}+1}{p_i}{1}{V^1}{1}^h +  \LpNBH{\frac{d}{p_i}}{p_i}{1}{(W,V^2)}{\infty}^h+  \LpNBH{\frac{d}{p_i}+2}{p_i}{1}{(W,V^2)}{1}^h\\\lesssim \mathbb{V}_{{p_i},0} 
   +(2^{-N_{0}}+ \mathbb{V}_{p_2}+\mathbb{V}_{p_2}^2)\mathbb{V}_{p_i}+2^{-N_{0}}\LpNBH{\frac{d}{p_i}+2}{p_i}{1}{V^2}{1}^h.
\end{multline}
The last term may be absorbed by the left-hand side if we choose $N$ large enough, 
which  completes the proof of \eqref{est:V:W:hf}. 


\subsection{Low frequencies analysis.}
\label{sub:sec:low}
In this part we aim at proving some priori estimates for the low frequencies of the solution in a $L^2$ framework. In fact, as mentioned above, we are
stuck to the $L^2$ framework for the low frequencies of the solution, since looking at the linear system \eqref{eq:lindiff},  it turns out that the first order terms predominate, so that the system \eqref{eq:base:V} has to be treated by means of hyperbolic energy methods. 
 For instance, consider the following barotropic Navier-Stokes equations:
\begin{align}
    \label{NS:chap44}
    \begin{cases}
        &\pt \rho+ {\rm div} (\rho v)=0\\
        &\rho\pt v+ \rho v\cdot \n v -\Delta v +\nabla \ro=0,
    \end{cases}
\end{align}
with $\rho$, $v$ are respectively the density and the velocity of the fluid. The linear system around the state $(1,0)$ of \eqref{NS:chap44} reads
\begin{align}
    \label{LNS}
    \pt U +\suma A^\alpha \pal U=\Delta \left(
    \begin{pmatrix}
       & 0\\
        & v
    \end{pmatrix}
    \right)
\end{align}
with $U=
 \begin{pmatrix}
       & a\defn\rho-1\\
        & v
    \end{pmatrix}
$
and $A^\alpha\in \mathcal{M}_{d+1}(\Rd)$ are symmetric  with   components $ (A^\alpha)_{1,\alpha+1}=(A^\alpha)_{\alpha+1,1}=1$ and $0$ in the other cases. It is clear that the  matrices $A^\alpha$ ($\alpha=1,\cdots, d$) don't commute between themselves. Consequently in light of \cite{Bren66}, we can't expect to get estimates in $L^p$ for $p\neq 2$, since the system \eqref{LNS} can be considered as symmetric hyperbolic: note that in low frequencies,  $a$ and the potential part of the velocity $v$ tend 
to satisfy a wave equation. 

At the nonlinear level, the fact that the high frequencies of the solution belong to some space of type $\BH{s}{p}{1} $ and not to $\BH{s^{\prime}}{2}{1}$ constitutes a supplementary difficulty: 
the low frequencies of the nonlinear terms contain some amount of high frequencies.
To bypass this difficulty, we will have to restrict the range of $p_1 $ and $p_2$ and resort to more
elaborate product laws (see Lemma \ref{lem:produ:lf:hf}).
The goal of this part is to prove the following proposition. 
\begin{propo}
\label{propo:est:lf:V}
    Let $d\ge 2$ and $p_2$ be such that 
    \begin{align}
        \label{cnd:p2:lf}
         p_2<2d \esp{and} 2\le p_2\le \min\left(4,\frac{2d}{d-2}\right)\cdotp
    \end{align}
Consider a smooth (decreasing) solution of \eqref{eq:base:V} under assumptions
 (\textbf{H1})-(\textbf{H7}). Assume that the solution $V$ verifies the smallness condition \eqref{smaal:cond:V}. Then, the following inequality holds true:
\begin{equation}
    \label{est:lf:V}
     \LpNBH{\frac{d}{2}-1}{2}{1}{V}{\infty}^{l}+ \LpNBH{\frac{d}{2}+1}{2}{1}{V}{1}^{l}
    \lesssim \NBH{\frac{d}{2}-1}{2}{1}{V_0}^{l}+ \mathbb{V}_{p_2}^2(1+\mathbb{V}_{p_2}).
\end{equation}
\end{propo}

\begin{proof}
   We need to differentiate the following functional that contains enough information to get inequality \eqref{est:lf:V}:
   \begin{equation}
    \label{def:Ij}
   \forall j\le N_0+1 \quad \mathcal{L}_j\overset{def}{=} \frac12\intd \overline{S}^0 V_j\cdot V_j+ 2^j \mathcal{I}_j\simeq  \normed{V_j}^2 ,
\end{equation}
    with 
    \begin{equation}
    \label{df:corect}
    \mathcal{I}_j \overset{def}{=}  \sumk \intd\ee_k {\rm  Re} (M_\omega N_\omega^{k-1} \widehat{V_j}\cdot M_\omega N_\omega^{k} \widehat{V_j} )
\end{equation}
and where $\ee_1,\cdots,\ee_{n-1}>0$ are chosen suitably small enough. The matrices $M_\omega$, $N_\omega$ have been defined in \eqref{A:def}, \eqref{B_ome:def} respectively.
\medbreak
We note that
\begin{equation}
\label{dt:lyp}
  \frac12  \frac{d}{dt} \mathcal{L}_j= \frac{d}{dt} \intd \overline{S}^0\pt V_j\cdot V_j+2^j\frac{d}{dt} \mathcal{I}_j.
\end{equation}

\paragraph{\textbf{First Step: Energy estimates for the first term of \eqref{dt:lyp}.}} 
We localize System \eqref{eq:lindiff} by means of the operator $\Dj$ as follows
\begin{align}
\label{eq:lh:V}
     \overline{S}^0\pt V_j+\suma \overline{S}^\alpha\pal V_j-\sum_{\alpha,\beta=1}^d\overline{Y}^{\alpha\beta} \partial_{\alpha\beta}V_j = \Dj F\esp{with} V_j \defn \Dj V. 
 \end{align}
For $j\le N_{0}+1 $, taking the scalar product in $L^2(\Rd;\mathbb{R}^n)$ with $V_j$, then performing an integration by parts in the second term and using the symmetry of the matrices $S^\alpha(U)$ yields,
\begin{align}
 \label{1':es:hf:}
 \begin{split}
    \frac12 \frac{d}{dt} \intd  \overline{S}^0 V_j\cdot  V_j +\sumab\intd \overline{Z}^{\alpha\beta} \partial_{\alpha\beta}V_j^2 \cdot V_j^2  &\le \normed{V_j}\normed{\Dj F}.
      \end{split}
 \end{align}

Next, thanks to \eqref{small:data:inf},  \eqref{strong_elli} and remarking that the smooth function $V_j$ is spectrally supported in dyadic
annuli of size $2^j$ leads to:
\begin{align}
 \label{1:es:hf:}
 \begin{split}
    \frac12 \frac{d}{dt} \intd  \overline{S}^0 V_j\cdot  V_j +\overline{c_1} 2^{2j} \normed{V^2_j}^2 &\lesssim \normed{V_j}\normed{ \Dj F}
      \end{split}
 \end{align}
 where $c_1$ is the constant appearing in \eqref{strong_elli}.
\paragraph{\textbf{Second step: Differentiation of $ \mathcal{I}_j$.}} As we can see in \eqref{1:es:hf:}, we miss the dissipation of the component $V^1$. To recover the dissipation of all the components of $V$, we have to look at the time derivative of $\Ij$ defined in \eqref{df:corect}. Applying  Lemma \ref{lem:derI} and using the fact that the support of $\widehat{V}_j$ is contained in an annulus of magnitude $2^j$ (this implies in particular that the term $\ro$ appearing in Inequality \eqref{lya:lo:in} is equivalent to $2^j$, that is, $\ro \simeq 2^j$ ) we get 
\begin{equation}
\label{low:lya:eqj}
    \frac{d}{dt}\mathcal{I}_j +2^{j} \sumk \ee_k |M_\omega N_\omega^{k} \widehat{V}_j|^2 \le C\ee_0\max(2^{j},2^{3j})| M_\omega  \widehat{V}_j|^2+C \normed{\Dj F}\normed{V_j}.
\end{equation}

\paragraph{\textbf{Third step: Estimating $\Lj$}}
Multiplying inequality \eqref{low:lya:eqj} by $2^j$ and putting it together with \eqref{1:es:hf:} provides a positive constant $c>0$ such that for all $j\le N_{0}+1$
\begin{align*}
    \frac{d}{dt}\Lj+  2^{2j}\sum_{k=1}^{n-1} \ee_k\intd |M_\omega N_\omega^{k}\widehat{V}_j|^2 +  c2^{2j} \normed{V^2_j}^2\lesssim   
    \normed{V_j}\normed{\Dj F}+ \ee_0 2^{2j}\intd| M_\omega  \widehat{V}_j|^2.
\end{align*}
  Recall that the matrix $M_\omega$ has the diagonal form \eqref{B_ome:def}. Choosing $\ee_0$ small enough, the last term of the right-hand side term can be absorbed by the left-hand side. 
 For all $j\le N_{0}+1$, it therefore holds (up to a slight modification of $\ee_0$)
that 

 \begin{align}
    \label{01:est:ly:l}
    \frac{d}{dt}\Lj+  2^{2j}\sum_{k=0}^{n-1} \ee_k\intd |M_\omega N_\omega^{k}\widehat{V}_j|^2
    \lesssim  
    & \normed{V_j}\normed{\Dj F}.
\end{align}

\paragraph{\textbf{Fourth step: Nonlinear estimates}}
In order to bound the nonlinear term $F,$ we will use Lemma \ref{lem:produ:lf:hf}, that is, the substitute of Proposition \ref{propo_produc_BH} since the high frequencies and low frequencies of $V$ belong to Besov spaces with different Lebesgue indices. 
Here, $p_2<2d$ and $2\le p_2\le \min\left(4,\frac{2d}{d-2}\right)$ will be required.
 
Recall that $F=F^V+F^t+F^{V^2}+f$  where $F^V, F^t$ and $F^{V^2}$ have been defined in \eqref{F_t} while $f$ is defined in \eqref{globa:f=f(f1,f2)}.
Now,   taking $s=\frac{d}{p_2}-1$, $p=p_2$ and $\sigma=1$ in \eqref{lem:pro:lf:hf} yields
\begin{multline*}
     \normed{\Dj F^V}\lesssim 2^{-j(\frac{d}{2}-1)}c_j
        \suma\biggl( \NBH{\frac{d}{p_2}}{p_2}{1}{\Bar{S}^\alpha-S^\alpha(U)}  \NBH{\frac{d}{p_2}-1}{p_2}{1}{\pal V}\\+ \NBH{\frac{d}{p_2}-1}{p_2}{1}{\Bar{S}^\alpha-S^\alpha(U)}  \NBH{\frac{d}{p_2}}{p_2}{1}{\pal V}\biggr)\cdotp
\end{multline*}
Bearing in mind \eqref{smaal:cond:V}, to control the first term (resp. the second term) in the r.h.s.,  we leverage \eqref{comp:u:propo:p} (resp. \eqref{comp:uv:propo:1}) and we get:
\begin{align}
     \label{est:Fv:til:V}
     \normed{\Dj F^V}
       \lesssim 2^{-j(\frac{d}{2}-1)} c_j\biggl(\NBH{\frac{d}{p_2}}{p_2}{1}{V}\NBH{\frac{d}{p_2}-1}{p_2}{1}{\n V}+ \NBH{\frac{d}{p_2}-1}{p_2}{1}{V}\NBH{\frac{d}{p_2}}{p_2}{1}{\n V}\biggr).
\end{align}

Similarly, we obtain  
\begin{align}
    \label{est:Ft}\nonumber
        \normed{\Dj F^t}&\lesssim 2^{-j(\frac{d}{2}-1)}c_j \biggl(\NBH{\frac{d}{p_2}}{p_2}{1}{(\overline{S}^0-S^0(U))}\NBH{\frac{d}{p_2}-1}{p_2}{1}{\pt V}+ \NBH{\frac{d}{p_2}-1}{p_2}{1}{(\overline{S}^0-S^0(U))}\NBH{\frac{d}{p_2}}{p_2}{1}{\pt V}\biggr)\\\nonumber
        &\lesssim 2^{-j(\frac{d}{2}-1)}c_j\biggl( \NBH{\frac{d}{p_2}}{p_2}{1}{V}\NBH{\frac{d}{p_2}-1}{p_2}{1}{\pt V}+\NBH{\frac{d}{p_2}-1}{p_2}{1}{V}\NBH{\frac{d}{p_2}}{p_2}{1}{\pt V}\biggr)\\\nonumber
        &\lesssim 2^{-j(\frac{d}{2}-1)}  c_j
        \NBH{\frac{d}{p_2}}{p_2}{1}{V}\biggl( \NBH{\cdp{p_2}+1}{p_2}{1}{V}+\NBH{\cdp{p_2}}{p_2}{1}{V} \biggr) \\
        &\hspace{1cm}+2^{-j(\frac{d}{2}-1)}  c_j
        \NBH{\frac{d}{p_2}-1}{p_2}{1}{V}\biggl(\NBH{\frac{d}{p_2}+1}{p_2}{1}{V}+ \NBH{\frac{d}{p_2}+2}{p_2}{1}{V^2}+\NBH{\cdp{p_2}+1}{p_2}{1}{V}\NBH{\cdp{p_2}+1}{p_2}{1}{V^2}\biggr),
\end{align}
where in the last inequality  we used Inequality \eqref{est:ptV:w:v} both in the two terms, with $\kappa=\cdp{p_2}-1$, $q=p_2$  and $\kappa=\cdp{p_2}$, $q=p_2$ respectively.

Next, combining  \eqref{lem:pro:lf:hf} (with $s=\frac{d}{p_2}-1$, $p=p_2$ and $\sigma=1$), \eqref{comp:uv:propo:1} and \eqref{comp:u:propo:p} and the smallness condition \eqref{smaal:cond:V}, we obtain  
$$\NBH{\frac{d}{2}-1}{2}{1}{f} \lesssim \NBH{\frac{d}{2}-1}{2}{1}{\n V\otimes \n V}
+\NBH{\frac{d}{p_2}}{p_2}{1}{V}\NBH{\frac{d}{p_2}-1}{p_2}{1}{\n V\otimes \n V}+
\NBH{\frac{d}{p_2}-1}{p_2}{1}{V}\NBH{\frac{d}{p_2}}{p_2}{1}{\n V\otimes \n V}$$
whence, using once more the same inequalities,
\begin{align}
    \label{est:f:lf}
    \normed{\Dj f}\lesssim  2^{-j(\frac{d}{2}-1)}c_j\biggl(
(1+ \NBH{\frac{d}{p_2}}{p_2}{1}{V})\NBH{\frac{d}{p_2}-1}{p_2}{1}{\n V}\NBH{\frac{d}{p_2}}{p_2}{1}{\n V}
    + \NBH{\frac{d}{p_2}-1}{p_2}{1}{V}\NBH{\frac{d}{p_2}}{p_2}{1}{\n V}^2\biggr)\cdotp
\end{align}
 Estimating  the last term $ F^{V^2}$ is based on the following decomposition:
\begin{equation}
    \label{dec:F:V2}
    F^{V^2}=\sumab \pal (r^{\alpha\beta}(U))\,\pbe V+ \sumab r^{\alpha\beta}(U)\partial^2_{\alpha\beta} V ,
\end{equation}
where $r^{\alpha\beta}(U)$ is the function defined in \eqref{F_t}. In particular, we have $ r^{\alpha\beta}(\overline{U})=0$
As $p_2<2d$ and $2\le p_2\le \min\left(4,\frac{2d}{d-2}\right)$, applying again \eqref{lem:pro:lf:hf} with $s=\frac{d}{p_2}-1$, $p=p_2$ and $\sigma=1$, using Proposition \ref{propo:comp} and observing that the matrices $r^{\alpha\beta}$ don't have any effect on the component  $V^1$   provides
\begin{align*}
\normed{ (\pal r^{\alpha\beta}(U))\,\pbe V}&\lesssim 2^{-j(\frac{d}{2}-1)} c_j \biggl( \NBH{\frac{d}{p_2}}{p_2}{1}{ V}\NBH{\frac{d}{p_2}}{p_2}{1}{\n V^2}+ \NBH{\frac{d}{p_2}}{p_2}{1}{\n V}\NBH{\frac{d}{p_2}}{p_2}{1}{ V^2} \biggr),\\
 \normed{  r^{\alpha\beta}(U)\partial^2_{\alpha\beta} V} &\lesssim 2^{-j(\frac{d}{2}-1)} c_j \biggl(  \NBH{\frac{d}{p_2}-1}{p_2}{1}{V}\NBH{\frac{d}{p_2}}{p_2}{1}{\n^2 V^2} + \NBH{\frac{d}{p_2}}{p_2}{1}{V}\NBH{\frac{d}{p_2}+1}{p_2}{1}{ V^2} \biggr),\end{align*}
whence
\begin{equation}\label{est:FV2}
        \normed{\Dj F^{V^2}}\lesssim 2^{-j(\frac{d}{2}-1)} c_j \biggl(\NBH{\frac{d}{p_2}}{p_2}{1}{V}\NBH{\frac{d}{p_2}+1}{p_2}{1}{ V}+ \NBH{\frac{d}{p_2}-1}{p_2}{1}{V}\NBH{\frac{d}{p_2}}{p_2}{1}{\n^2 V^2} \biggr)\cdotp
\end{equation}  

\paragraph{\textbf{Last step : Final estimate of $\Lj$}}
Inserting \eqref{est:Fv:til:V}, \eqref{est:Ft},  \eqref{est:f:lf}  and 
\eqref{est:FV2} into \eqref{01:est:ly:l}  and observing that 
$$\NBH{\frac{d}{p_2}}{p_2}{1}{ V}^2\lesssim  
\NBH{\frac{d}{p_2}-1}{p_2}{1}{V}\NBH{\frac{d}{p_2}+1}{p_2}{1}{ V}$$
yields
 \begin{multline}\label{2:est:ly:l}
    \frac{d}{dt}\Lj+  2^{2j}\sum_{k=0}^{n-1}\intd \ee_k |B_\omega A_\omega^{k}\widehat{V}_j|^2 \lesssim    
     2^{-j(\frac{d}{2}-1)}c_j \normed{V_j}  \biggl((1+ \NBH{\frac{d}{p_2}-1}{p_2}{1}{V})\NBH{\frac{d}{p_2}}{p_2}{1}{\n V}^2\\
    + (\NBH{\frac{d}{p_2}}{p_2}{1}{V}+ \NBH{\frac{d}{p_2}-1}{p_2}{1}{V})  \NBH{\frac{d}{p_2}+1}{p_2}{1}{ V}
    + \NBH{\frac{d}{p_2}-1}{p_2}{1}{V}\NBH{\frac{d}{p_2}+2}{p_2}{1}{ V^2}
    \biggr)\cdotp
       \end{multline}

       Thanks to  Shizuta-Kawashima (SK) condition (see Lemma \ref{lem:SK_diss}) and the fact that
\begin{align}
    \label{eq:Lj:Vj}
    \Lj\simeq \normed{V_j}^2
\end{align}
we thus get some $\bar c_0$ such that for all $j\le N_{0}+1$
\begin{multline}
    \label{3:est:ly:l}
\frac{d}{dt}\Lj+  2^{2j}\bar c_0\Lj \lesssim  2^{-j(\frac{d}{2}-1)}c_j \sqrt{\Lj }  \biggl( 
(1+ \NBH{\frac{d}{p_2}-1}{p_2}{1}{V})\NBH{\frac{d}{p_2}}{p_2}{1}{\n V}^2\\
    + (\NBH{\frac{d}{p_2}}{p_2}{1}{V}+ \NBH{\frac{d}{p_2}-1}{p_2}{1}{V})  \NBH{\frac{d}{p_2}+1}{p_2}{1}{ V}
    + \NBH{\frac{d}{p_2}-1}{p_2}{1}{V}\NBH{\frac{d}{p_2}+2}{p_2}{1}{ V^2}\biggr)\cdotp
\end{multline}
Therefore taking $p=2$ and $X= \Lj$ in Lemma  \eqref{lem_der_int}
 then multiplying by $2^{j(\frac{d}{2}-1)} $  delivers for all $j\le N_{0}+1$ and $t\in [0,T]$
\begin{multline}
    \label{4:est:ly:l}
    2^{j(\frac{d}{2}-1)}\sqrt{\Lj(t)}+  2^{j(\frac{d}{2}+1)}\bar c_0\int^t_0\sqrt{\Lj} \lesssim 2^{j(\frac{d}{2}-1)}\sqrt{\Lj(0)} \\
    + \int^t_0 c_j\biggl((1+ \NBH{\frac{d}{p_2}-1}{p_2}{1}{V})\NBH{\frac{d}{p_2}+1}{p_2}{1}{V}^2
    + (\NBH{\frac{d}{p_2}}{p_2}{1}{V}+ \NBH{\frac{d}{p_2}-1}{p_2}{1}{V})  \NBH{\frac{d}{p_2}+1}{p_2}{1}{ V}
    + \NBH{\frac{d}{p_2}-1}{p_2}{1}{V}\NBH{\frac{d}{p_2}+2}{p_2}{1}{ V^2}\biggr)\cdotp\end{multline}

Let us now simplify the right-hand side of the previous inequality. Thanks to \eqref{est:V:Bs:p:Vi} and \eqref{pr:int:V:V:} we have
\begin{align*}
    \int^t_0 (\NBH{\frac{d}{p_2}-1}{p_2}{1}{V}+ \NBH{\frac{d}{p_2}}{p_2}{1}{V}) \NBH{\frac{d}{p_2}+1}{p_2}{1}{ V}
    &\le (\LpNBH{\cdp{p_2}-1}{p_2}{1}{V}{\infty}+\LpNBH{\cdp{p_2}}{p_2}{1}{V}{\infty})\LpNBH{\cdp{p_2}+1}{p_2}{1}{V}{1}\\
    &\lesssim \mathbb{V}_{p_2}^2,\end{align*}
    \begin{equation*}
    \int^t_0 \NBH{\frac{d}{p_2}-1}{p_2}{1}{V}\NBH{\frac{d}{p_2}+2}{p_2}{1}{ V^2}\lesssim \LpNBH{\cdp{p_2}-1}{p_2}{1}{V}{\infty}\LpNBH{\cdp{p_2}+2}{p_2}{1}{V^2}{1}\lesssim  \mathbb{V}_{p_2}^2
\end{equation*}
and 
\begin{gather*}
    \int^t_0  (1+\NBH{\frac{d}{p_2}-1}{p_2}{1}{V})     \NBH{\frac{d}{p_2}+1}{p_2}{1}{ V}^2 \lesssim (1+\LpNBH{\cdp{p_2}-1}{p_2}{1}{V}{\infty}) \int^t_0 \NBH{\frac{d}{p_2}+1}{p_2}{1}{ V}^2\lesssim (1+\mathbb{V}_{p_2})\mathbb{V}_{p_2}^2\cdotp 
\end{gather*}
 Inequality \eqref{2:est:ly:l} becomes

\begin{equation}
    \label{5:est:ly:l}
    2^{j(\frac{d}{2}-1)}\sqrt{\Lj(t)}+  2^{j(\frac{d}{2}+1)}\bar c_0\int^t_0\sqrt{\Lj} \lesssim 2^{j(\frac{d}{2}-1)}\sqrt{\Lj(0)} +c_j (1+ \mathbb{V}_{p_2})\mathbb{V}_{p_2}^2.
\end{equation}
       Summing over $j\le N_{0}+1$ and remembering \eqref{eq:Lj:Vj} yields \eqref{est:lf:V}.
 \end{proof}


\section{Proof of Theorem \ref{thm:p}}\label{proof:Thms:p:and:cri:sec}

Granted with the a priori estimates of the previous section, we are now ready to prove Theorem \ref{thm:p}.

\subsection{Existence}

Proving the existence of a global solution under the assumptions of Theorem \ref{thm:p} will be carried out in three steps:
\begin{enumerate}
    \item We smooth out the data and get a sequence of maximal smooth solutions $(V^1_n,V^2_n)_{n\ge 1}$ on $[0,T^*_n)$ to \eqref{eq:base:V} on $\mathbb{R}^d$ by using our recent result in \cite{DanADOlocal}.
    \item We prove uniform estimates on $(V^1_n,V^2_n) $ on $[0,T_n^*)$, by taking advantage of Propositions \ref{propo:est:hf} and \ref{propo:est:lf:V}. This will enable us to conclude that $T_n^*=\infty$.
    \item We use compactness to prove that the sequence $(V^1_n, V^2_n)$ converges, up to extraction,
to a solution of System \eqref{eq:base:V}.
\end{enumerate}

\paragraph{\textbf{$\blacklozenge$ First step: Construction of approximate solutions.}} In order to apply the local-in-time results of \cite{DanADOlocal}, we have to work with data belonging to space $\B{\cd+1}{2}{1}\times \B{\cd}{2}{1}$. For that purpose, we smooth out the data as follows: For all  $n\in \mathbb{N}$, we set

\begin{align}
    \label{def:V_0:n}
    V_{0,n}\defn \chi_n V_0^{l} +\chi_n \Dot{S}_{2n}V_0^h
\end{align}
with $\chi_n\defn \chi( n^{-1}\cdot)$, where $\chi$ stands  (for instance) for a smooth function with range in $[0,1]$, supported in $]-4/3, 4/3[$
and such that $\chi\equiv 1$ on $[-3/4, 3/4]$.
It is obvious that the sequence $(V_{0,n})_{n\in\N} $  tends to $V_0=(V_0^1,V_0^2)$ in the sense of
distributions, when $n$ tends to infinity. We claim that $V_{0,n} \in \B{\sigma}{2}{1}$, for all $\sigma\ge d/2.$ 
First, since $\chi_n$ is in $L^2$ and 
both $V_0^l$ and $V_0^h$ are bounded, we have
$V_0^n\in L^2.$ Hence it suffices to establish 
that $V_{0,n} \in \dot B^{\sigma}_{2,1}$ for all $\sigma\ge d/2.$ 
\smallbreak
For the first part of $V_{0,n},$  the product law  \eqref{product_propo1}
 guarantees that
\begin{align*}
      \NBH{\sigma}{2}{1}{ \chi_n V_0^{l} } \le \normeinf{\chi_n} \NBH{\sigma}{2}{1}{ V_0^{l} }+ \normeinf{V_0^{l}}  \NBH{\sigma}{2}{1}{  \chi_n} 
\end{align*}
which becomes, using \eqref{link:lf:hf}, the embedding $L^\infty \hookrightarrow \BH{\cd}{2}{1}$ and the the scaling properties $ \NBH{\sigma}{2}{1}{  \chi_n} \simeq n^{-(\sigma-\cd)} \NBH{\sigma}{2}{1}{  \chi} $ (see Remark 2.19 of \cite{HajDanChe11})
\begin{align}
    \label{est:lf:data:n:data}
     \NBH{\sigma}{2}{1}{ \chi_n V_0^{l} } \le 2^{N_0(\sigma-\cd+1)}\NBH{\cd-1}{2}{1}{ V_0 }^{l} + 2^{N_{0}}n^{-(\sigma-\cd)}\NBH{\cd-1}{2}{1}{ V_0 }^{l} \NBH{\sigma}{2}{1}{  \chi}.
\end{align}
To show that the second part of $V_{0,n}$ is also in $\dot B^\sigma_{2,1},$
 use the decomposition
 \begin{align}
 \label{decom::chin:}
     \chi_n \Dot{S}_{2n}V_0^h= T_{ \Dot{S}_{2n}V_0^h}\chi_n + T_{\chi_n} \Dot{S}_{2n}V_0^h+ R({\chi_n} ,\Dot{S}_{2n}V_0^h).
 \end{align}
Paraproduct properties give for all $\sigma\ge \cd$
\begin{align*}
    \NB{\sigma}{2}{1}{T_{ \Dot{S}_{2n}V_0^h}\chi_n}&\lesssim \NB{-1}{\infty}{\infty}{ \Dot{S}_{2n}V_0^h} \NB{\sigma+1}{2}{1}{\chi_n}
\end{align*}
which becomes using  the scaling properties of the space $ \B{\sigma+1}{2}{1} $ and Besov embedding
\begin{align*}
    \NB{\sigma}{2}{1}{T_{ \Dot{S}_{2n}V_0^h}\chi_n}\lesssim n^{-(\sigma+1-d/2)} \NB{\sigma+1}{2}{1}{\chi} \NB{\cdp{p_2}-1}{p_2}{1}{ V_0}^h.
\end{align*}
Similarly, from  paraproduct properties we have also for all $\sigma_0\ge\sigma\ge \cd$, 
\begin{align*}
    \NB{\sigma}{2}{1}{ T_{\chi_n} \Dot{S}_{2n}V_0^h}& \lesssim \NB{\sigma-\sigma_0}{q}{1}{\chi_n }\NB{\sigma_0}{p_2}{1}{ \Dot{S}_{2n}V_0^h }\;\; \esp{with} \frac{1}{2}=\frac{1}{p_2}+\frac{1}{q}\\
    &\lesssim n^{-(\sigma-\sigma_0-d/q)} \NB{\sigma-\sigma_0}{q}{1}{\chi }\NB{\sigma_0}{p_2}{1}{ \Dot{S}_{2n}V_0^h }\\
    &\lesssim 2^{(n+N)(\sigma_0-d/p_2+1)}n^{-(\sigma-\sigma_0-d/q)} \NB{\sigma-\sigma_0}{q}{1}{\chi }\NBH{\cdp{p_2}-1}{p_2}{1}{V_0}^h.
\end{align*}
Finally, using again  paraproduct properties we have thus for all $\sigma\ge \cd$ 
\begin{align*}
    \NB{\sigma}{2}{1}{ R({\chi_n}, \Dot{S}_{2n}V_0^h)}& \lesssim \NB{\sigma-\cdp{p_2}+1}{q}{1}{\chi_n }\NB{\cdp{p_2}-1}{p_2}{1}{ \Dot{S}_{2n}V_0^h }\\
    &\lesssim   n^{-(\sigma-\cdp{p_2}+1-d/q)} \NB{\sigma-\cdp{p_2}+1}{q}{1}{\chi }\NBH{\cdp{p_2}-1}{p_2}{1}{V_0}^h,
\end{align*}
which completes the proof of $V_{0,n}\in \B{\sigma}{2}{1} $, for all 
$\sigma\ge \cd$.
\medbreak
Next, we claim that we have for $i=1,2$ and $\sigma_i=\cdp{p_i}$
\begin{align}
    \label{comp:data:n:data}
    \NBH{\cdp{2}-1}{2}{1}{V_{0,n}}^{l}+ \NBH{\sigma_i+1}{p_i}{1}{V_{0,n}^1}^h+ \NBH{\sigma_i}{p_i}{1}{V_{0,n}^2}^h \lesssim   \NBH{\cdp{2}-1}{2}{1}{V_{0}}^{l} +\NBH{\sigma_i+1}{p_i}{1}{V_{0}^1}^h+ \NBH{\sigma_i}{p_i}{1}{V_{0}^2}^h .
\end{align}
Indeed,  using \eqref{product_propo2} and  the scaling properties  and Besov embedding, one gets
\begin{align}
\label{2:comp:data:n:data}
\begin{split}
      \NBH{\cdp{p_i}}{p_i}{1}{V_{0,n}^2}^h&\lesssim \NBH{\cdp{p_i}}{p_i}{1}{\chi} \NBH{\cdp{2}}{2}{1}{V_{0}^2}^{l}+  \NBH{\cdp{p_i}}{p_i}{1}{\chi}  \NBH{\cdp{p_i}}{p_i}{1}{\Dot{S}_{2n}V_{0}^{2,h}}
      \lesssim \NBH{\cdp{2}-1}{2}{1}{V_{0}^2}^{l}+ \NBH{\cdp{p_i}}{p_i}{1}{V_{0}^{2,h}}
      \cdotp\end{split}
\end{align}

From \eqref{product_propo3:chap2}, it holds
\begin{multline}
\label{3:comp:data:n:data}
     \NBH{\cdp{p_i}+1}{p_i}{1}{V_{0,n}^1}^h \lesssim \NBH{\cdp{p_i}}{p_i}{1}{\chi} \NBH{\cdp{2}+1}{2}{1}{V_{0}^1}^{l}+ n^{-1}\NBH{\cdp{p_i}}{p_i}{1}{V_{0}^1} \NBH{\cdp{p_i}+1}{p_i}{1}{\chi}\\
   +  \NBH{\cdp{p_i}}{p_i}{1}{\chi} \NBH{\cdp{p_i}+1}{p_i}{1}{\Dot{S}_{2n}V_{0}^{1,h}}+n^{-1} \NBH{\cdp{p_i}}{p_i}{1}{\Dot{S}_{2n}V_{0}^{1,h}} \NBH{\cdp{p_i}+1}{p_i}{1}{\chi} \lesssim \NBH{\cdp{2}}{2}{1}{V_{0}^1}^{l}+ \NBH{\cdp{p_i}+1}{p_i}{1}{V_{0}^{1,h}}.
\end{multline}
To bound the term $  \NBH{\cdp{2}-1}{2}{1}{V_{0,n}}^{l}$, we use the decomposition \eqref{def:V_0:n} and (the homogeneous version of) Bony's decomposition \eqref{decom::chin:}. We get:
\begin{equation*}
     \NBH{\cd-1}{2}{1}{V_{0,n}}^{l}\lesssim  \NBH{\cdp{2}}{2}{1}{\chi_n}\NBH{\cdp{2}-1}{2}{1}{ V_0}^{l}
     + \NBH{\cd-1}{2}{1}{ \Dot{T}_{\Dot{S}_{2n}V_0^h}\chi_n + \Dot{T}_{\chi_n} \Dot{S}_{2n}V_0^h+ \Dot{R}({\chi_n} ,\Dot{S}_{2n}V_0^h)}^{l}\cdotp
\end{equation*}
The first term can clearly be bounded as we want in \eqref{comp:data:n:data}.
For the other terms, since $2\le p_2<2d$ and $p_2\leq \min(4,\frac{2d}{d-2}),$
taking $s=\cdp{p_2}-1$, $\sigma=1$ in \eqref{lem:T:lf:hf} and \eqref{lem:R:lf:hf} yields
    \begin{equation*}
        \NBH{\cd-1}{2}{1}{\bigl(\Dot{T}_{\Dot{S}_{2n}V_0^h}\chi_n, \Dot{R}({\chi_n} ,\Dot{S}_{2n}V_0^h)\bigr) }\lesssim \NBH{\cdp{p_2}-1}{p_2}{1}{\Dot{S}_{2n}V_0^h}\NBH{\cdp{p_2}}{p_2}{1}{\chi_n}.
    \end{equation*}
    Bounding (the low frequencies of) the term $ \Dot{T}_{\chi_n} \Dot{S}_{2n}V_0^h $  is more delicate.
    For all $j\le N_0+1$, one has
    \begin{align*}
        \dot{\Delta}_j \Dot{T}_{\chi_n} \Dot{S}_{2n}V_0^h= \sum_{|k-j|\le 4} \dot{\Delta}_j(\dot{S}_{k-1} \chi_n\,\dot{\Delta}_k\dot{S}_{2n}V_0^h).
    \end{align*}
    Assuming that $n$ is large enough, we have
    \begin{align}
        \label{S2n:V_0:h}
      V_0^h=\sum_{m\ge N_0+1} \Dot{\Delta}_m V_0 \esp{and}  \Dot{S}_{2n}V_0^h= \sum_{N_0+1\le l\le 2n} \Dot{S}_{2n}\Dot{\Delta}_l V_0,
    \end{align}
    from which we deduce that for all $j\le N_0+1,$
    \begin{align*}
         \dot{\Delta}_j \Dot{T}_{\chi_n} \Dot{S}_{2n}V_0^h= \sum_{\substack{|k-j|\le 1 \\ N_0\le j\le N_0+1 }} \Dot{S}_{k-1}\chi_n \Dot{\Delta}_j\Dot{\Delta}_k\Dot{S}_{2n}V_0^h +
         \sum_{|k-j|\le 4}  \bigl[ \Dot{\Delta}_j, \dot{S}_{k-1} \chi_n \bigr](\dot{\Delta}_k\dot{S}_{2n}V_0^h)\defn  A_j+B_j.
    \end{align*}
   Owing to H\"{o}lder inequality the first term in the previous inequality is bounded as follows:
   \begin{align*}
       \sum_{j\le N_0+1}2^{j(\cd-1)} \normede{A_j}\le C_{N_0}\|\chi\|_{L^q}\|(\Delta_{N_0}V_0,\Delta_{N_0+1}V_0)\|_{L^{p_2}}\lesssim \NBH{\cd-1}{2}{1}{V_0}^l.
   \end{align*}
   Next, from \eqref{S2n:V_0:h} it holds that
   \begin{align*}
        \normede{B_j}\le \sum_{\substack{N_0\le k\le 2n+1\\ |k-j|\le 4 }}  \bigl[ \Dot{\Delta}_j, \dot{S}_{k-1} \chi_n \bigr](\dot{\Delta}_k\dot{S}_{2n}V_0^h)
   \end{align*}
   from which and Lemma \cite[Lemma 2.97]{HajDanChe11} we deduce 
   since $\frac{d}{\widetilde{p}_2}-1\le 0$ that
   that 
   \begin{align*}
       \normede{B_j}&\lesssim \sum_{\substack{N_0\le k\le 2n+1\\ |k-j|\le 4 }}  2^{-j} \normes{\n \Dot{S}_{k-1}\chi_n}{\widetilde{p}_2} \normes{\Dot{\Delta}_k\dot{S}_{2n}V_0^h}{p_2} \esp{with} \frac{1}{\widetilde{p}_2}=\frac{1}{2}-\frac{1}{p_2}\\
        &\lesssim  \sum_{\substack{N_0\le k\le 2n+1\\ |k-j|\le 4 }} 2^{-k(\frac d2-1)}\,2^{k(\frac{d}{\widetilde{p}_2}-1)} \normes{\n \Dot{S}_{k-1}\chi_n}{\widetilde{p}_2} 2^{k(\cdp{{p}_2}-1)} \normes{\Dot{\Delta}_k V_0^h}{p_2}\\
      &\lesssim 2^{-j(\frac d2-1)}\sum_{|\nu|\le 4}  2^{(j+\nu)(\frac{d}{p_2}-1)}  \normes{\Dot{\Delta}_{j+\nu}V_0^h}{p_2}  \sum_{\substack{N_0\le k\le 2n+1 }} 2^{(k-1)(\frac{d}{\widetilde{p}_2}-1)} \normes{\n \Dot{S}_{k-1}\chi_n}{\widetilde{p}_2}     \\
       &\lesssim c_j 2^{-(\cd-1)j}  \NBH{\cdp{\widetilde{p}_2}}{\widetilde{p}_2}{1}{ \chi} \NBH{\cdp{p_2}-1}{p_2}{1}{V_0^h}  \esp{with} \sum_{l}c_l\le 1.
   \end{align*}

We conclude that 
\begin{align}
  \label{1:comp:data:n:data}
  \begin{split}
    \NBH{\cd-1}{2}{1}{V_{0,n}}^{l}
    &\lesssim (1 +\NBH{\cdp{2}}{2}{1}{\chi})\NBH{\cdp{2}-1}{2}{1}{V_0}^{l}+ \NBH{\cdp{p_2}-1}{p_2}{1}{\Dot{S}_{2n}V_0^h}\NBH{\cdp{p_2}}{p_2}{1}{\chi}+\NBH{\cdp{\widetilde{p}_2}}{\widetilde{p}_2}{1}{ \chi} \NBH{\cdp{p_2}-1}{p_2}{1}{V_0^h}.
     \end{split}
\end{align}

Finally, 
from \eqref{3:comp:data:n:data}, \eqref{2:comp:data:n:data}, \eqref{1:comp:data:n:data} and the fact that for all $\sigma^{\prime}\ge d/p_i-1$ and large $n$,
\begin{align}
\label{comp:data:n:hf}\nonumber
     \NBH{\sigma^{\prime}}{p_i}{1}{\Dot{S}_{2n}z^h} &\lesssim  2^{N\sigma^{\prime}}\lVert \Dot{\Delta}_{N_0} z \rVert_{L^{p_i}}+ 2^{(N_0+1)\sigma^{\prime}}\lVert \Dot{\Delta}_{N_0+1} z \rVert_{L^{p_i}}+\NBH{\sigma^{\prime}}{p_i}{1}{ \Dot{S}_{2n}z^h}^h\\
     & \lesssim  \NBH{\cdp{2}-1}{2}{1}{ z}^{l}+  \NBH{\sigma^{\prime}}{p_i}{1}{ z}^h,
\end{align}
one concludes the proof of \eqref{comp:data:n:data}.
\medbreak
Now, from the local existence theory (see \cite{DanADOlocal} ), for all $n\ge 3$; we get a maximal solution
$(V^1_n,V^2_n) $ over the time interval $[0, T^*_n)$ such that for all $ T< T^*_n$ we have 
\begin{align*}
     V^1_n\in  {C}([0,T];\B{\sigma+1}{2}{1}) \esp{and} V^2_n\in  {C}([0,T];\B{\sigma}{2}{1})\cap  L^1_{T}(\B{\sigma+2}{2}{1}), \text{    for all    } \sigma\ge \frac{d}{2}\cdotp
\end{align*}
Since $\B{\sigma}{2}{1}$ is embedded in $\BH{\sigma}{2}{1}$, for $\sigma>0$, it is easy to prove from \eqref{syste_part_hper_diff} and the composition and 
product laws \eqref{comp:u:propo:p} and \eqref{product_propo2} that $\pt V$ is in $L^\infty (0,T; \BH{\frac{d}{2}-1}{2}{1} )$, and as $V_{0,n}$ belongs to $\BH{\frac{d}{2}-1}{2}{1}$, we deduce that $V_n$ is actually in $C([0,T_n^*); \BH{\frac{d}{2}-1}{2}{1}\cap \B{\sigma}{2}{1}  )$ for all $\sigma\ge \frac{d}{2}\cdotp$
Note that,
because $p_i\ge 2$; the embedding $ \BH{s}{2}{1} \hookrightarrow \BH{s-\frac{d}{2}+\frac{d}{p_i}}{p_i}{1} $ ensures that
\begin{align*}
\begin{split}
    &   V^{l}_n \in  \widetilde{C}([0,T];\BH{\cd-1}{2}{1}),\;\;\; V^{1,h}_n\in  \widetilde{C}([0,T];\BH{\sigma+1}{p_i}{1}),\;\\
       &V^{2,h}_n\in  \widetilde{C}([0,T];\BH{\sigma}{p_i}{1})\cap  L^1_{T}(\BH{\sigma+2}{p_i}{1}), \text{    for all    } \sigma\ge \frac{d}{p_i}\cdotp
\end{split}
\end{align*}

For the sequel, we define 
\begin{align}
    \label{def:W:n}
   W_n\defn   i|D|^{-1}(\mathcal{Z}^{-1}(D)\left( \Ss_{21}(D) V^1_n+  \Ss_{22}(D)V^2_n\right)  +V^2_n.
\end{align}
\paragraph{\textbf{ $\blacklozenge$ Second step:  Uniform estimates.}} Since the  solutions $(V^1_n,V^2_n)$ constructed in the first step are regular and decreasing, we can apply the results of Proposition \ref{propo:est:hf} and \ref{propo:est:lf:V} to bound our sequence. Let us denote 
\begin{multline*}
    \mathbb{V}_{p_i}^n(T)\defn  \LpNBH{\frac{d}{2}-1}{2}{1}{V_n}{\infty}^{l}+\LpNBH{\frac{d}{2}+1}{2}{1}{V_n}{1}^{l} + \Vert V^1_n\Vert_{L^\infty_T\cap L^1_T(\BH{\frac{d}{p_i}+1}{p_i}{1})}^h\\
    + 
    \LpNBH{\frac{d}{p_i}}{p_i}{1}{(W_n,V^2_n)}{\infty}^h + \LpNBH{\frac{d}{p_i}+2}{p_i}{1}{(W_n,V^2_n)}{1}^h  + 
    \LpNBH{\frac{d}{p_i}}{p_i}{1}{\partial_t V_n}{1}^h,
\end{multline*}
and then
\begin{align*}
    \mathbb{V}_{p_i}^n(0)=  \NBH{\frac{d}{2}-1}{2}{1}{V_{0,n})}^l + \NBH{\frac{d}{p_i}+1}{p_i}{1}{V^1_{0,n}}+\NBH{\frac{d}{p_i}}{p_i}{1}{(W_n,V^2_n)}.
\end{align*}
Thanks to \eqref{pt:V:est:p1}, \eqref{est:V:W:hf}, \eqref{est:ptV:w:v} and \eqref{est:lf:V} we get for all $t\in [0,T_n]$
\begin{align}
    \label{1:est:VV:n}
 \mathbb{V}_{p_i}^n(t)\le  \mathbb{V}_{p_i}^n(0)+ C\mathbb{V}_{p_2}^n(t) (1+ \mathbb{V}_{p_2}^n(t))(\mathbb{V}_{p_i}^n(t)+\mathbb{V}_{p_2}^n(t))+C2^{-N_{0}} \mathbb{V}_{p_i}^n(t),
\end{align}
where $C$ depending only $p_1,p_2,d$ and the matrices involved in the system \eqref{systemE}. Choosing $N_0$ sufficiently large, the last term of inequality \eqref{1:est:VV:n} may be absorbed by the l.h.s. Furthermore,
from a standard bootstrap argument, it is now easy to conclude that there exists a positive
constant $c_0$ such that if the data has been chosen so small as to satisfy
\begin{align}
    \label{VV0<c0}
    \mathbb{V}_{p_2,0}\le c_0,
\end{align}
then from \eqref{comp:data:n:data} and for all $i=1,2$ and $t\in [0,T_n]$
\begin{align}
\label{est:mathbb:V:n:final}
  \mathbb{V}_{p_i}^n(t)\le 2C \mathbb{V}_{p_i}^n(0)\le 2C \mathbb{V}_{{p_i},0}.
\end{align}
We are now going to prove that the uniform estimate \eqref{est:mathbb:V:n:final} is global (i.e $T_n^*=\infty$). Assume by contradiction that $ T^*_n $ is finite. Since $  \sup_{0\le t<T^*_n} \mathbb{V}_{p_2}^n (t) $ and $\LpNBH{\frac{d}{p_2}}{p_2}{1}{\pt V_n}{1}$ are dominated by $ 2C \mathbb{V}_{p_2} (0) $, then Besov embeddings implies that
\begin{align}
    \label{blow:Vn:BH}
    \begin{split}
   & \bullet U([0,T^*_n)\times \Rd) \text{  is contained  in a compact subset of         } \mathcal{U}, \text{   if   } \mathbb{V}_{{2}}(0) \text{  is sufficiently small, } \\
&\bullet  \int^{T_n^*}_0\Bigl(\normeinf{\n V_n}^2+ \normeinf{\pt V_n}\Bigr)<\infty\,.
   \end{split}
\end{align}
 Now according to the blow up criterion in \cite{DanADOlocal}  $ (V^1_n, V^2_n)$ may be continued beyond $T^*_n$ into a
solution $ (\widetilde{V}^1_n, \widetilde{V}^2_n) $ of \eqref{eq:base:V} which coincides with $ (V^1_n, V^2_n)$ on $[0, T^*)$ and such that, for some $T>T^*_n$;
\begin{align*}
    \widetilde{V}^1_n\in \mathcal{C}([0, T], \B{\cd+1}{2}{1}) \esp{and}  \widetilde{V}^2_n\in \mathcal{C}([0, T], \B{\cd}{2}{1})\cap L^1_T(\B{\cd+2}{2}{1}).
\end{align*}
Then the previous step ensures that $(\widetilde{V}^1_n, \widetilde{V}^2_n) \in \mathbb{E}_{p_i}(T)$ : This stands in contradiction with the
definition of $T^*_n.$ Hence $T^*_n=\infty$ and \eqref{est:mathbb:V:n:final} holds true globally.

\paragraph{\textbf{ $\blacklozenge$ Third step: Convergence.}}
In order to show that $(V)_n$ tends, up to subsequence, to some $V\in \mathbb{E}_{p_i}$, in the sense of distributions, that satisfies \eqref{eq:base:V}, under assumptions (\textbf{H1})-(\textbf{H7}) one can use Ascoli Theorem and suitable compact embeddings in a similar fashion as in \cite{BaratDan22D1}.

\subsection{Uniqueness}
\label{sec:uni:p}
 Here, for all $T > 0$, we shall prove the uniqueness  in the space $  \mathbb{F} $  defined by: for all $2\le p_2<2d$
\begin{align}
    \label{F_p}
    \mathbb{F}\defn \{ (V^1,V^2):\;\;  V^2\in L^\infty(0,T;\BH{\frac{d}{p_2}-1}{p_2}{1})\cap L^1(0,T;\BH{\frac{d}{p_2}+1}{p_2}{1}) \text{    and   }  V^1\in L^\infty_T(\BH{\frac{d}{p_2}}{p_2}{1})  \}.
\end{align}
Observe that, compared to the space $\mathbb{E}_{p_2}$, in high frequencies, the exponent is smaller. The reason for that is the usual loss of one derivative when proving stability estimates for quasilinear hyperbolic systems. The exponent for low frequencies looks to be the best one for controlling the nonlinearities. Note that Condition $p_2<2d$ appears
when bounding the difference of the component $V^2$ of two solutions in the space $L^\infty_T(\BH{\cdp{p_2}-1}{p_2}{1})\cap L^1_T(\BH{\cdp{p_2}+1}{p_2}{1})$. In fact, to control the non-linear terms, we use the product law $  \BH{\cdp{p_2}-1}{p_2}{1}\times \BH{\cdp{p_2}-1}{p_2}{1}  \longrightarrow \BH{\cdp{p_2}-1}{p_2}{1}    $, which is true if $p_2<2d$ (see the section \ref{tool:box:glo:Lp}).  To simplify, we assume that $f\equiv0$.

The key point to prove the uniqueness is the following lemma.

\begin{lem}
    \label{est:uni}
 Assume that $2\le p_2<2d$. Let $V_1$ and $V_1$ be two solutions of \eqref{eq:base:V} under assumptions (\textbf{H1})-(\textbf{H7})  on $[0,T]$ supplemented by initial data $V_{1,0}$ and $V_{2,0},$ respectively. Then, there exists a constant  $C>0$ depending on $ \LpNBH{\cdp{p_2}}{p_2}{1}{(V_2,V_1)}{\infty} $such that $ \tV\defn V_1-V_2  $ satisfies the following a priori estimate 
  \begin{multline}
      \label{est:lem:tV}
            \Vert\tV\Vert_{\mathbb{F}(T)} \le   C\biggl(\NBH{\frac{d}{p_2}-1}{p_2}{1}{\tV^2_0}+ \NBH{\frac{d}{p_2}}{p_2}{1}{\tV^1_0}+  \LpNBH{\cdp{p_2}}{p_2}{1}{V_2}{\infty}\LpNBH{\cdp{p_2}+1}{p_2}{1}{\tV^2}{1}\\
    + \sqrt{T}\LpNBH{\cdp{p_2}}{p_2}{1}{\tV}{2}\LpNBH{\cdp{p_2}}{p_2}{1}{ (V_1,V_2)}{\infty}+ \LpNBH{\cdp{p_2}}{p_2}{1}{\tV}{2}\LpNBH{\cdp{p_2}}{p_2}{1}{\n V_1^2}{2}\biggr)\cdotp
        \end{multline}
\end{lem}

\begin{rmq}
Let us emphasize that the stability inequality \eqref{est:lem:tV} is derived without relying on any smallness condition.
\end{rmq}

\begin{proof} Throughout the proof, we agree that $A\lesssim B$ means that there exists a positive constant $C$ depending only on $p,d$, the matrices involved in system \eqref{systemE} and $  \LpNBH{\cdp{p_2}}{p_2}{1}{(V_2,V_1)}{\infty} $ such that $A\le CB$.  We divide the proof of Lemma \ref{est:uni} into several steps. Let us start with the following estimate for $\pt \tV$.
\paragraph{\textbf{Step 1: Estimates for $\pt \tV$.}}
\begin{lem}
    \label{lem:pt:tV:k}
 For all  $p_2\in [1,2d[ $ and $\kappa\in [\frac{d}{q}-1,\frac{d}{q}]$, we have for some 
 constant $C$ depending only on $d,$ $p_2,$ the parameters of the system and $  \LpNBH{\cdp{p_2}}{p_2}{1}{(V_2,V_1)}{\infty} $ :
\begin{multline}
    \label{pt:tV:k}
    \NBH{\kappa}{p_2}{1}{ \pt \tV } \leq C\biggl( \NBH{\kappa}{p_2}{1}{\nabla \tV}+ \NBH{\frac{d}{p_2}}{p_2}{1}{\tV}\NBH{\kappa}{p_2}{1}{\nabla V_2}\\
     + \NBH{\kappa+2}{p_2}{1}{ \tV^2}+ \NBH{\kappa+1}{p_2}{1}{V_1}\NBH{\frac{d}{p_2}+1}{p_2}{1}{ \tV^2} 
    + \NBH{\frac{d}{p_2}}{p_2}{1}{ \tV}  \NBH{\kappa+2}{p_2}{1}{ V^2_2}+  \NBH{\kappa+1}{p_2}{1}{\tV}\NBH{\frac{d}{p_2}+1}{p_2}{1}{ V^2_2}\\
     +  \NBH{\frac{d}{p_2}}{p_2}{1}{ \tV}\bigl(  \NBH{\kappa+2}{p_2}{1}{ V_2^2}   +\NBH{\kappa+1}{p_2}{1}{V_1}\NBH{\frac{d}{p_2}+1}{p_2}{1}{ V^2_2}\bigr)\biggr)\cdotp
    \end{multline}
   \end{lem}
\begin{proof}
From \eqref{sysdif::modif},  $\pt \tV$ satisfies
   \begin{align*}
       \pt \tV &=-\sum_{\alpha}{(S^0(U_1))^{-1}S^\alpha (U_1) \pal \tV}- \sum_{\alpha}{((S^0(U_1))^{-1}S^\alpha (U_1)-(S^0(U_2))^{-1}S^\alpha (U_2)) \pal V_2}  \\
       &+\sum_{\alpha,\beta}{ (S^0(U_1))^{-1}\pal (Y(U_1)^{\alpha\beta} \pbe \tV)  }+ \sum_{\alpha,\beta}{ (S^0(U_2))^{-1}\pal \left((Y(U_1)^{\alpha\beta}-Y(U_2)^{\alpha\beta}) \pbe V_2\right)  }\\
       &+ \sum_{\alpha,\beta}{ \left( (S^0(U_1))^{-1}-  (S^0(U_2))^{-1} \right)\pal (Y(U_1)^{\alpha\beta} \pbe V_2)  }
.
   \end{align*}
 As $1\le p_2<2d$ and $\kappa\in [\frac{d}{p_2}-1,\frac{d}{p_2}]$, taking advantage of  Proposition \ref{propo_produc_BH}  yields \eqref{pt:tV:k}.
 \end{proof}

\paragraph{\textbf{Step 2: Estimates for $ \tV^1$.}}
Instead of $\tV^1$, we shall obtain estimates for $$\tz\defn z_1-z_2= P(V_1^1-V^1_2)=P\tV^1.$$ In fact, since $P$ is an invertible and constant matrix, any Besov norm of $\tV^1$ and $\tz$ are equivalent.
\smallbreak
Remark that, from \eqref{eq:V1=f(H)},$ \tz  $  is solution of:
\begin{align}
    \label{eq:tz}
     &\pt  \tz  + \suma H^\alpha(U_1)\pal \tz=\widetilde{h}^{V^2}(U_1)-\widetilde{h}^{V^2}(U_2)- \suma \left( H^\alpha(U_2)-H^\alpha(U_1)\right)\pal z_2
\end{align}
with 
\begin{align*}
    \widetilde{h}^{V^2}(U)\defn - \suma P ( S^0_{11}(U))^{-1} S^\alpha_{12}(U)\pal V^2.
\end{align*}

    Following the proof of Lemma \ref{est:V1:hf:lem}, we deduce the following estimate 
    \begin{multline}
    \label{est:tv1=f(source:terms)}
         \LpNBH{\frac{d}{p_2}}{p_2}{1}{\tV^1}{\infty}\lesssim \NBH{\frac{d}{p_2}}{p_2}{1}{\tV^1_0}+\int^T_0 \NBH{\frac{d}{p_2}}{p_2}{1}{\tV}\NBH{\frac{d}{p_2}}{p_2}{1}{\n V_1}\\
        + \int^T_0 \biggl( \NBH{\cdp{p_1}}{p_1}{1}{\widetilde{h}^{V^2}(U_1)-\widetilde{h}^{V^2}(U_2)}+ \NBH{\cdp{p_1}}{p_1}{1}{ \left( H^\alpha(U_2)-H^\alpha(U_1)\right) \partial_\alpha z_2 }\biggr)\cdotp
    \end{multline}
       
   In order to estimate the right-hand side, we have to use many times the fact that the product is an algebra on the space $\BH{\frac{d}{p_i}}{p_i}{1}$. Taking advantage of Proposition \ref{propo:comp}, we get
   \begin{align*}
       \int^T_0 \NBH{\cdp{p_2}}{p_2}{1}{ \left( H^\alpha(U_2)-H^\alpha(U_1)\right)\pal  z_2 }\lesssim   \sqrt{T}\LpNBH{\cdp{p_2}}{p_2}{1}{ \tV }{2}\LpNBH{\cdp{p_2}}{p_2}{1}{ \n V^1_2 }{\infty}.
   \end{align*}
Using the decomposition 
\begin{multline}
\label{decomp:tidle:h:V2}
    \widetilde{h}^{V^2}(U_1)-\widetilde{h}^{V^2}(U_2)= \suma P(S^0_{11}(\overline{U}))^{-1}S^\alpha_{12}(\overline{U})\pal  \tV^2  \\
      -\suma P(S^0_{11}(\overline{U}))^{-1}S^\alpha_{12}(\overline{U}) -(S^0_{11}(U_2))^{-1}S^\alpha_{12}(U_{2})) \pal  \tV^2\\
      +\suma P\left( (S^0_{11}(U_1))^{-1}S^\alpha_{12}(U_1)- (S^0_{11}(U_2))^{-1}S^\alpha_{12}(U_{2})\right)\pal V^2_1,
\end{multline}
combined with Proposition \ref{propo:comp}, we obtain
\begin{multline*}
     \int^T_0 \NBH{\cdp{p_2}}{p_2}{1}{  \widetilde{h}^{V^2}(U_1)-\widetilde{h}^{V^2}(U_2) }^h\\\lesssim 
      \LpNBH{\cdp{p_2}}{p_2}{1}{\n \tV^2}{1}+ \LpNBH{\cdp{p_2}}{p_2}{1}{V_2}{\infty}\LpNBH{\cdp{p_2}+1}{p_2}{1}{\tV^2}{1}+ \sqrt{T}\LpNBH{\cdp{p_2}}{p_2}{1}{\tV}{2}\LpNBH{\cdp{p_2}+1}{p_2}{1}{V^2_1}{\infty}.
\end{multline*}
Gathering the above estimates and using once more the embedding $ \BH{\cdp{p_2}}{p_2}{1} \hookrightarrow L^\infty $, we obtain for large enough $K>0$
    \begin{multline}
\label{est:tz:}
    \frac{1}{K( 1+\LpNBH{\cdp{p_2}}{p_2}{1}{V_2}{\infty}) }\LpNBH{\frac{d}{p_2}}{p_2}{1}{\tV^1}{\infty}  \le \NBH{\frac{d}{p_2}}{p_2}{1}{\tV^1_0}\\
    +\sqrt{T} \LpNBH{\frac{d}{p_2}}{p_2}{1}{\tV}{2}\LpNBH{\frac{d}{p_2}+1}{p_2}{1}{(V_2, V_1)}{2}
     +\frac{1}{2}\LpNBH{\cdp{p_2}+1}{p_2}{1}{\tV^2}{1} .
    \end{multline}
\paragraph{\textbf{Step 3: Estimates for $ \tV^2$ .}}
From \eqref{rel:W:V2}, we discover that $ \tV^2 $ is solution of 

\begin{equation}
    \label{eq:tW}
    \overline{S}^0_{22}\pt \tV^2 -\Delta\mathcal{Z}(D) \tV^2= (g^t(U_1)-g^t(U_2))+ (g^V(U_1)-g^V(U_2))+(g^{V^2}(U_1)-g^{V^2}(U_2)).
\end{equation}
   Following the lines of the proof of the Lemma \ref{est:M:lem} we arrive at
   \begin{multline*}
    \LpNBH{\frac{d}{p_2}-1}{p_2}{1}{\tV^2}{\infty} +  \LpNBH{\frac{d}{p_2}+1}{p_i}{1}{\tV^2}{1}\lesssim \NBH{\frac{d}{p_2}-1}{p_i}{1}{\tV^2_0} \\
    + \int^T_0  \NBH{\frac{d}{p_2}-1}{p_2}{1}{ (g^t(U_1)-g^t(U_2), g^V(U_1)-g^V(U_2),g^{V^2}(U_1)-g^{V^2}(U_2))}.
   \end{multline*}
   In the spirit of calculations leading to \eqref{est:tz:}, we are going to estimate the non linear terms on the right-hand side of the previous inequality. The general idea is similar to what we did for the term $ \widetilde{h}^{V^2}(U_1)-\widetilde{h}^{V^2}(U_2) $ in \eqref{est:tv1=f(source:terms)}. In order to bound the nonlinear terms, we have to decompose  each term of right-hand side of the above inequality (that is, $ g^t(U_1)-g^t(U_2), g^V(U_1)-g^V(U_2)$ and $g^{V^2}(U_1)-g^{V^2}(U_2)$) as follows:

   \begin{multline}
    \label{decomp:gt:V:V2}
       g^t(U_1)-g^t(U_2)=  (S^0_{22}(U_2)- S^0_{22}(U_1))\pt V_1^2+  (\overline{S}^0_{22}-S^0_{22}(U_2))\pt \tV^2,\\
        g^V(U_1)-g^V(U_2)= \suma\left((S^\alpha_{21}(U_2)-S^\alpha_{21}(U_1))\pal V^1_1+(S^\alpha_{22}(U_2)-S^\alpha_{22}(U_1))\pal V^2_1)\right)\\
        + \suma\left((\barS_{21}-S^\alpha_{21}(U_2))\pal \tV^1+(\barS_{22}-S^\alpha_{22}(U_2)\pal \tV^2)\right),\\
        g^{V^2}(U_1)-g^{V^2}(U_2)= \sumab \pal\left((Z^{\alpha\beta}(U_2)-Z^{\alpha\beta}(U_1))\pbe V^2_1\right)\\
        + \sumab \pal\left((\overline{Z}^{\alpha\beta}-Z^{\alpha\beta}(U_2))\pbe \tV^2\right).
   \end{multline}
Now, using Proposition \ref{propo:comp} and the fact that
the product maps $\BH{\frac{d}{p_2}}{p_2}{1}\times \BH{\frac{d}{p_2}-1}{p_2}{1}\mapsto \BH{\frac{d}{p_2}-1}{p_2}{1} $ (let us remember that $p_2<2d$), we easily get
\begin{multline}
\label{est:tg:t:V}
    \int^T_0 \NBH{\cdp{p_2}-1}{p_2}{1}{g^t(U_1)-g^t(U_2)}\lesssim \int^T_0\left( \NBH{\cdp{p_2}}{p_2}{1}{\tV}\NBH{\cdp{p_2}-1}{p_2}{1}{\pt V^2_1}+ \NBH{\cdp{p_2}}{p_2}{1}{V_2}\NBH{\cdp{p_2}-1}{p_2}{1}{\pt \tV^2}\right),\\
    \int^T_0 \NBH{\cdp{p_2}-1}{p_2}{1}{g^V(U_1)-g^V(U_2)}\lesssim \sqrt{T}\left( \LpNBH{\cdp{p_2}}{p_2}{1}{\tV}{2}\LpNBH{\cdp{p_2}-1}{p_2}{1}{\n V_1}{\infty}\right.\\
    \left.+ \LpNBH{\cdp{p_2}}{p_2}{1}{V_2}{\infty}\LpNBH{\cdp{p_2}-1}{p_2}{1}{\n \tV}{2}\right),
\end{multline}
while using the product law  $\BH{\frac{d}{p_2}}{p_2}{1}\times \BH{\frac{d}{p_2}}{p_2}{1}\mapsto \BH{\frac{d}{p_2}}{p_2}{1} $  one may handle the term $g^{V^2}(U_1)-g^{V^2}(U_2) $ as follows 
\begin{multline}
    \label{est:tg:t:V2}
      \int^T_0 \NBH{\cdp{p_2}-1}{p_2}{1}{g^{V^2}(U_1)-g^{V^2}(U_2)}\lesssim \LpNBH{\cdp{p_2}}{p_2}{1}{\tV}{2}\LpNBH{\cdp{p_2}}{p_2}{1}{\n V_1^2}{2}\\+ \LpNBH{\cdp{p_2}}{p_2}{1}{V_2}{\infty}\LpNBH{\cdp{p_2}}{p_2}{1}{\n \tV^2}{1}.
\end{multline}
   
Let us now give more information about the first inequality of \eqref{est:tg:t:V}. Taking advantage of Lemma \ref{lem:pt:tV:k} (we take $\kappa=\cdp{p_2}-1$) and remembering that $p_2<2d$, one has
\begin{multline}\label{eq:dttildeV}
    \NBH{\cdp{p_2}-1}{p_2}{1}{\pt \tV} \lesssim \NBH{\cdp{p_2}}{p_2}{1}{ \tV}\bigl(1+\NBH{\cdp{p_2}}{p_2}{1}{V_2}+(1+\NBH{\cdp{p_2}}{p_2}{1}{V_1})\NBH{\cdp{p_2}+1}{p_2}{1}{V_2^2}\bigr)\\ + 
    \NBH{\cdp{p_2}+1}{p_2}{1}{ \tV^2}(1+\NBH{\cdp{p_2}}{p_2}{1}{V_1}\bigr)\cdotp
\end{multline}
 Then one may write
\begin{multline}
    \label{1:est:tg:t:}
    \int^T_0 \NBH{\cdp{p_2}}{p_2}{1}{V_2}\NBH{\cdp{p_2}-1}{p_2}{1}{\pt \tV^2}\le \sqrt{T}C\LpNBH{\cdp{p_2}}{p_2}{1}{V_2}{\infty}\LpNBH{\cdp{p_2}}{p_2}{1}{\tV}{2}\\
  +C \LpNBH{\cdp{p_2}}{p_2}{1}{V_2}{\infty}\LpNBH{\cdp{p_2}+1}{p_2}{1}{\tV^2}{1}
   + C\LpNBH{\cdp{p_2}}{p_2}{1}{V_2}{\infty}\LpNBH{\cdp{p_2}+1}{p_2}{1}{V^2_2}{2}\LpNBH{\cdp{p_2}}{p_2}{1}{\tV}{2},
\end{multline}
where $C$ depends on $  \LpNBH{\cdp{p_2}}{p_2}{1}{(V_2,V_1)}{\infty} $.\\
Meanwhile taking $\kappa=\cdp{p_2}-1$ in \eqref{pt:V:est} leads to
\begin{align*}
     \NBH{\cdp{p_2}-1}{p_2}{1}{\pt V^2_1} &\le C \NBH{\cdp{p_2}}{p_2}{1}{ V_1}+ \NBH{\cdp{p_2}+1}{p_2}{1}{ V_1},\;\; C=C(\LpNBH{\cdp{p_2}}{p_2}{1}{V_1})\cdotp
\end{align*}
One may deduce that
\begin{multline}
    \label{2:est:tg:t:}
  \int^T_0 \NBH{\cdp{p_2}}{p_2}{1}{\tV}\NBH{\cdp{p_2}-1}{p_2}{1}{\pt V^2_1}\le C \sqrt{T} \LpNBH{\cdp{p_2}}{p_2}{1}{V_1}{\infty}\LpNBH{\cdp{p_2}}{p_2}{1}{\tV}{2}\\
  + C\LpNBH{\cdp{p_2}+1}{p_2}{1}{V^2_1}{2}\LpNBH{\cdp{p_2}}{p_2}{1}{\tV}{2}.
\end{multline}

Finally, combining \eqref{2:est:tg:t:}, \eqref{1:est:tg:t:}, \eqref{est:tg:t:V2} and  \eqref{est:tg:t:V}, we end with 
    \begin{multline}
\label{est:tW:}
\LpNBH{\frac{d}{p_2}-1}{p_2}{1}{\tV^2}{\infty}+  \LpNBH{\frac{d}{p_2}+1}{p_i}{1}{\tV^2}{1}\le   \NBH{\frac{d}{p_2}-1}{p_i}{1}{\tV^2_0}\\
     +C\sqrt{T}\LpNBH{\cdp{p_2}}{p_2}{1}{\tV}{2}\LpNBH{\cdp{p_2}}{p_2}{1}{ (V_1,V_2)}{\infty}+C  \LpNBH{\cdp{p_2}}{p_2}{1}{\tV}{2}\LpNBH{\cdp{p_2}}{p_2}{1}{\n V_1^2}{2}\\
  + C\LpNBH{\cdp{p_2}}{p_2}{1}{V_2}{\infty}\LpNBH{\cdp{p_2}+1}{p_2}{1}{\tV^2}{1}.
    \end{multline}
with  $C>0$ depending on $ \LpNBH{\cdp{p_2}}{p_2}{1}{(V_2,V_1)}{\infty} $.
\paragraph{\textbf{Step 3: Proof of Lemma \ref{est:uni}}.}
Summing \eqref{est:tW:}, \eqref{est:tz:} together and remarking that the last term of \eqref{est:tz:} may be absorbed by the left-hand side of \eqref{est:tW:}, we get for all $t\ge 0$,
\begin{multline*}
     \LpNBH{\frac{d}{p_2}-1}{p_2}{1}{\tV^2}{\infty}+ \frac{1}{2} \LpNBH{\frac{d}{p_2}+1}{p_i}{1}{\tV^2}{1}+\frac{1}{ 2(1+\LpNBH{\cdp{p_2}}{p_2}{1}{V_2}{\infty})}\LpNBH{\frac{d}{p_2}}{p_2}{1}{\tV^1}{\infty}\\
     \le  C\left(\NBH{\frac{d}{p_2}-1}{p_i}{1}{\tV^2_0}+ \NBH{\frac{d}{p_2}}{p_2}{1}{\tV^1_0}\right)+  C\LpNBH{\cdp{p_2}}{p_2}{1}{V_2}{\infty}\LpNBH{\cdp{p_2}+1}{p_2}{1}{\tV^2}{1}\\
     C \sqrt{T}\LpNBH{\cdp{p_2}}{p_2}{1}{\tV}{2}\LpNBH{\cdp{p_2}}{p_2}{1}{ (V_1,V_2,)}{\infty}+ C \LpNBH{\cdp{p_2}}{p_2}{1}{\tV}{2}\LpNBH{\cdp{p_2}}{p_2}{1}{\n V_1^2}{2},
\end{multline*}
where $C>0$ depends on $  \LpNBH{\cdp{p_2}}{p_2}{1}{(V_2,V_1)}{\infty} $. By interpolation inequality in Besov spaces, one obtains the desired estimate of Lemma \ref{est:uni}. 
\end{proof}

In order to prove the uniqueness part of Theorem \ref{thm:p}, consider two solutions $V_1$ and $V_2$ of \eqref{eq:base:V} in the space $ \mathbb{E}_{p_i},$ corresponding to the same (small)
initial data. We can assume without loss of generality that $V_2$
is the solution constructed before, hence it is small, but $V_1$ need not be small. 
 Before applying Lemma \ref{est:uni}, we have to check that $\tV$ belongs to $\mathbb{F}_T.$
This may be obtained from \eqref{eq:dttildeV}, which implies that $\partial_t\widetilde{V}$ is in
$L^1(0,T;\dot B^{\frac d{p_2}-1}_{p_2,1}),$ and thus $\widetilde{V}$ is in $\mathcal C([0,T];\dot B^{\frac d{p_2}-1}_{p_2,1}).$

Then, Inequality \eqref{est:lem:tV} combined with interpolation in Besov spaces 
allows us to write
 \begin{align*}
      \Vert\tV\Vert_{\mathbb{F}(T)} \le   C\biggl(\LpNBH{\cdp{p_2}}{p_2}{1}{V_2}{\infty}+ \sqrt{T}\LpNBH{\cdp{p_2}}{p_2}{1}{ (V_1,V_2,)}{\infty}+ \LpNBH{\cdp{p_2}+1}{p_2}{1}{ V_1^2}{2}\biggr) \Vert\tV\Vert_{\mathbb{F}(T)}.
  \end{align*}
Next, we note that the term $ \sqrt{T}\LpNBH{\cdp{p_2}}{p_2}{1}{ (V_1,V_2,)}{\infty}  $ tends to $0$ when t goes to $0$ . At the same time by virtue of the Lebesgue dominated convergence theorem, the term $ \LpNBH{\cdp{p_2}+1}{p_2}{1}{ V_1^2}{2} $ also tends to $0$ when t goes to $0.$
Finally, remember that $\LpNBH{\cdp{p_2}}{p_2}{1}{V_2}{\infty}$ is small.  In the end, there exists a positive $T_1$
such that the term on the right-hand side may  be absorbed for all $T\in [0,T_1]$. We conclude that $ \tV $ is identically equal to $0$ in $[0,T_1]$. A bootstrap argument then yields uniqueness on the whole half-line $\mathbb{R}_+$. \qed

\section{Proof of Theorem \ref{thm:p:s} }
\label{sec:proof:thm:p:s}
Here we wish to avoid the condition $p_i \le  \max(4, \frac{d}{d-2})$
which was needed when bounding the low frequencies of some nonlinear terms. To overcome this problem, as in the papers of F. Charve and R. Danchin \cite{DanCha10} or by B. Haspot \cite{Haspot11} where the authors addressed the question of global well-posedness results
for the compressible Navier-Stokes equations, with low and high frequencies of the data belonging to different types of Besov spaces, we need an additional condition in high frequencies on the data $V_0^1$ and $V_0^2$.
We want, then, to follow the same strategy as in the proof of Theorem \ref{thm:p}. This
means that we use the same standard scheme which consists of the construction of approximate solutions satisfying uniform bounds, and results of compactness and uniqueness. Here we want to treat the non-linear term where the condition $ p_i \le  \max(4, \frac{d}{d-2}) $ appears in another
way by using the additional hypothesis  $  V^2_0, V^{1,l}_0\in \BH{s}{2}{r}$ and $V^{1,h}_0\in   \BH{s+1}{2}{r} $
    for some $r\in [1,\infty] $ and $s$ satisfying \eqref{def:s:r:p2}.
We will work with the same space as in the proof of the Theorem  \ref{thm:p}, except that we expect additional regularity on $(V^1,V^2)$ in the space $E$ defined by:
\begin{align*}
    E\defn \left\{ (V^1,V^2): V^2, V^{1,l}\in \widetilde{\mathcal{C}}_b(\mathbb{R}^+\!,\BH{s}{2}{r})\cap L^1(\mathbb{R}^+\!,\BH{s+2}{2}{r}); \;  V^{1,s}\in \widetilde{\mathcal{C}}_b(\mathbb{R}^+\!,\BH{s+1}{2}{r})\cap L^1(\mathbb{R}^+\!,\BH{s+1}{2}{r})\right\}\cdotp
\end{align*}
This will be achieved by using the following key Lemma.
\begin{lem}
\label{lem:prod}
Let $p_2\ge 2,\; 1\le r\le \infty$ and $\theta>-\frac{d}{p_2}>0$. We consider $f,g:\Rd\mapsto \mathbb{R}$ some functions. Then, there exists   $C>0$ such that
\begin{align}
    \label{prod:p2:gen}
    \NBH{\theta}{2}{r}{gf }\le C\left( (\normeinf{f}+\NBH{\frac{d}{p_2}}{p_2}{\infty}{f})\NBH{\theta}{2}{r}{g}+ \NBH{\theta}{2}{r}{f} \normeinf{g}\right)\cdotp
\end{align}
In particular,  
\begin{align}
    \label{prod:p2:gen:s}
    \NBH{\theta}{2}{r}{gf }\le C\bigl(\normeinf{f}+\NBH{\frac{d}{p_2}}{p_2}{\infty}{f}\bigr)\NBH{\theta}{2}{r}{g} \esp{if} \theta<\frac{d}{p_2}\cdotp
\end{align}
Moreover, for any functions $u:\Rd\mapsto \mathbb{R}$ and $f_2:\mathbb R\to\mathbb R$
with $f_2(0)=0,$ there exists a constant $C=C(f_2,\normeinf{u},\theta)$ such that
\begin{align}
    \label{prod:2:p2}
    \NBH{\theta}{2}{r}{gf_2(u)}\le C\left( (\normeinf{u}+\NBH{\frac{d}{p_2}}{p_2}{\infty}{u})\NBH{\theta}{2}{r}{g}+ \NBH{\theta}{2}{r}{u} \normeinf{g}\right),\text{  if  } \theta>0
\end{align}
and 
\begin{align}
    \label{prod:2:p2:0}
    \NBH{\theta}{2}{r}{gf_2(u)}\le C\bigl( \normeinf{u}+\NBH{\frac{d}{p_2}}{p_2}{\infty}{u}\bigr) \NBH{\theta}{2}{r}{g} \esp{if }\theta<\frac{d}{p_2}\cdotp
\end{align}
Finally, if $\theta<0$, then there still exists a constant $C=C(f_2,\normeinf{u},\theta)>0$ such that  the following inequality holds
\begin{align}
    \label{prod:2':p2}
    \NBH{\theta}{2}{r}{gf_2(u)}\le C\left( (\normeinf{u}+\NBH{\frac{d}{p_2}}{p_2}{\infty}{u})\NBH{\theta}{2}{r}{g}+ \NBH{\theta+\sigma}{2}{r}{u} \NBH{\cdp{p_2}-\sigma}{p_2}{\infty}{g}\right)
\end{align}
for any $\sigma>0$ such that $\theta+\sigma>0$.
    
\end{lem}
\begin{proof}
We recall the Bony decomposition
\[ f_1g=\Dot{T}_{f_1} g+\Dot{T}_g f_1+\Dot{R}(f_1,g).\]
    From the continuity of the paraproduct 
    we get
    \begin{align}
    \label{eq:Tf:g:Tg:f:s}
    \NBH{\theta}{2}{r}{\Dot{T}_g f}\le C\normeinf{g} \NBH{\theta}{2}{r}{f}\esp{and} 
         \NBH{\theta}{2}{r}{\Dot{T}_{f} g}\le C\normeinf{f}\NBH{\theta}{2}{r}{g}. 
    \end{align}
    As $\theta+\frac{d}{p_2}>0$ and $p_2\ge 2$ we easily obtain (with $\frac{1}{\lambda}\defn \frac{1}{2}+\frac{1}{p_2}$)
    \begin{align}
    \label{R:f:g:s:est}
        \NBH{\theta}{2}{r}{\Dot{R}(g, f)}\le C \NBH{\theta+\frac{d}{p_2}}{\lambda}{r}{\Dot{R}(g, f)} \le C \NBH{\frac{d}{p_2}}{p_2}{\infty}{f} \NBH{\theta}{2}{r}{g}.
    \end{align}
   Putting together all these estimates gives \eqref{prod:p2:gen}. In order to prove \eqref{prod:p2:gen:s}, we have to bound $ \Dot{T}_g f$ differently. In fact, 
 in the case $ \theta<\frac{d}{p_2}$,  using the continuity of the paraproduct  and Besov embedding  \[  \Dot{T}:\BH{\theta-\frac{d}{p_2}}{\Bar{p}}{r}\times \BH{\frac{d}{p_2}}{p_2}{\infty} \longrightarrow \BH{\theta}{2}{r} \esp{and} \BH{\theta}{2}{r} \hookrightarrow  \BH{\theta-\frac{d}{p_2}}{\Bar{p}}{r} \]
    with $\frac{1}{2}=\frac{1}{p_2}+\frac{1}{\Bar{p}}$ we have 
    \begin{align*}
      \NBH{\theta}{2}{r}{\Dot{T}_g f}\le C  \NBH{\theta}{2}{r}{g}\NBH{\frac{d}{p_2}}{p_2}{\infty}{f}.
    \end{align*}
    Let us now prove \eqref{prod:2:p2}, \eqref{prod:2':p2} and \eqref{prod:2:p2:0}.\\  On the one hand, combining \eqref{prod:p2:gen} and Proposition \ref{propo:comp} leads to  \eqref{prod:2:p2}.\\
    On the other hand, if $\theta< \frac{d}{p_2}$, using \eqref{prod:p2:gen:s} and Proposition \ref{propo:comp} leads to \eqref{prod:2:p2:0}.\\ 
     Finally, if $\theta+\sigma>0$, combining \eqref{R:f:g:s:est}, the second inequality of \eqref{eq:Tf:g:Tg:f:s} together with the Proposition \ref{propo:comp} and the fact 
    \begin{align*}
          \NBH{\theta}{2}{r}{\Dot{T}_g f_2(u)}\le C  \NBH{-\sigma}{\infty}{\infty}{g}\NBH{\theta+\sigma}{2}{r}{f_2(u)}\le C\NBH{-\sigma+\cdp{p_2}}{p_2}{\infty}{g}\NBH{\theta+\sigma}{2}{r}{u},
    \end{align*}
     gives \eqref{prod:2':p2}.    
\end{proof}
 We assume that we are given a smooth (and decaying) solution $V$ of \eqref{eq:base:V}, under assumptions  (\textbf{H1})-(\textbf{H7}),   on $ [0, T ]\times \Rd$ with initial data $V_0$, satisfying
 \begin{align}
     \label{smal:cnd:Vs}
     \TLpNBH{s}{2}{r}{(V^2,V^{1,l})}{\infty}+ \Vert V^{1,h} \Vert_{\widetilde{L}^\infty_T\cap \widetilde{L}^1_T(\BH{s+1}{2}{r})}+ \LpNBH{\frac{d}{p_2}+1}{p_2}{1}{V^1}{\infty}^h+\LpNBH{\frac{d}{p_2}}{p_2}{1}{V^2}{\infty}^h\ll 1.
 \end{align}
Owing to the embedding $\BH{\frac{d}{p}}{p}{1} \hookrightarrow L^\infty$, to \eqref{est:lh:hf} and to \eqref{smal:cnd:Vs}, Inequality \eqref{small:data:inf} also holds.
Let us recall that 
\begin{multline}
    \label{recal:math:X:s}
        X_{p_i}(T)=\TLpNBH{s}{2}{r}{V}{\infty}^{l}+ \TLpNBH{s+2}{2}{r}{V}{1}^{l}+ \Vert(W,V^2)\Vert_{\Tilde{L}^\infty_T(\BH{s}{2}{r}\cap\BH{\frac{d}{p_i}}{p_i}{1})}^h+ \Vert(W,V^2)\Vert_{\Tilde{L}^1_T(\BH{s+2}{2}{r}\cap\BH{\frac{d}{p_i}+2}{p_i}{1})}^h\\
        + \Vert V^1 \Vert_{\Tilde{L}^\infty_T(\BH{s+1}{2}{r}\cap\BH{\frac{d}{p_i}+1}{p_i}{1})}^h+  \Vert V^1 \Vert_{\Tilde{L}^1_T(\BH{s+1}{2}{r}\cap\BH{\frac{d}{p_i}+1}{p_i}{1})}^h+  \Vert V^1 \Vert_{{L}^1_T(\BH{s+1}{2}{r}\cap\BH{\frac{d}{p_i}+1}{p_i}{1})}^h+ \Vert \pt V \Vert_{{L}^1_T(\BH{\frac{d}{p_i}}{p_i}{1})}  .
    \end{multline}

We apply exactly the same proof as for Theorem \ref{thm:p}. However, we must complete the uniform bounds by showing that $(V^1,V^2)$ is uniformly bounded in $ E $. To do that, Lemma \ref{lem:prod} will be very useful. Moreover,
we must treat the term in low frequencies differently where the condition $  p_i \le  \max(4, \frac{d}{d-2}) $ 
appears by using the fact that $(V^1,V^2)$ has in addition some regularities in spaces based on energy space.

We recall that $(c_k)_{k\in \mathbb{Z}}$ designates various nonnegative sequences such that $\lVert(c_k)\rVert_{l^1(\mathbb{Z})}\le 1 $.
Throughout this section, $(k_j)_{j\in \mathbb{Z}}$ denotes nonnegative sequences that verify $\lVert(k_j)\rVert_{l^r(\mathbb{Z})}\le 1 $. For simplicity we assume in this section that $f\equiv 0$.
\medbreak

We have the following result.
\begin{propo}
    \label{est:X:s:pro}
   For all $i=1,2$, the following inequality holds
    \begin{equation}
    \label{pro:est:X:s}
     X_{p_i}\le  X_{{p_i},0}     +C(2^{-N_{0}}+ X_{p_2}+X_{p_2}X_{p_2})(X_{p_i}+X_{p_2}+X_{p_i}X_{p_2}).
    \end{equation}
\end{propo}
The proof of Proposition \ref{est:X:s:pro} will be separated into several steps.
\smallbreak\
\paragraph{\textbf{$\bullet$ Step 1: Estimate in L$^2$-type Besov spaces in low frequencies:}}

\begin{lem}
For all $p_2\ge 2$ and $ s>-\min(1,\frac{d}{p_2})$. There exist two positive constants $c$ and $C$ such that, for all $t\in [0,T]$, the following inequality holds
    \begin{multline}
    \label{s:est:ly:l}
     \TLpNBH{s}{2}{r}{V}{\infty}^{l}+ c\TLpNBH{s+2}{2}{r}{V}{1}^{l} \le 2\NBH{s}{2}{r}{V_0}^{l}+ C\Bigl(\Vert \n V \Vert_{L^1_T(L^\infty)} \TLpNBH{s}{2}{r}{V}{\infty}^{l}+ \TLpNBH{s}{2}{r}{F^t}{1}\\
    +\Vert V \Vert_{\widetilde{L}^{2}_T(\BH{\frac{d}{p_2}}{p_2}{\infty})\cap L^{2}_T(L^\infty)}\TLpNBH{s}{2}{r}{\n V}{2}+ \TLpNBH{s+1}{2}{r}{V}{2}\Vert \n V\Vert_{\widetilde{L}^{2}_T(\BH{\cdp{p_2}-1}{p_2}{\infty})}\\
   + \Vert V \Vert_{\widetilde{L}^{\infty}_T(\BH{\frac{d}{p_2}}{p_2}{\infty})\cap L^{\infty}_T(L^\infty)} \TLpNBH{s}{2}{r}{ \n^2 V^2}{1}+\TLpNBH{s+1}{2}{r}{V}{2} \Vert \n V^2\Vert_{L^{2}_T(L^\infty)}\Bigr),       
    \end{multline}
with $F^t$ defined in \eqref{F_t}.
\end{lem}
\begin{proof}
Starting from \eqref{01:est:ly:l}, using Condition (SK) combined with \eqref{eq:Lj:Vj} provides the existence of a positive constant $c_0$ such that

 \begin{equation}
    \label{1:est:ly:l:s}
    \frac{d}{dt}\Lj+  2^{2j}c_0\Lj \lesssim  \normeinf{\n V}\sqrt{\Lj}\normed{V_j}
+ \sqrt{\Lj}\normed{\Dj F}.
\end{equation}
Then, taking $ X=\Lj$, $p=2$ in the Lemma \ref{lem_der_int} guarantees that for all $t\in [0,T]$
\begin{equation}
    \label{2:est:ly:l:s}
    \sup_{0\le t\le T}\sqrt{\Lj(t)}+  2^{2j}c_0\int^t_0\sqrt{\Lj} \lesssim \sqrt{\Lj(0)}+ \Vert \n V \Vert_{L^1_T(L^\infty)} \Vert V_j \Vert_{L^\infty_T(L^2)}\\
    + \Vert \Dj F \Vert_{L^1_T(L^2)}.
    \end{equation}

Now, we need to deal with the nonlinear terms $ F^V$ and $F^{V^2}$ defined in \eqref{F_t}. 
From the estimate \eqref{prod:2':p2} of Lemma \ref{lem:prod} (we choose  $\theta=s$, $\sigma=1$) and H\"{o}lder inequality in time, we get 
\begin{multline}
    \label{est:lh:F}
   \Vert \Dj F^V \Vert_{L^1_T(L^2)}\\ \le
     Ck_j2^{-js}\left(\Vert V \Vert_{\widetilde{L}^{2}_T(\BH{\frac{d}{p_2}}{p_2}{\infty})\cap L^{2}_T(L^\infty)} \TLpNBH{s}{2}{r}{\n V}{2}+ \TLpNBH{s+1}{2}{r}{V}{2} \Vert \n V \Vert_{\widetilde{L}^{2}_T(\BH{\cdp{p_2}-1}{p_2}{\infty})}\right),
\end{multline}
while taking  $\theta=s+1$ in \eqref{prod:2:p2} produces
\begin{multline}
    \label{est:lh:F2}
    \Vert \Dj F^{V^2} \Vert_{L^1_T(L^2)} \le Ck_j2^{-js}\Bigl(\Vert V \Vert_{\widetilde{L}^{\infty}_T(\BH{\frac{d}{p_2}}{p_2}{\infty})\cap L^{\infty}(L^\infty)} \TLpNBH{s}{2}{r}{\n^2 V^2 }{1}\\
    + \TLpNBH{s+1}{2}{r}{V}{2} \Vert \n V^2\Vert_{L^{2}_T(L^\infty)}\Bigr)\cdotp
    \end{multline}
Plugging  this estimate into \eqref{2:est:ly:l:s}, and using \eqref{eq:Lj:Vj} combined with interpolation inequality in time  Lebesgue space yields for all $j\le N_{0}+1$
and $\gamma\in [1,\infty]$:
\begin{multline}
    \label{:3:s:est:ly:l}
    (c_0 2^{2j})^\frac{1}{\gamma}\Vert V_j \Vert_{L^\gamma_T(L^2)} \lesssim \normed{V_{0,j}}+ \Vert \n V \Vert_{L^1_T(L^\infty)} \Vert V_j \Vert_{L^\infty_T(L^2)}+  \Vert \Dj F^{t} \Vert_{L^1_T(L^2)} \\
    + k_j2^{-js}C\biggl( \Vert V \Vert_{\widetilde{L}^{2}_T(\BH{\frac{d}{p_2}}{p_2}{\infty})\cap L^{2}_T(L^\infty)}\TLpNBH{s}{2}{r}{\n V}{2}+ \TLpNBH{s+1}{2}{r}{V}{2}\Vert\n V\Vert_{\widetilde{L}^{2}_T(\BH{\cdp{p_2}-1}{p_2}{\infty})}\\
   +\Vert V \Vert_{\widetilde{L}^{\infty}_T(\BH{\frac{d}{p_2}}{p_2}{\infty})\cap L^{\infty}(L^\infty)} \TLpNBH{s}{2}{r}{ \n^2 V^2)}{1}+ \TLpNBH{s+1}{2}{r}{V}{2} \Vert \n V^2\Vert_{L^{2}_T(L^\infty)}\biggr)\cdotp
\end{multline}

We get \eqref{s:est:ly:l} by multiplying \eqref{:3:s:est:ly:l} by $ 2^{js} $ and taking the $l^r$ norm ($j\le N_{0}+1$) of the previous inequality.
\end{proof}
\medbreak
The next step is to estimate the parabolic mode $W,$ and $V^1$ in Besov spaces based on $ L^2 $. 
\smallbreak
\paragraph{\textbf{ $\bullet$ Step 2: Estimates in L$^2$-type Besov space in high frequencies:}}
\begin{lem}[Estimate for $W$]
    \label{lem:est:W:s}
   Let $p_2\ge 2$ and $ s>-\min(1,{d}/{p_2})$. There exist two positive constants $c$ and $C$ such that, for all $t\in [0,T]$, the following inequality holds
\begin{multline}
    \label{est:W:s}
    \TLpNBH{s}{2}{r}{W}{\infty}^h+c   \TLpNBH{s+2}{2}{r}{W}{1}^h \le  2 \NBH{s}{2}{r}{W_0}^h+ 2^{-N_{0}}C   \TLpNBH{s}{2}{r}{\pt V}{1}+  \TLpNBH{s+2}{2}{r}{g^t}{1} \\
        +  C\biggl(\Vert V \Vert_{\widetilde{L}^{2}_T(\BH{\frac{d}{p_2}}{p_2}{\infty})\cap L^{2}_T(L^\infty)}\TLpNBH{s}{2}{r}{\n V}{2}+ \TLpNBH{s+1}{2}{r}{V}{2}\Vert \n V \Vert_{\widetilde{L}^{2}_T(\BH{\cdp{p_2}-1}{p_2}{\infty})}\\
   + \Vert V \Vert_{\widetilde{L}^{\infty}_T(\BH{\frac{d}{p_2}}{p_2}{\infty})\cap L^{\infty}(L^\infty)} \TLpNBH{s}{2}{r}{\n^2 V^2}{1}+ \TLpNBH{s+1}{2}{r}{V}{2} \Vert \n V^2\Vert_{L^{2}_T(L^\infty)}\biggr),
\end{multline}
   with $g^t$ defined in \eqref{def:g}.
\end{lem}
\begin{proof}
Taking $p_i=2$  in  \eqref{est:p:W:1} leads to 
    \begin{align}
        \label{est:W:L2}
          \frac{d}{dt}\intd \overline{S}^0_{22} W_j\cdot W_j+ c_0 2^{2j} \normed{W_j}^2\le C \normed{W_j}\normed{ \Dj g},
    \end{align}
    where $g=g^V+g^{V^2}+g^t+\Bar{g}^V$ has been defined in \eqref{def:g}.
    
    Using  Lemma \ref{lem_der_int} with $p=2 $ and $X=\overline{S}^0_{22} W_j\cdot W_j,$ and remembering that the matrix $ \overline{S}^0_{22} $ is symmetric positive definite yields
\begin{align}
    \label{1:s:est:W} \Vert  W_j\Vert_{L^\infty_T(L^2)}+2^{2j}c_0   \Vert  W_j\Vert_{L^1_T(L^2)}\le C\normed{W_j(0)}+  C\Vert  \Dj g \Vert_{L^\infty_T(L^2)}.
\end{align}
The terms $ \Dj g^V$ and $\Dj  g^{V^2}$ can be bounded in $L^1_T(L^2)$ exactly as 
for proving \eqref{s:est:ly:l}. Hence, from \eqref{est:lh:F} and \eqref{est:lh:F2} it holds:
\begin{align}
    \label{s:est:lh:g}
    \begin{split}
      \blacktriangleright&\Vert \Dj g^V \Vert_{L^1_T(L^2)} \\&\le
     Ck_j2^{-js}\left(\Vert V \Vert_{\widetilde{L}^{2}_T(\BH{\frac{d}{p_2}}{p_2}{\infty})\cap L^{2}_T(L^\infty)} \TLpNBH{s}{2}{r}{\n V}{2}+ \TLpNBH{s+1}{2}{r}{V}{2} \Vert \n V \Vert_{\widetilde{L}^{2}_T(\BH{\cdp{p_2}-1}{p_2}{\infty})}\right),\\
      \blacktriangleright &\Vert \Dj g^{V^2} \Vert_{L^1_T(L^2)} \\&\le 
      Ck_j2^{-js}\left(\Vert V \Vert_{\widetilde{L}^{\infty}_T(\BH{\frac{d}{p_2}}{p_2}{\infty})\cap L^{\infty}(L^\infty)} \TLpNBH{s}{2}{r}{\n^2 V^2 }{1}+ \TLpNBH{s+1}{2}{r}{V}{2} \Vert \n V^2\Vert_{L^{2}_T(L^\infty})\right)\cdotp
     \end{split}
\end{align}
 Remembering the definition of $\Bar{g}^V$ (see \eqref{def:g}), using \eqref{bern:ine:M}  and bearing in mind that the homogeneous operators $ |D|^{-1}\mathcal{Z}^{-1}(D)\Ss_{21}(D)$ and $|D|^{-1}\mathcal{Z}^{-1}(D)\Ss_{22}(D)$ are of order $-1$, we get for all $j\ge N_{0}+2$.
 \begin{align*}
     &\Vert\Dj(\Bar{g}^V) \Vert_{L^1_T(L^2)}   \le C 2^{-j} \Vert\pt V \Vert_{L^1_T(L^2)}  .
 \end{align*}

We insert all this information into  \eqref{est:W:L2}, we then use interpolation inequalities in time  Lebesgue spaces (up to slight modification of $c_0$) and  we get for 
all $j\in \mathbb{Z}$ and $\gamma\in [1,\infty]$:
\begin{multline}
   \label{2:s:est:W}
  ( c_0 2^{2j})^\frac{1}{\gamma}\Vert  W_j\Vert_{L^\gamma_T(L^2)}\le   \normed{W_j(0)}+ C2^{-(s+1)j}   k_j\TLpNBH{s}{2}{r}{\pt V}{1}^h+\Vert  \Dj g^t\Vert_{L^\gamma_T(L^2)} \\
         + Ck_j2^{-js}\biggl( \Vert V \Vert_{\widetilde{L}^{2}_T(\BH{\frac{d}{p_2}}{p_2}{\infty})\cap L^{2}_T(L^\infty)}\TLpNBH{s}{2}{r}{\n V}{2}+ \TLpNBH{s+1}{2}{r}{V}{2}\Vert \n V \Vert_{\widetilde{L}^{2}_T(\BH{\cdp{p_2}-1}{p_2}{\infty})}\\
   + \Vert V \Vert_{\widetilde{L}^{\infty}_T(\BH{\frac{d}{p_2}}{p_2}{\infty})\cap L^{\infty}(L^\infty)} \TLpNBH{s}{2}{r}{ \n^2 V^2}{1}+\TLpNBH{s+1}{2}{r}{V}{2} \Vert \n V^2\Vert_{L^{2}_T(L^\infty)} \biggr)\cdotp
\end{multline}

To get \eqref{est:W:s}, we multiply \eqref{2:s:est:W} by $2^{js}$, then take its $l^r$ norm
on $j\geq N_0+2.$
\end{proof}
The following result concerns the estimates on $V^1$ in high frequencies in Besov spaces based on $L^2$.

\begin{lem}[Estimate for $V^1$]
    \label{lem:est:V1:s}
   Let $p_2\ge 2$ and $ s>-\min(1,\frac{d}{p_2})$. We have for all $t\in [0,T]$:
\begin{multline}
    \label{est:V1:s}
     \TLpNBH{s+1}{2}{r}{V^1}{\infty}^h+c_0   \TLpNBH{s+1}{2}{r}{V^1}{1}^h \le   \NBH{s+1}{2}{r}{V^1_0}^h+ 2^{-N_{0}}   \TLpNBH{s+2}{2}{r}{ V^2}{1}^h+ \TLpNBH{s+2}{2}{r}{ W}{1}^h\\
    Ck_j2^{-j(s+1)}\biggl(  \TLpNBH{s+1}{2}{r}{V}{2} \Vert  \n V^1\Vert_{L^{2}_T(L^\infty)} +\lVert  \n V\rVert_{\Tilde{L}^{2}_T(\BH{\cdp{p_2}}{p_2}{\infty})\cap L^{2}_T(L^\infty)} \TLpNBH{s+1}{2}{r}{V^1}{2}\\
     +  \lVert  V\rVert_{\Tilde{L}^{\infty}_T(\BH{\cdp{p_2}}{p_2}{\infty})\cap L^{\infty}_T(L^\infty)}\TLpNBH{s+1}{2}{r}{ \n V^2}{1}  + \TLpNBH{s+1}{2}{r}{V}{\infty}\Vert   \n V^2\Vert_{L^{1}_T(\BH{\cdp{p_2}}{p_2}{\infty})}\biggr)\cdotp
\end{multline}
\end{lem}
\begin{proof}
    The proof of Lemma \ref{lem:est:V1:s} follows the lines of Lemma \ref{est:V1:hf:lem}. The difference lies on estimating the source term in  $L^1(\BH{s+1}{2}{r})$ instead  of $ L^1(\BH{\cdp{p_i}+1}{p_i}{1})$, with $i=1,2$. 
    Starting from \eqref{est:V1:ini} (where we take $p_i=2$) and using Lemma \ref{lem_der_int} with $p=2$ yields for all $j\in \mathbb{Z}$:
     \begin{equation}
     \label{est:z:s:hf}
    \Vert  z_j\Vert_{L^\infty_T(L^2)}+c_0   \Vert  z_j\Vert_{L^1_T(L^2)}\le \normed{z_j(0)}+ C \Vert  ( \Dj h ,R_j)\Vert_{L^1_T(L^2)}\\
       +C\int^T_0 \normeinf{\n H(U)}\normess{z_j}{{2}},
       \end{equation}
where we recall $h= h^W+h^{V^2}_0+h^{V^2}_1  $ (see \eqref{def:ht:V:V2}) and 
\begin{align}
     \label{def:Rj:recal}
    R_j=\suma \left[\Dj,H^\alpha(U)- H^\alpha(\overline{U})\right]\pal z.
 \end{align}
 Choosing $p=2, \Tilde{p}=p_2$ in \eqref{comm:est:1} and applying the outcome result to $R_j$, since $s+1>-\frac{d}{p_2}$ and $s>-1$, we have
 \begin{multline}
 \label{est:Rj:s:hf}
       \Vert  R_j \Vert_{L^1_T(L^2)}\\
       \le C2^{-(s+1)j}k_j\left( \lVert \n V\rVert_{\Tilde{L}^{2}_T(\BH{\cdp{p_2}}{p_2}{\infty})\cap L^{2}(L^\infty)} \TLpNBH{s+1}{2}{r}{z}{2}+  \lVert\n z \rVert_{L^{2}(L^\infty)}  \TLpNBH{s+1}{2}{r}{ V}{2}\right)\cdotp
\end{multline}
From Bernstein's inequality we have
\begin{align}
\label{est:hz:s:hf}
      \Vert  (\Dj h^W, \Dj h^{V^2}_0) \Vert_{L^1_T(L^2)}\le C  \Vert  (2^j \Dj W, \Dj V^2) \Vert_{L^1_T(L^2)}.
\end{align}
The term $ h^{V^2}_{1}$ can be treated using Lemma \ref{lem:prod}  combined with Proposition \ref{propo:comp} . As we are working under the assumption that $s>-\frac{d}{p_2}$, applying \eqref{prod:2':p2} yields
\begin{multline}
\label{est:hW:s:hf}
       \Vert  \Dj h^{V^2}_{1} \Vert_{L^1_T(L^2)}\le Ck_j2^{-j(s+1)}\biggl(\lVert  V\rVert_{\Tilde{L}^{\infty}_T(\BH{\cdp{p_2}}{p_2}{\infty})\cap L^{\infty}(L^\infty)} \TLpNBH{s+1}{2}{r}{\n V^2}{1}\\
    + \TLpNBH{s+1}{2}{r}{V}{\infty}\lVert \n V^2\rVert_{\Tilde{L}^{1}_T(\BH{\cdp{p_2}}{p_2}{\infty})}\biggr)\cdotp
\end{multline}
Now, inserting \eqref{est:hz:s:hf}, \eqref{est:hW:s:hf} and \eqref{est:Rj:s:hf}   into \eqref{est:z:s:hf} and using interpolation inequalities in time Lebesgue spaces yields for all $j\ge N_{0}+2$ and $\gamma\in [1,\infty]$:
  \begin{multline}
   \label{2:s:est:v1}
   ( c_0 2^{2j})^\frac{1}{\gamma}\Vert  z\Vert_{L^\gamma_T(L^2)}\le   \normed{z_j(0)}+C\int^T_0 \normeinf{\n V}\normess{z_j}{{2}}+ C  \Vert  (\Dj V^2,2^j\Dj W) \Vert_{L^1_T(L^2)}\\
    + Ck_j2^{-j(s+1)}\biggl( \TLpNBH{s+1}{2}{r}{V}{2} \Vert   \n z\Vert_{L^{2}_T(L^\infty)}+\lVert  \n V\rVert_{\Tilde{L}^{2}_T(\BH{\cdp{p_2}}{p_2}{\infty})\cap L^{2}_T(L^\infty)} \TLpNBH{s+1}{2}{r}{z}{2}\\
     + \lVert  V\rVert_{\Tilde{L}^{\infty}_T(\BH{\cdp{p_2}}{p_2}{\infty})\cap L^{\infty}_T(L^\infty)}\TLpNBH{s+1}{2}{r}{ \n V^2}{1}  + \TLpNBH{s+1}{2}{r}{V}{\infty}\Vert   \n V^2 \Vert_{L^{1}_T(\BH{\cdp{p_2}}{p_2}{\infty})}\biggr)\cdotp
  \end{multline}

We get the desired estimate after summing \eqref{2:s:est:v1} over $j\ge N_{0}+2$ and remarking that any Lebesgue norm of $\Dj z$ is equivalent to $ \Dj V^1$ (hence $\NBH{\alpha}{p}{r}{z}\simeq \NBH{\alpha}{p}{r}{V^1}$).
\end{proof}

\paragraph{\textbf{Step 3: Final estimates and proof of Proposition \ref{est:X:s:pro}.}}
The goal of this step is to prove Estimate \eqref{pro:est:X:s}. We recall that
\begin{multline}
    \label{def:math:X:s}
        X_{p_i}\defn\TLpNBH{s}{2}{r}{V}{\infty}^{l}+ \TLpNBH{s+2}{2}{r}{V}{1}^{l}+ \Vert(W,V^2)\Vert_{\Tilde{L}^\infty(\BH{s}{2}{r}\cap\BH{\frac{d}{p_i}}{p_i}{1})}^h+ \Vert(W,V^2)\Vert_{\Tilde{L}^1(\BH{s+2}{2}{r}\cap\BH{\frac{d}{p_i}+2}{p_i}{1})}^h\\
        + \Vert V^1 \Vert_{\Tilde{L}^\infty(\BH{s+1}{2}{r}\cap\BH{\frac{d}{p_i}+1}{p_i}{1})}^h+  \Vert V^1 \Vert_{\Tilde{L}^1(\BH{s+1}{2}{r}\cap\BH{\frac{d}{p_i}+1}{p_i}{1})}^h.
    \end{multline}

Also, let us recall that from \eqref{est:W:V^1:1}  we have proved that
\begin{multline}
\label{W:V^1=f(VV)}
  \LpNBH{\frac{d}{p_i}+1}{p_i}{1}{V^1}{\infty}^h+  \LpNBH{\frac{d}{p_i}+1}{p_i}{1}{V^1}{1}^h +  \LpNBH{\frac{d}{p_i}}{p_i}{1}{(W,V^2)}{\infty}^h+  \LpNBH{\frac{d}{p_i}+2}{p_i}{1}{(W,V^2)}{1}^h\\\lesssim \mathbb{V}_{{p_i},0} 
   +(2^{-N_{0}}+ \mathbb{V}_{p_2}+\mathbb{V}_{p_2}^2)\mathbb{V}_{p_i}+2^{-N_{0}}\LpNBH{\frac{d}{p_i}+2}{p_i}{1}{V^2}{1}^h.
\end{multline}
Choosing $N$ large enough, the last term of \eqref{W:V^1=f(VV)} may be absorbed by the left-hand side. Moreover, from the fact that the high frequencies of $  \mathbb{V}_{p_i}$ are bounded by $ X_{p_i}$ and that  
\begin{align}
    \label{cond:s=f(r)}
    -\min(1,\frac{d}{p_2})<s<\frac{d}{2}-1\; \text{and}\; r>1  \esp{or} s=\frac{d}{2}-1 \;\; \text{and}\;\; r=1
\end{align}
which provides, for all $\ro\in [1,+\infty]$ and $\eta\in \mathbb{R}$
\begin{align}
    \TLpNBH{\frac{d}{2}-1+\eta}{2}{1}{V}{\ro}^{l}\lesssim \TLpNBH{s+\eta}{2}{r}{V}{\ro}^{l},
\end{align}
we get 
\begin{align}
    \label{VV<X}
   \mathbb{V}_{p_i}\lesssim  X_{p_i}.
\end{align}
From this, we deduce that 
\begin{multline}
    \label{hf<X_i}
    \LpNBH{\frac{d}{p_i}+1}{p_i}{1}{V^1}{\infty}^h+  \LpNBH{\frac{d}{p_i}+1}{p_i}{1}{V^1}{1}^h +  \LpNBH{\frac{d}{p_i}}{p_i}{1}{(W,V^2)}{\infty}^h+  \LpNBH{\frac{d}{p_i}+2}{p_i}{1}{(W,V^2)}{1}^h\\\lesssim 
    X_{{p_i},0}    +(2^{-N_{0}}+ X_{p_2}+X_{p_2}^2)X_{p_i}.
\end{multline}
The rest of proof consists in proving that the norms that
appear on the right-hand side of Inequalities \eqref{s:est:ly:l}, \eqref{est:W:s} and  \eqref{est:V1:s}  are dominated by $X_{p_1}$ or by $X_{p_2}$, which is possible using Bernstein inequality, interpolation inequality and the hypothesis
 \begin{align}
     \label{pi:s:r:d}
    p_1\le p_2;\;\;  s<\frac{d}{2}-1 \;\; \text{and}\;\; r>1 \esp{or} s=\frac{d}{2}-1 \;\; \text{and}\;\; r=1.
 \end{align}
 Let us start by proving the result for the term $ \TLpNBH{s}{2}{r}{(g^t,F^t)}{1} $.
\begin{lem}
    \label{lem:est:pt:V:s}
Let $k:\mathbb{R}^n\mapsto \mathbb{R}$ be  a regular function satisfying $k(\overline{U})=0$.
It holds
    \begin{align}
        \label{est:ptV:s:h}
        \TLpNBH{s}{2}{r}{k(U)\pt V}{1}^h\lesssim X_{p_2}+(X_{p_2})^2.
    \end{align}
\end{lem}
\begin{proof}
  The proof of Lemma \ref{lem:est:pt:V:s} will be based on Lemma \ref{lem:prod}.  Let us remark that, since $k(\overline{U}) =0$, then  $k(U)\pt V$ satisfies the following equation
  \begin{multline}
      \label{eq:pt V:s}
  k(U) \pt V = \suma \left( k(U)(S^0(U))^{-1} S^\alpha(U)- k(\overline{U})(\overline{S^0})^{-1}\overline{S^\alpha}\right) \pal V \\
       + \sumab \left( k(U)(S^0(U))^{-1} Y^{\alpha\beta}(U)-k(\overline{U})(\overline{S^0})^{-1} \overline{Y}^{\alpha\beta}\right)\partial^2_{\alpha\beta} V\\+  k(U)(S^0(U))^{-1} \sumab \pal (Y^{\alpha\beta}(V))\pbe V.
         \end{multline}
 First, using \eqref{prod:2:p2} and \eqref{smal:cnd:Vs} combined with 
 Proposition \ref{propo:comp} and the form of each matrix $Y^{\alpha\beta}(U)$ we get for  $s>-\min(1,\cdp{p_2})$
  \begin{multline*}
       \TLpNBH{s}{2}{r}{  k(U)(S^0(U))^{-1}\sumab \pal (Y^{\alpha\beta}(V))\pbe V} {1} \\
      \lesssim \lVert \n V\rVert_{\widetilde{L}^{2}_T(\BH{\cdp{p_2}}{p_2}{\infty})\cap L^{2}(L^\infty)}\TLpNBH{s}{2}{r}{\n V^2}{2}
      + \TLpNBH{s+1}{2}{r}{V}{2}\Vert \n V^2 \Vert_{L^{2}_T(L^\infty)}
  \end{multline*}
 while using \eqref{prod:2':p2} with $\theta=s, \sigma=1$ we have for all $\alpha,\beta=1,\cdots, d$
  \begin{multline*}
      \TLpNBH{s}{2}{r}{(k(U)(S^0(U))^{-1} Y^{\alpha\beta}(U)-k(\overline{U})(\overline{S^0})^{-1} \overline{Y}^{\alpha\beta})\partial^2_{\alpha\beta}} {1}\\\lesssim 
     \TLpNBH{s}{2}{r}{ \n^2 V^2} {1}^h+  \lVert  V\rVert_{\Tilde{L}^{\infty}_T(\BH{\cdp{p_2}}{p_2}{\infty})\cap L_T^{\infty}(L^\infty)}\TLpNBH{s}{2}{r}{\n^2 V^2}{1}+ \TLpNBH{s+1}{2}{r}{V}{\infty}\Vert \n^2 V^2 \Vert_{L^{1}_T(\BH{\cdp{p_2}-1}{p_2}{\infty})}.
  \end{multline*}
  Finally, taking again $\theta=s, \sigma=1$  in \eqref{prod:2':p2} of Lemma \ref{lem:prod} for the first   middle term on the right-hand side of \eqref{eq:pt V:s},  we get for all $s>-\min(\frac{d}{p_2},1)$,
  \begin{multline}
  \label{est:ptV:hf:s:fin}
    \TLpNBH{s}{2}{r}{k(U)\pt V}{1}^h\lesssim 
      \lVert \n V\rVert_{\widetilde{L}^{2}_T(\BH{\cdp{p_2}}{p_2}{\infty})\cap L^{2}_T(L^\infty)}\TLpNBH{s}{2}{r}{\n V^2}{2}+ \TLpNBH{s+1}{2}{r}{V}{2}\Vert \n V^2 \Vert_{L^{2}_T(L^\infty)}\\
       +\lVert  V\rVert_{\Tilde{L}^{2}_T(\BH{\cdp{p_2}}{p_2}{\infty})\cap L^{\ro_1}(L^\infty)}\TLpNBH{s}{2}{r}{\n V}{2}  + \TLpNBH{s+1}{2}{r}{V}{2}\Vert \n V \Vert_{\widetilde{L}^{2}_T(\BH{\cdp{p_2}-1}{p_2}{\infty})}\\
       +\lVert  V\rVert_{\Tilde{L}^{\infty}_T(\BH{\cdp{p_2}}{p_2}{\infty})\cap L^{\infty}(L^\infty)}\TLpNBH{s}{2}{r}{\n^2 V^2}{1}+ \TLpNBH{s+1}{2}{r}{V}{\infty}\Vert \n^2 V^2 \Vert_{L^{1}_T(\BH{\cdp{p_2}-1}{p_2}{\infty})}.
  \end{multline}
   The first two terms on the right-hand side of the above are bounded by $(X_{p_2})^2$.\\
  Note that, the use of basic Besov embeddings \eqref{link:lf:hf} and \eqref{pi:s:r:d} provides
\begin{align}
    \label{es:inf:X}
    \normeinf{f}\lesssim \NBH{\frac{d}{2}}{2}{1}{f^{l}}+\NBH{\frac{d}{p_2}}{p_2}{1}{f^h}\lesssim \NBH{s+1}{2}{r}{f}^{l}+\NBH{\frac{d}{p_2}}{p_2}{1}{f}. 
\end{align}
From \eqref{est:lh:hf}, \eqref{es:inf:X} and \eqref{link:CL:Besov:space:chap2} that shows the relationship between the Besov space $L^\ro_T(\BH{s}{p}{r})$ and Chemin-Lerner-Besov space $ \widetilde{L}^\ro(\BH{s}{p}{r})$,
we have
\begin{align}\label{est:st:s:1}
   \lVert  V\rVert_{\Tilde{L}^{2}_T(\BH{\cdp{p_2}}{p_2}{\infty})\cap L^{2}(L^\infty)}\TLpNBH{s}{2}{r}{\n V}{2} &\lesssim   \left(\TLpNBH{s+1}{2}{r}{V}{2}^{l}+\lVert V\rVert^h_{L^2_T(\BH{\cdp{p_2}}{p_2}{\infty}\cap L^\infty)}\right) \TLpNBH{s}{2}{r}{\n V}{2}^{l}  \nonumber\\\nonumber
   &+  \left(\TLpNBH{s}{2}{r}{V}{\infty}^{l}+ \lVert V\rVert^h_{L^\infty_T(\BH{\cdp{p_2}}{p_2}{\infty}\cap L^\infty)}\right) \LpNBH{s}{2}{r}{\n V}{1}^h\\
   &\lesssim  (X_{p_2})^2,
\end{align}
\begin{align}\label{est:st:s:2}
   \Vert V \Vert_{\widetilde{L}^{2}_T(\BH{s+1}{2}{r})}\Vert \n V\Vert_{\widetilde{L}^{2}_T(\BH{\cdp{p_2}-1}{p_2}{\infty})} &\lesssim   \TLpNBH{s+1}{2}{r}{V}{2}(\TLpNBH{\cdp{p_2}}{p_2}{\infty}{V}{2}^{l}+ \LpNBH{\cdp{p_2}}{p_2}{1}{V}{2}^h)
   \nonumber\\\nonumber
     &\lesssim   \TLpNBH{s+1}{2}{r}{V}{2}(\TLpNBH{s}{2}{r}{V}{2}^{l}+ \LpNBH{\cdp{p_2}}{p_2}{1}{V}{2}^h)\\
   &\lesssim  (X_{p_2})^2.
   \end{align}

Similarly, we can prove that the remaining terms of \eqref{est:ptV:hf:s:fin} are bounded by $(X_{p_2})^2$, which finishes the proof of Lemma \ref{lem:est:pt:V:s}.
\end{proof}
Lemma \ref{lem:est:pt:V:s} shows that the terms $F^t$ and $g^t$ are dominated by $X_{p_2}$ in the space $\widetilde{L}^1_T(\BH{s}{2}{r})$. Arguing as in Lemma \ref{lem:est:pt:V:s} and based on inequality \eqref{est:lh:hf}, estimates \eqref{es:inf:X} and \eqref{link:CL:Besov:space:chap2} one can prove that the right-hand side of inequalities \eqref{est:V1:s}, \eqref{est:W:s} and \eqref{s:est:ly:l} are dominated by $X_{p_1}$ or by $X_{p_2}$. The proof is similar to what we did to get \eqref{est:st:s:1} and \eqref{est:st:s:2}. 
This completes the proof of Proposition \ref{est:X:s:pro}.\qed
\medbreak
The rest of the proof of existence of Theorem \ref{thm:p:s} just consists in following the steps from $1$ to $3$ of section \ref{proof:Thms:p:and:cri:sec}. It is worthwhile to emphasize that for mollifying the initial data, one may proceed differently. In fact, as the high frequencies of the data belong to Besov spaces constructed on $L^2$, it suffices to smooth out 
the data as follows: 
\begin{align*}
     V_{0,n}= \Dot{S}_n V_0 \esp{for all } n\ge 1.
\end{align*}
It is easy to show that the sequence $(V_{0,n})_n$ converges to $V_0$ in the space 
of the initial data. Moreover, one can prove without difficulty that $X_{i,n}\le CX_{{p_i},0}$ for all $i=1,2$ and $n\ge 1$, where $X_{i,n}$ stands for  the function $X_{p_i}$ pertaining to $V_n$.

The uniqueness follows from Lemma \ref{est:uni}.


\section{ Application to compressible viscous and heat-conductive gases.}
\label{sec:NSF}
The motion of a general Newtonian fluid is described by the following system
 \begin{align}
     \label{CFNS}
     \begin{cases}
         &  \pt \ro+ \dig(\ro u )=0,\\
          &   \ro \pt u + \ro u\cdot \n u- \dig( 2\mu D u+ \lambda I_d\dig(u))+\n P=0,\\
          &\ro e_\theta(\pt \theta +u\cdot \n \theta)+\theta p_\theta \dig(u)-\dig(k \n \theta)=\Psi,
     \end{cases}
 \end{align}
with 
\begin{align}
  \Psi \defn \label{def:vaphi:Dx}
    \frac{\mu}{2}\sum_{i,j=1}^d(\partial_{x_j} u^i+ \partial_{x_i} u^j)^2+ \lambda (\dig(u))^2.
\end{align}
where $\ro=\ro(t,x)$ denotes the density, $u=u(t,x)$, the velocity field and $e=e(t,x)$
, the internal energy per unit mass and $\theta$ the temperature. 
\smallbreak
In \cite{DanADOlocal}, we addressed the question of local existence and uniqueness of solutions for System \eqref{CFNS} with large data in spaces with critical  regularity. 
Subsequently, the global existence, for small initial data was
 proved in \cite{DanADOglob}). All these works have been 
performed in Besov spaces constructed on $L^2$,  under
 the following (physically relevant) assumption:
\paragraph{\textbf{\textbf{Assumption G}}}
\begin{description}
    \item[(G1)] There exists an entropy function $s=s(\rho,e)$
       such that the following first law of thermodynamics is satisfied: 
    \begin{align}
    \label{Gibs}
    Tds=de+pd(\frac{1}{\ro})\cdotp
\end{align}
    Furthermore we assume that: for all $\ro>0$ and $\theta>0$,
    \begin{align}
        \label{positive:Pr:eth}
        p_\ro(\ro,\theta) =\frac{\partial p}{\partial\ro}>0 \esp{and} e_\theta(\ro,\theta)= \frac{\partial e}{\partial\theta}>0.
    \end{align}
    \item[(G2)] The viscosities $\mu$, $\lambda$ and the heat conductivity $k$ are smooth functions of $\ro>0$ and $\theta,$ and satisfy 
    \begin{align}
        \label{LAme:cnd}
       \mu>0,\esp{ } \nu\defn 2\mu+\lambda>0 \esp{and} k>0.
    \end{align}
\end{description}
It is important to emphasize that the condition $  2\mu + \lambda > 0 $ is stronger than the one usually used, namely $  2\mu + d\lambda > 0 $, which we actually adopt in \cite{DanADOglob}. Indeed, since $\mu > 0$, it is clear that $2\mu > -d\lambda$ implies $2\mu > -\lambda$.
\smallbreak
In order to apply the above results to the system \eqref{CFNS}, we need to rewrite it as partially symmetric hyperbolic diffuse system. For that purpose, we set
\begin{align*}
    \mathcal{U}\defn \{(\ro,u,\theta)\in 
    \mathbb R\times\mathbb{R}^{d}\times\mathbb R/ \ro>0,\; \theta>0\}.
\end{align*}
For some positive constants $\overline{\ro}>0$ and $\overline{\theta}>0$, define 
\begin{align*}
    U\defn(\ro,u,\theta) ,\;\; \overline{U}\defn (\overline{\ro},0_{\Rd},\overline{\theta}) \esp{and} V\defn U-\overline{U}.
\end{align*}
Then, \eqref{CFNS} is reduced to
\begin{align}
    \label{SHPDS:NS}
    S^0(U)\frac{d}{dt} U+\suma S^\alpha(U)\pal U-\sumab \pal\left( Y^{\alpha\beta}\pbe V\right)=H(U)
\end{align}
where
\begin{align}
    \label{def:S0:H:xi}
    S^0(U)\defn 
    \begin{pmatrix}
\frac{p_\ro}{\ro} & 0 & 0\\
0 & \ro Id & 0\\
0&0 & \frac{\ro e_\theta}{\theta}
\end{pmatrix},\;\;\;
 H(U)\defn  \begin{pmatrix}
0\\
0 \\
\Psi-k\n \theta\cdot \n (\frac{1}{\theta})
\end{pmatrix},
\end{align}

\begin{align}
    \label{def:Salpha:}
\suma S^\alpha(U)\xi_\alpha \defn
 \begin{pmatrix}
\frac{p_\ro}{\ro}u\cdot \xi & p_\ro \xi & 0\\
p_\ro ^T\xi & \ro(u\cdot \xi) Id & p_\theta ^T\xi\\
0& p_\theta \xi & \frac{\ro e_\theta}{\theta}u\cdot \xi
\end{pmatrix},
\end{align}
and
\begin{align}
    \label{Yalbe:xi}
    Y^{\alpha\beta} \defn
    \begin{pmatrix}
0& 0 \\
0 & Z^{\alpha\beta}
\end{pmatrix} 
\end{align}
\text{  with  }
\begin{align}
\label{def:Zal:be:xi}
\sumab Z^{\alpha\beta}\xi_\alpha\zeta_\beta\defn 
\begin{pmatrix}
 \mu \xi\cdot \zeta{\rm Id}+(\mu+\lambda)\xi\otimes \zeta   & 0\\
0 & \frac{k}{\theta} \xi\cdot \zeta
\end{pmatrix},\quad \xi,\zeta\in \Rd\cdotp
\end{align}

We recall that
\begin{align*}
    \xi\cdot \zeta=\sum_{\alpha=1}^d \xi_\alpha\zeta_\alpha \cdotp
\end{align*}

We observe that $S^0(U)$ is a (real) diagonal positive definite matrix for all $U\in \mathcal{U}$, and that the matrices $S^\alpha(U)$ are real symmetric. Furthermore a simple calculation shows that: for all $X,\xi\in \Rd$,  $Y\in \mathbb{R}$, 
\begin{align}
    \label{sca:Yal:be:V}
     \sumab \left\langle Z^{\alpha\beta}\xi_\alpha\xi_\beta A, A\right\rangle\ge \min(\mu,\nu)|X|^2+\frac{k}{\theta}Y^2
\end{align}
where $A\defn (X,Y)$ and $\left\langle \cdot ,\cdot \right\rangle$ denotes the standard inner product in $\mathbb{R}^{d+1}$. As for the right-hand side~$H,$ it is regarded as a lower order quadratic term. 

It is obvious that System \eqref{CFNS} under assumption  \textbf{G} satisfies hypotheses  (\textbf{H1})-(\textbf{H3}). As $n_1=1$, Assumption  (\textbf{H4}) is clearly satisfied and we can choose $P$ equal to the constant $1$. Therefore   (\textbf{H5}) is verified. Moreover, if we  set
$$B:\omega\in \mathbb{S}^{d-1}\mapsto B(\omega)=\overline{\ro}P_{\ro}(\overline{\ro},\overline{\theta})^2 \omega N^{-1}(\omega) {}^T\omega \esp{with} N(\omega)\defn \overline{\mu}{\rm Id} +(\overline{\mu}+\overline{\lambda})\omega\otimes\omega, $$
and $ \overline{\mu}\defn \mu(\overline{\ro},\overline{\theta}),  \overline{\lambda}\defn \lambda(\overline{\ro},\overline{\theta})  $ then, denoting $\overline{\nu}\defn \nu(\overline{\ro},\overline{\theta})$, we can verify that:
$$ N^{-1}(\omega)= \overline{\mu}^{-1}{\rm Id}+ (\overline{\nu}^{-1}-\overline{\mu}^{-1})\omega\otimes\omega.$$

Thus, for all $\omega\in \mathbb{S}^{d-1}$,
\begin{align}
    B(\omega)= \overline{\ro}P_{\ro}(\overline{\ro},\overline{\theta})^2 \overline{\nu}^{-1}.
\end{align}
Hence $B$ is independent of $\omega$ and  under assumption \textbf{G}, $B>0$. So assumption (\textbf{H6}) is fulfilled. The vector $H $ defined in \eqref{def:S0:H:xi} verifies the assumption (\textbf{H7}). Here the (SK) condition is equivalent to $P_\ro(\overline{\ro},\overline{\theta})\neq 0$, which is satisfied, under Assumption \textbf{G}.
\smallbreak
We can thus apply Theorem \ref{thm:p} to System \eqref{CFNS}, getting:
\begin{thm}    \label{thm:glo:CFNS:p}
   Let $d\ge 2$, $ 2\le  p_1\le p_2< 2d $ and $p_2\le \min( 4,\frac{2d}{d-2})$.  Let assumption \textbf{G} be in force.  Assume that the initial data $(\ro_0,u_0,\theta_0)$ are in $\mathcal{U}$ and set $V^1_0=\ro_0-\overline{\ro}$, $V^2_0=(u_0,\theta_0-\overline{\theta})$ and $V_0=(V^1_0,V^2_0)$. 

     Then, there exists an integer $N_0$ depending only on the coefficients of the system
     and on $\bar\rho,$      
     and a positive constant $c_0$ such that if   $ (\ro_0-\overline{\ro},u_0,\theta_0-\overline{\theta})^{l,N_0}\in \BH{\frac{d}{2}-1}{2}{1} $, $ (\ro_0-\overline{\ro})^{H,N_0}\in \BH{\frac{d}{p_1}}{p_1}{1}$ and $ (u_0,\theta_0-\overline{\theta})^{H,N_0}\in \BH{\frac{d}{p_1}-1}{p_1}{1}$ 
    and
    \[   \NBH{\frac{d}{2}-1}{2}{1}{(\ro_0-\overline{\ro},u_0,\theta_0-\overline{\theta})}^{l,N_0} +\NBH{\frac{d}{p_2}+1}{p_2}{1}{\ro_0-\overline{\ro}}^{h,N_0}+ \NBH{\frac{d}{p_2}}{p_2}{1}{(u_0,\theta_0-\overline{\theta})}^{h,N_0} \le c_0, \]
    then \eqref{CFNS} has a unique global-in-time solution $(\ro,u,\theta)\in \U $. Moreover  $V\defn(\ro-\overline{\ro},u,\theta-\overline{\theta})$ in the space  $\mathbb{E}_{p_1}\cap \mathbb{E}_{p_2}$ with $\mathbb{E}_{p_i}$ defined by
    \begin{multline*}
           (\ro-\overline{\ro},u,\theta-\overline{\theta})^{l,N_0}\in \mathcal{C}_b(\mathbb{R}^+\!,\BH{\frac{d}{2}-1}{2}{1})\cap L^1(\mathbb{R}^+\!,\BH{\frac{d}{2}+1}{2}{1}); \;\;  (\ro-\overline{\ro})^{H,N_0}\in \mathcal{C}_b(\mathbb{R}^+\!,\BH{\frac{d}{p_i}+1}{p_i}{1})\cap L^1(\mathbb{R}^+\!,\BH{\frac{d}{p_i}+1}{p_i}{1}) \\
             (u,\theta-\overline{\theta})^{H,N_0}, W^{H,N_0}\in \mathcal{C}_b(\mathbb{R}^+,\BH{\frac{d}{p_i}}{p_i}{1})\cap L^1(\mathbb{R}^+,\BH{\frac{d}{p_i}+2}{p_i}{1});\;\; (\pt\ro, \pt u, \pt \theta)\in L^1(\mathbb{R}^+,\BH{\frac{d}{p_i}}{p_i}{1}) 
    \end{multline*}
with 
\[ W= i|D|^{-1}\mathcal{Z}^{-1}(D)\left( \Ss_{21}(D) (\ro-\overline{\ro})+  \Ss_{22}(D)\begin{pmatrix}
    u\\
    \theta-\overline{\theta}
\end{pmatrix}\right)  +\begin{pmatrix}
    u\\
    \theta-\overline{\theta}
\end{pmatrix}\]
and
$$\displaylines{
 \Ss_{21}(D)=
    \begin{pmatrix}
        & \frac{P(\overline{\ro},\overline{\theta})^T\n}{|D|}\\
        &0
    \end{pmatrix},\;
    \Delta\mathcal{Z}(D)=\begin{pmatrix}
       & -\overline{\mu}\Delta {\rm Id} + (\overline{\lambda}+\overline{\mu}\n{\rm div}) & 0\\
        & 0 & -\frac{k(\overline{\ro},\overline{\theta})}{\overline{\theta}}
    \end{pmatrix},\;\cr
    \Ss_{22}(D)=\begin{pmatrix}
        &0 & P_\theta(\overline{\ro},\overline{\theta})\frac{^T \n}{|D|}\\
       & P_\theta(\overline{\ro},\overline{\theta})\frac{ \n}{|D|} & 0
    \end{pmatrix}\cdotp}$$

Furthermore we have the following bounds: 
\begin{align*}
    \mathbb{V}_{p_i}&\lesssim  \mathbb{V}_{{p_i},0}\defn  \NBH{\frac{d}{2}-1}{2}{1}{(\ro_0-\overline{\ro},u_0,\theta_0-\overline{\theta})}^{l,N_0} +\NBH{\frac{d}{p_2}+1}{p_2}{1}{\ro_0-\overline{\ro}}^{h,N_0}+ \NBH{\frac{d}{p_2}}{p_2}{1}{(u_0,\theta_0-\overline{\theta})}^{h,N_0}
\end{align*}
with  for $1\le i\le 2$
\begin{multline*}
      \mathbb{V}_{p_i} = \LpNBH{\frac{d}{2}-1}{2}{1}{(\ro-\overline{\ro},u,\theta-\overline{\theta})}{\infty}^{l,N_0}+\LpNBH{\frac{d}{2}+1}{2}{1}{(\ro-\overline{\ro},u,\theta-\overline{\theta})}{1}^{l,N_0} + \Vert \ro-\overline{\ro}\Vert_{L^\infty\cap L^1(\BH{\frac{d}{p_i}+1}{p_i}{1})}^{h,N_0}\\
    + 
    \LpNBH{\frac{d}{p_i}}{p_i}{1}{(W,u,\theta-\overline{\theta})}{\infty}^{h,N_0} + \LpNBH{\frac{d}{p_i}+2}{p_i}{1}{(W,u,\theta-\overline{\theta})}{1}^{h,N_0}+  \LpNBH{\frac{d}{p_i}}{p_i}{1}{\pt V}{1}.
\end{multline*}

\end{thm}

Applying Theorem \ref{thm:p:s} gives us

\begin{thm}
    \label{thm:glo:CFNS:p:s}
    Let $d\ge 2$. Let Assumption \textbf{G} be in force.   Assume that $(\ro_0-\overline{\ro})\in \BH{\frac{d}{p_1}+1}{p_1}{1}$, $ (u_0,\theta_0-\overline{\theta})\in \BH{\frac{d}{p_1}}{p_1}{1} $ for some $p_1\ge 2$ with $\ro_0,u_0,\theta_0\in \U$. Let $p_2\ge p_1$. Then there exists an integer $N_0$ 
    depending only on the coefficients of the system and on $\bar\rho,$ such that if we make the following additional  (lower order) assumptions 
    \[  (u_0,\theta_0-\overline{\theta}), (\ro_0-\overline{\ro})^{l,N_0}\in \BH{s}{2}{r} ;\;\; 
    (\rho_0^{h,N_0}-\bar\rho)\in   \BH{s+1}{2}{r} \]
    for some $r\in [1,\infty]$ and $s$ such that
    \begin{align}
        \label{def:s:r:p2:CFNS}
         s<\frac{d}{2} \;\; \text{and}\;\; r>1 \esp{or} s=\frac{d}{2}-1 \;\; \text{and}\;\; r=1,
    \end{align}
     then there exists a  positive constant $c_0$ such that if  
   \begin{align*}
       \NBH{s}{2}{r}{(\ro_0-\overline{\ro})}^{l,N_0}+ \NBH{s}{2}{r}{(u_0,\theta_0-\overline{\theta})}+  \NBH{s+1}{2}{r}{(\ro_0-\overline{\ro})}^{h,N_0} +\NBH{\frac{d}{p_2}}{p_2}{1}{(u_0,\theta_0-\overline{\theta})}^{h,N_0}+ \NBH{\frac{d}{p_2}+1}{p_2}{1}{(\ro_0-\overline{\ro})}^{h,N_0}\le  c_0,
   \end{align*}
    then \eqref{CFNS} has a global-in-time solution $(\ro,u,\theta)\in \U$. Moreover  $V\defn(\ro-\overline{\ro},u,\theta-\overline{\theta})$ in the space $\mathfrak{E}_{p_1}^s\cap \mathfrak{E}_{p_2}^s$ with 
    $\mathfrak{E}_{p_i}^s$ defined by
     \begin{align*}
          &   (u,\theta-\overline{\theta}), (\ro-\overline{\ro})^{l,N_0}\in \widetilde{\mathcal{C}}_b(\mathbb{R}^+,\BH{s}{2}{r})\cap L^1(\mathbb{R}^+,\BH{s+2}{2}{r}); \;\; (\ro-\overline{\ro})^{h,N_0}\in \widetilde{\mathcal{C}}_b(\mathbb{R}^+,\BH{s+1}{2}{r})\cap \widetilde L^1(\mathbb{R}^+,\BH{s+1}{2}{r})\\
            &  (\ro-\overline{\ro})^{H,N_0}\in \mathcal{C}_b(\mathbb{R}^+,\BH{\frac{d}{p_i}+1}{p_i}{1})\cap L^1(\mathbb{R}^+,\BH{\frac{d}{p_i}+1}{p_i}{1}) ;\;\; (\pt\ro,\pt u,\pt \theta)\in L^1(\mathbb{R}^+,\BH{\frac{d}{p_i}}{p_i}{1})  \\
            & (u,\theta-\overline{\theta})^{H,N_0}, W^{H,N_0}\in \mathcal{C}_b(\mathbb{R}^+,\BH{\frac{d}{p_i}}{p_i}{1})\cap L^1(\mathbb{R}^+,\BH{\frac{d}{p_i}+2}{p_i}{1}),
    \end{align*}
with $W$ defined in Theorem \ref{thm:glo:CFNS:p} .
Next, we have the following bound for all $t\ge 0$:
\begin{multline*}
  X_{p_i}(t)\lesssim  
     \NBH{s}{2}{r}{\ro_0-\overline{\ro}}^{l,N_0}+ \NBH{s}{2}{r}{(u_0,\theta_0-\overline{\theta})}+  \NBH{s+1}{2}{r}{\ro_0-\overline{\ro}}^{h,N_0} +\NBH{\frac{d}{p_2}}{p_2}{1}{(u_0,\theta_0-\overline{\theta})}^{h,N_0}+ \NBH{\frac{d}{p_2}+1}{p_2}{1}{\ro_0-\overline{\ro}}^{h,N_0}
\end{multline*}
with 
\begin{multline*}
     X_{p_i}(T) \defn \TLpNBH{s}{2}{r}{\ro-\overline{\ro}}{\infty}^{l,N_0}+ \TLpNBH{s}{2}{r}{(u,\theta-\overline{\theta})}{\infty}+\TLpNBH{s+2}{2}{r}{(u,\theta-\overline{\theta}) }{1}+\TLpNBH{s+2}{2}{r}{\ro-\overline{\ro}}{1}^{l,N_0}\\
    +\TLpNBH{s+1}{2}{r}{\ro-\overline{\ro}}{\infty}^{h,N_0}+\TLpNBH{s+1}{2}{r}{\ro-\overline{\ro}}{1}^{h,N_0}+ \Vert \ro-\overline{\ro}\Vert_{L^\infty\cap L^1(\BH{\frac{d}{p_i}+1}{p_i}{1})}\\
    + 
    \LpNBH{\frac{d}{p_i}}{p_i}{1}{(W,u,\theta-\overline{\theta})}{\infty}^{h,N_0} + \LpNBH{\frac{d}{p_i}+2}{p_i}{1}{(W,u,\theta-\overline{\theta})}{1}^{h,N_0}.
\end{multline*}

In addition, the uniqueness holds for $p_2<2d$.

\medbreak
\paragraph{\textbf{Acknowledgments}}\textit{The first author has received funding from the European Union's Horizon 2020 research and innovation programme under the Marie Sk\l{}odowska-Curie grant agreement \textnumero~945332.}  I would like to acknowledge my PhD advisor Rapha\"{e}l Danchin  for the fruitful
discussions and careful reading of this work.
    
\end{thm}

\appendix
\label{app}

\section{Some inequalities }
\label{tool:box:glo:Lp}
Here we gather a few technical results that have been used repeatedly in this thesis. The first one is the following well known result which consists in integrating suitably a differential  inequality.
\begin{lem}
\label{lem_der_int} 
\label{lem_der_int:loc}
\label{lem_der_int:glo}
    Let $X : [0, T]\longrightarrow \mathbb{R}_+$ be a continuous function such that $X$ is differentiable, and let $1\le p<\infty$ be a Lebesgue index.
    
Assume that there exist a constant $B\ge 0$ and a measurable function $A :\mathbb{R}_+\longrightarrow [0, T] $  such that
\[ \frac{1}{p}\frac{d}{dt} X +BX\le AX^{1-\frac{1}{p}} \esp{a.e.} \text{   on   }  [0, T].\]
Then, for all $t\in [0,T]$, we have
\[  X^{\frac{1}{p}}(t)+B\int^t_0 X^{\frac{1}{p}} \le X^{\frac{1}{p}}(0) + \int^t_0 A.\]

\end{lem}
\begin{proof}Le $\varepsilon>0$.
It is easy to get the following inequality verified by $X$
\[ \frac{1}{p}\frac{d}{dt} (X+\ee^p) +B (X+\varepsilon^p)\le A(X+\varepsilon^p)^{1-\frac{1}{p}}+B\varepsilon(X+\varepsilon^p)^{1-\frac{1}{p}}.\]

As $X(t)+\varepsilon^p>0$ for all $t\in [0,T]$, we have 
\begin{align*}
  \frac{d}{dt}\left(X+\ee^p\right)^\frac{1}{p}+B \left(X+\ee^p\right)^\frac{1}{p}\le A+B\ee. 
\end{align*}
Then, integrating the previous inequality over $[0,t]$ yields for all $t\in[0,T]$:

\begin{align*}
    \left(X(t)+\ee^p\right)^\frac{1}{p} +B\int^t_0 \left(X(\tau)+\ee^p\right)^\frac{1}{p} d\tau\le \int^t_0 A(\tau)d\tau+t\ee B.
\end{align*}

To obtain the desired inequality, it suffices to have $\varepsilon$ going to zero. 
\end{proof}
\smallbreak

We used maximal regularity estimates in Besov spaces with last index $1$ for parabolic system: we consider functions $\mathcal{A}_\gamma\in \mathcal C^\infty(\Rd\backslash \{0\};\mathcal{M}_n(\mathbb{C}))$ hermitian,  homogeneous of degree $\gamma$ such that matrix $\mathcal{A}_\gamma(\xi)$ 
    satisfies for some constant $c>0$
    \begin{align}
        \label{posi:cond:A}
        (\mathcal{A}_\gamma(\xi)z\cdot z)\ge c|\xi|^\gamma|z|^2\;\; \xi\in \Rd\backslash \{0\}\; z\in \mathbb{C}^n.
    \end{align}
   Let $u\in \mathcal C([0,T]; \mathcal{S}^{\prime})$ satisfy 
   \begin{align*}
       \begin{cases}
         & \pt u+\mathcal{A}_\gamma u=f\\
          & u_{|t=0}=u_0.
       \end{cases}
   \end{align*}
    Then, according to  Lemma A.2 of \cite{Danc22}), there exist two constants $c_0, C$ such that for all $p\in [1,\infty]$ and $s\in \mathbb{R}$, the following inequality holds true for all $0<t\le T$ and $j\in \mathbb{Z}$:
    \begin{align}
    \label{maxi:regu:z}
    \normepp{\Dj u(t)}\le C\biggl(e^{-c_02^{j\gamma }t} \normepp{\Dj u_{0}}+\int^t_0 e^{-c_02^{j\gamma }(t-\tau)} \normepp{\Dj f(\tau)}d\tau \biggr)\cdotp
    \end{align}
Arguing as in the proof of \cite[Lem. A.2]{Danc22}),
one can extend the above inequality to the case where $\mathcal{A}_\gamma$ is not Hermitian
as follows:
    \begin{lem}
        \label{lem:A:posi:maxi}
         Let us consider $\mathcal{A}_\gamma\in \mathcal C^\infty(\Rd\backslash \{0\};\mathcal{M}_n(\mathbb{C}))$,  homogeneous of degree $\gamma$, such that for all $\xi\neq 0$, $\mathcal{A}_\gamma(\xi)$ is diagonalizable, with positive real parts eigenvalues. Assume furthermore that  the matrices $\mathcal{A}_\gamma(\xi)$ is diagonalized by a matrix $\mathcal{P}_\gamma(\xi)$,
         with the maps $\xi\mapsto \mathcal{P}_\gamma(\xi)$ and $\xi\mapsto \mathcal{P}^{-1}_\gamma(\xi)$  bounded on any compact subset of $\Rd\backslash\{0\}$.

          Then, there exist two constants $c_0, C$ such that for all $p\in [1,\infty]$, $s\in \mathbb{R}$, $0<t\le T$ and $j\in \mathbb{Z},$ Inequality \eqref{maxi:regu:z} holds true.
    \end{lem}
    \smallbreak
\begin{proof}
   The Lemma \ref{lem:A:posi:maxi} is similar to \cite[Lem. A.2]{Danc22}). For the sake of completeness, we shall provide the proof. From Duhamel formula, 
   \begin{align}
       \label{sol:u}
     \Dot{\Delta}_j  u(t)= {\rm e}^{-t\mathcal{A}_\gamma(D)} \Dot{\Delta}_j u_0 +\int^t_0  {\rm e}^{-(t-\tau)\mathcal{A}_\gamma(D)} \Dot{\Delta}_j f(\tau)\cdotp
   \end{align}
   Hence it suffices to prove that there exists two positive constants $c_0$, C such that, for all $j \in \mathbb{Z}$
   \begin{align}
       \label{est:exp:p}
       \normep{ \Dot{\Delta}_j  {\rm e}^{-t\mathcal{A}_\gamma(D)}} \dot{\Delta}_j z  \le {\rm e}^{-c_02^{\gamma j}t} \normep{ \Dot{\Delta}_j z} .
   \end{align}

   Let us first prove inequality \eqref{est:exp:p} for $j=0$. Consider $\phi\in C^\infty(\Rd\backslash \{0\})$ with compact support and value $1$ on a neighborhood of the support of $\varphi$ defined in \eqref{def:varphi:chap3}. We can write 
   \begin{align*}
       {\rm e^{-\lambda\mathcal{A}_\gamma(D)}} \dot{\Delta}_0 z &=\mathcal{F}^{-1}\bigl( \phi {\rm e}^{-\lambda \mathcal{A}_\gamma(\cdot)} \varphi \widehat{z}\bigr),\;\; \lambda\ge 0\\
       &= H_\lambda \star \Dot{\Delta}_0 z \esp{with } H_\lambda\defn (2\pi)^{-d}\intd {\rm e}^{i x\cdot \xi }\phi(\xi) {\rm e}^{-\lambda\mathcal{A}(\xi)}d\xi.
   \end{align*}
   From Young inequality, it suffices to prove that 
   \begin{align}
       \label{est:conv:L1}
       \normes{H_\lambda}{1}\le C {\rm e}^{-c_0 t}
   \end{align}
   to conclude that
   \begin{align*}
       \normep{ \Dot{\Delta}_0  {\rm e^{-t\mathcal{A}_\gamma(D)}} \dot{\Delta}_j z } \le  \normes{H_\lambda}{1} \normep{ \Dot{\Delta}_0 z } \le C {\rm e}^{-c_0 t}\normep{ \Dot{\Delta}_0 z },
   \end{align*}
  which corresponds to \eqref{est:exp:p} with $j=0$.

  The proof of \eqref{est:conv:L1} passes through the validity of the following inequality 
  \begin{align*}
      |H_\lambda(x)|\le C(1+|x|^2)^{-d} {\rm e}^{-c_0 \lambda}, \;\;\; x\in \Rd,\; \lambda>0.
  \end{align*}
   From integration by parts, one has
   \begin{align*}
       H_\lambda(x)=\frac{1}{(2\pi)^d} (1+|x|^2)^{-d} h_\lambda(x), \esp{with} h_\lambda\defn \intd {\rm e}^{i x\cdot \xi}\phi(\xi) ({\rm Id}-\Delta_\xi)^d \bigl( {\rm e}^{-\lambda \mathcal{A}_\lambda(\xi)} \bigr) d\xi 
   \end{align*}
   Recall that the differential of the exponential map may be written as:
   \begin{align*}
       D {\rm e}^X\cdot G =\int^1_0 {\rm e}^{(1-\tau)X}G   {\rm e}^{\tau X}, \esp{ } G\in \mathcal{M}_{n}(\mathbb{R}).
   \end{align*}
   Hence, from the chain rule one may write
   \begin{align*}
       D_\xi\bigl( {\rm e}^{-\lambda \mathcal{A}_\lambda(\xi)}\bigr)\cdot G &=-\lambda\int^1_0 {\rm e}^{-\lambda(1-\tau)\mathcal{A}(\xi)}\bigl( D_\xi \mathcal{A}_\gamma(\xi)\cdot G\bigr) {\rm e}^{-\lambda \tau \mathcal{A}_\gamma(\xi)}d\tau\\
       & =-\lambda\int^1_0 \mathcal{P}_\gamma(\xi) {\rm e}^{-\lambda(1-\tau)\mathcal{D}_\gamma(\xi)} \mathcal{P}_\gamma(\xi)^{-1}\bigl( D_\xi \mathcal{A}_\gamma(\xi)\cdot G\bigr)\mathcal{P}_\gamma(\xi) {\rm e}^{-\lambda \tau \mathcal{D}_\gamma(\xi)}\mathcal{P}_\gamma(\xi)^{-1}d\tau,
   \end{align*}
   with  $\mathcal{A}_\gamma(\xi) =\mathcal{P}_\gamma(\xi) \mathcal{D}_\gamma(\xi) \mathcal{P}_\gamma(\xi)^{-1} $ and $  \mathcal{D}_\gamma(\xi) $ is diagonal matrix with positive real part entries.
Remember that $\xi\in {\rm Supp}\phi$ which is a compact subset of $\Rd\backslash \{0\}$. Since $ \mathcal{P}_\gamma $ and $\mathcal{P}_\gamma^{-1}$ are bounded on any compact subset of $\Rd\backslash\{0\}$ and $\mathcal{A}_\gamma $ is smooth on $\Rd\backslash\{0\}$, there exists a constant $C>0$ such that
$$ |\mathcal{P}_\gamma^{-1}(\xi)|\le C,\;\;  |\mathcal{P}_\gamma(\xi)|\le C,\;\; |D_\xi \mathcal{A}_\gamma(\xi)|\le C,\;\;\;\; \xi\in {\rm Supp}\; \phi. $$
Moreover, from \eqref{posi:cond:A} we deduce that there exists a positive constant $c_0$ such that the real parts of each entry of the matrix $ \mathcal{D}_\gamma(\xi) $ is bounded from below by $c_0$ on ${\rm Supp}\; \phi$. We deduce that (up to slight modification of $C$)
$$ |  D_\xi\bigl( {\rm e}^{-\lambda \mathcal{A}_\lambda(\xi)}\bigr)|\le C {\rm e}^{-\lambda c_0}. $$
By induction, one can get similar estimates for higher order derivatives of $\xi\mapsto \mathcal{A}_\gamma(\xi)$, which eventually yields
$$|h_\lambda(x)|\le C {\rm e}^{-\lambda c_0},$$
    and complete the proof of \eqref{est:exp:p} for $j=0$.
    
    The proof of \eqref{est:exp:p}  for any $j\in \mathbb{Z}$ can be deduced thanks to the homogeneity of $\mathcal{A}_\gamma$, and a suitable change of variables. In fact, setting $\zeta(x)=z(2^{-j}x)$, we have $\Dot{\Delta}_0 \zeta(2^j x)=\Dot{\Delta}_j z(x)$ and 
    $$ {\rm e}^{-2^{\gamma j}\lambda \mathcal{A}_\gamma(D)} \Dot{\Delta}_0 \zeta(2^j x)= {\rm e}^{-\lambda \mathcal{A}_\gamma(D)}\Dot{\Delta}_j z(x), \;\;\; \lambda\ge 0, $$
    which  reduces the proof of \eqref{est:exp:p} to the case $j=0$.
\end{proof}

    \smallbreak
Based on the Lemma \ref{lem:A:posi:maxi}, one can prove the following result:
    
\begin{lem}
    \label{lem:maxi:regula}
    Let us consider $S^0$ a positive symmetric matrix, $\mathcal{A}_\gamma\in C^\infty(\Rd\backslash \{0\};\mathcal{M}_n(\mathbb{C}))$ Hermitian,  homogeneous of degree $\gamma$ such that the matrix $\mathcal{A}_\gamma(\xi)$ 
    satisfies \eqref{posi:cond:A}.
   Let $u\in C([0,T]; \mathcal{S}^{\prime})$ satisfy 
   \begin{align*}
       \begin{cases}
         & S^0 \pt u+\mathcal{A}_\gamma u=f\\
          & u_{|t=0}=u_0.
       \end{cases}
   \end{align*}
    Then, there exists two constants $c_0, C$ such that for all $p\in [1,\infty]$ and $s\in \mathbb{R}$, $u$ satisfies Inequality \eqref{maxi:regu:z}, for all $0<t\le T$ and $j\in \mathbb{Z}$.
\end{lem}
\begin{proof}
According to Lemma \ref{lem:A:posi:maxi}, it suffices to prove that $(S^0)^{-1}\mathcal{A}_\gamma(\xi)$ is diagonalizable with positive eigenvalues.  Denote by $R$  the positive
symmetric square root of $(S^0)^{-1}$. Let  $\mathcal{A}_0$ be the
homogeneous of degree $0$ matrix such that 
\begin{align*}
    \mathcal{A}_\gamma(\xi) = |\xi|^\gamma \mathcal{A}_0(\xi).
\end{align*}

As $\mathcal{A}_\gamma(\xi)$
 is Hermitian (so does $\mathcal{A}_0(\xi) $), then $ R\mathcal{A}_0(\xi)R$ is diagonalizable in an orthonormal basis, that is to say there exists a matrix $O(\xi)\in \mathcal{U}_n(\mathbb{R})$ such that $ H(\xi)\defn   O(\xi) R\mathcal{A}_0(\xi)R ^*O(\xi)  $ is diagonal, with ${}^* O\defn {}^T\overline{O}$. Moreover, taking advantage of \eqref{posi:cond:A} we discover that the entries of $H(\xi)$ are positive. Since  $ \mathcal{A}_0(\xi)= R^{-1} {}^*\!O(\xi)H(\xi) O(\xi)R^{-1},$ this implies that $(S^0)^{-1}\mathcal{A}_\gamma(\xi)$ is diagonalizablez by $O(\xi)R^{-1} $ with $|O(\xi)R^{-1}|\le |R^{-1}|$ for all $\xi \neq 0$.
 Then, using Lemma \ref{lem:A:posi:maxi} completes the proof.

\end{proof}
The following lemma (see also \cite{BaratDan23}) establishes the equivalence between the (SK) condition and the positivity of certain matrices, while also highlighting the connection to the rank of the matrices involved in the (SK) condition
\begin{lem}
    \label{lem:equi:SK:positi:V1}
    Consider two $(n_1+ n_2)\times (n_1+ n_2)$ matrices A and B such that
    \begin{align}
        \label{form:A:B:appen}
        A=\begin{pmatrix}
            & 0 & A_{12}\\
            &A_{21}& A_{22}
        \end{pmatrix}
        \esp{and} 
         B=\begin{pmatrix}
            & 0 & 0\\
            &0& B_{22}
        \end{pmatrix}
    \end{align}
    with  $A_{21}\in \mathcal{M}_{n_2\times n_1}$,  $ A_{12}\in \mathcal{M}_{n_1\times n_2}$ and $B_{22}, A_{22}\in \mathcal{M}_{n_2\times n_2}$. Suppose that
    \begin{itemize}
        \item $B_{22}$ is positive,
        \item $A$ is a Hermitian matrix.
    \end{itemize}
    Then, the following two properties are equivalent:
    \begin{enumerate}
        \item The matrix $A_{12}B^{-1}_{22} A_{21}$ is a $n_1\times n_1$ positive matrix.
        \item Condition (SK) holds true (see Lemma \ref{lem:SK_diss}).
    \end{enumerate}
    Moreover if we have  one of these two results, then $n_1\le n_2$. 
\end{lem}
\begin{proof}
First, assume that $A_{12}B^{-1}_{22} A_{21}$ is positive. Then, the rank of $A_{21}$ must be equal to $n_1,$ which implies that $n_1\le n_2$, and so does the rank of $B_{22}A_{21}$. Now, we observe
\begin{align*}
BA=
    \begin{pmatrix}
        & 0 &0\\
       & B_{22}A_{21} & B_{22}A_{22}
    \end{pmatrix}.
\end{align*}
Hence, the rank of 
$ \begin{pmatrix}
        & B\\
       & BA
    \end{pmatrix} $ 
is equal to $n_1+n_2$, and Condition (SK) is thus satisfied.

Conversely, assume that the matrix $ A_{12}B^{-1}_{22} A_{21}  $ is not positive definite. Since $ ^t \overline{A}_{12}=A_{21} $ and $B$ is positive, there exists $z\in \mathbb{C}^{n_1}\backslash\{ 0\}$ such that
    \begin{align*}
        B^{-1}_{22} A_{21} z\cdot  A_{21} z = A_{12}B^{-1}_{22} A_{21}z\cdot z=0.
    \end{align*}
    Then, under the assumption of positivity of $B$, we get $A_{21} z=0$. As we assume that $ A_{11}=0 $, then the vector $ ^t(z,0_{n_2})\in \mathbb{C}^{n_1+n_2} $ is both an eigenvector of $A$ and belongs to the null space of matrix $B$. Appealing to the $(SK)$ condition gives us a contradiction. This proves that the matrix $A_{12}B^{-1}_{22} A_{21}$ is positive.
\end{proof}


We also have the following Lemma, of independent interest.
\begin{lem}\label{lem:comp:lf}
    Let $s>0$ and $2\le p\le 4$. Denote $\frac{1}{p^*}\overset{def}{=}\frac{1}{2}-\frac{1}{p}\cdotp$ Let $f$ be a smooth function on $\mathbb{R}$ with $f(0)=0$.
    Then there exists a constant C depending only on the threshold number $M$ between low and high frequencies and on $s, p, d$ and $f$, such that
    \begin{equation}
        \label{eq:comp:lf}
        \NBH{s}{2}{1}{f(u)}^{l} \le C\left( \NBH{s}{2}{1}{u}^{l}+ (1+ \Vert u\Vert_{L^{p^*}})(\NBH{s}{p}{1}{u}^h +\NBH{s}{2}{1}{u}^{l})\right)\cdotp
    \end{equation}  
\end{lem}

\begin{proof}
Set $  u_{M} \overset{def}{=} \Dot{S}_{M} u$. Next, decompose $ f(u)$ as follows:
\begin{equation}
    \label{decom:f(u)}
    f(u)=f(u_{M}) +\left(  f(u)-f(u_{M})\right)\cdotp
\end{equation}
   One the one hand,  according to Proposition \ref{propo:comp}, we have
   \begin{equation*}
       \NBH{s}{2}{1}{f(u_{M})} \lesssim \NBH{s}{2}{1}{u_{M}}
   \end{equation*}
   from which we deduce that
    \begin{equation}
    \label{ine:comp:low}
       \NBH{s}{2}{1}{f(u_{M})} \lesssim \NBH{s}{2}{1}{u}^{l}
   \end{equation}
On the other hand, according to \cite{HajDanChe11}[Lemma 262, p.94]  the following equality holds true in the tempered distributions
meaning

\begin{align}
    \label{decomp:fu:hf}
    f(u)-f(u_{M})= \sum_{k> M } f_k \esp{with} f_j=f(\Dot{S}_{k+1} u)-f(\Dot{S}_{k} u).
\end{align}
Furthermore it holds
\begin{align}
    \label{def:mk}
    f_k= m_k\Dot{\Delta}_k u  \esp{with} m_k =\int_0^1 f^{\prime}(\Dot{S}_{k+1}u+t\Dot{\Delta}_{k}u  )dt.
\end{align}
 Using the mean value theorem and the fact that for all $0\le t\le 1$,  $ \normeinf{\Dot{S}_{k+1}u+t\Dot{\Delta}_{k}u}\lesssim \normeinf{u}$; $f$ and all its derivatives are bounded on $B(0,C\normeinf{u})$, we get   for all $q\in [1,\infty],t\in [0,1]$ 
\begin{align*}
    \normess{f^{\prime}(\Dot{S}_{k+1}u+t\Dot{\Delta}_{k}u)-f^{\prime}(0)}{q} &\lesssim C(f^{\prime\prime}, \normeinf{u})\normess{\Dot{S}_{k+1}u+t\Dot{\Delta}_{k}u}{q}\\
    &\lesssim C(f^{\prime\prime}, \normeinf{u}) \normess{u}{q}.
\end{align*}
 As the operator $ \Dot{S}_{M+1}$ is continuous on $L^q$, we deduce 
\begin{align}
\label{est:f'(...):q}
    \normess{m_k}{q} 
    &\lesssim C(f^{\prime\prime}, \normeinf{u})(1+ \normess{u}{q}).
\end{align}

 For all $j\le M+1 $, we have from \eqref{decomp:fu:hf} and the continuity of $ \Dot{\Delta}_j $
 \begin{align}
    2^{js} \normed{\Dot{\Delta}_j\left(  f(u)-f(u_{M}) \right)} &\le \sum_{k> M}2^{-s(k-j)} 2^{ks}\normed{ f_k}\nonumber\\
    & \le 2^{js} \normed{ f_{M+1}}+\sum_{k> M+1}2^{-s(k-j)} 2^{ks}\normed{ f_k}\nonumber\\
    &\le 2^{js} \normed{ f_{M+1}}+ \sum_{k\ge j; \; k> M+1}2^{-s(k-j)} 2^{ks}\normed{ f_k}
 \end{align}
where the last inequality comes from the fact that $ j\le M+1\le k$.\\
Observing that 
$$ \sum_{k\ge j; \; k> M}2^{-s(k-j)} 2^{ks}\normed{ f_k} =(\alpha\star \beta)_j \esp{with} \alpha_k= 1_{\mathbb{N}^*}(k) 2^{-k} \esp{and} \beta_l = 1_{l> M} 2^{ls} \normed{ f_l} $$ and using Young's inequality and then H\"{o}lder inequality, combined with \eqref{est:f'(...):q} with $q=p^*$, \eqref{def:mk} and \eqref{link:lf:hf}, we deduce that
\begin{equation}
    \label{est:lf:u:hf}
   \sum_{j\le M+1} 2^{js} \normed{\Dot{\Delta}_j\left(  f(u)-f(u_{M}) \right)}\le  C(f^{\prime\prime}, \normeinf{u})(1+ \Vert u\Vert_{L^{p^*}})(\NBH{s}{p}{1}{u}^h +\NBH{s}{2}{1}{u}^{l}).
\end{equation}
\end{proof}
To handle the nonlinear terms in low frequencies, we need the following Lemma.
\begin{lem}
\label{lem:produ:lf:hf}
Assume that $d\geq2.$
Let $s\in\mathbb R.$ It holds:
\begin{enumerate}
    \item  If  $2\le  p\le 4$ and  $\sigma\geq d(\frac{p-2}{2p})$ then
\begin{equation}
    \label{lem:T:lf:hf}
    \NBH{s+\frac{d}{2}-\frac{d}{p}}{2}{1}{\Dot{T}_a b}\lesssim \NBH{\frac{d}{p}-\sigma}{p}{1}{a}\NBH{s+\sigma}{p}{1}{b}.
\end{equation}
\item  If   $2\le  p\le 4$ and  $s>-d/p$ then
\begin{equation}
    \label{lem:R:lf:hf}
     \NBH{s+\frac{d}{2}-\frac{d}{p}}{2}{1}{\Dot{R}(a ,b)}\lesssim \NBH{\frac{d}{p}-\sigma}{p}{1}{a}\NBH{s+\sigma}{p}{1}{b}. 
\end{equation}
\end{enumerate}
In particular if $ s> -\frac{d}{p}$,  $ 2\le  p\le 4 $ and  $\sigma\geq d(\frac{p-2}{2p})$
then we have
\begin{equation}
    \label{lem:pro:lf:hf}
    \NBH{s+\frac{d}{2}-\frac{d}{p}}{2}{1}{ab}\lesssim \NBH{\frac{d}{p}-\sigma}{p}{1}{a}\NBH{s+\sigma}{p}{1}{b}{+\NBH{\frac{d}{p}-\sigma}{p}{1}{b}\NBH{s+\sigma}{p}{1}{a}}.
\end{equation}
\end{lem}
\begin{proof}
    Let us define the quantity  $p^*$ as follows: 
    \begin{equation}
        \label{def:p^*:T:R}
        p^*\defn \frac{2p}{p-2}
    \end{equation}
    so that 
    \begin{align}
        \label{ine:p:p*:}
         \frac{1}{p}+\frac{1}{p^*}=\frac12\cdotp
    \end{align}
We recall  the definition of the  paraproduct $\Dot{T}_a b$ of $b$ by $a$:
\begin{align*}
    \Dot{T}_a b =\sum_{k\in \mathbb{Z}} \Dot{S}_{k-1} a\Dot{\Delta}_k b 
\end{align*}
For all $k\in \mathbb{Z}$, as $\mathcal F(\Dot{S}_{k-1} a\Dot{\Delta}_k b)$ is  localized in 
the annulus $\bigl\{\frac1{12}\cdot 2^k\leq|\xi|\leq\frac{10}3\cdot2^k\bigr\},$  it holds
\begin{align*}
   \Dj \Dot{T}_a b =\sum_{|k-j|\le 4} \Dj(\Dot{S}_{k-1} a\Dot{\Delta}_k b ) \esp{for all} j\in \mathbb{Z}.
\end{align*}
    Using H\"{o}lder's inequalities and   \eqref{ine:p:p*:}, we have
    \begin{align*}
        \normed{ \Dj \Dot{T}_a b}
        \lesssim \sum_{|k-j|\le 4 }\normess{\Dot{S}_{k-1} a}{p^*}\normess{\Dot{\Delta}_k b}{p}  .
    \end{align*}
Now, using  Bernstein's inequality, the definition of the low-frequency  cut-off operator $ \Dot{S}_{k-1}$  and the fact that $p^*\ge p,$ we can bound  $\Dot{S}_{k-1} a $ for all $k$ as follows:
\begin{align*}
    \Vert \Dot{S}_{k-1} a \Vert_{L^{p^{*}}}& \lesssim \sum_{i\le k-2}  2^{id(\frac{1}{p}-\frac{1}{p^*})}\normepp{\Dot{\Delta}_i a}\\
    & \lesssim 2^{k(\sigma-\frac{d}{p^*})} \NBH{\frac{d}{p}-\sigma}{p}{1}{ a}\esp{if}\sigma\ge \frac{d}{p^*}\cdotp
\end{align*} 
Then, using again \eqref{ine:p:p*:}, it holds
\begin{align*}
     \normed{ \Dj \Dot{T}_a b}
     &\lesssim \sum_{|k-j|\le 4 } 2^{k(\frac{d}{p}+\sigma-\frac{d}{2})}\normess{\Dot{\Delta}_k b}{p}  \NBH{\frac{d}{p}-\sigma}{p}{1}{ a}.
\end{align*}
We get \eqref{lem:T:lf:hf} by multiplying the last inequality by $2^{j(s-\frac{d}{p}+\frac{d}{2})}$ and summing over $j\in \mathbb{Z}$.
\smallbreak
Next,  as $1\le \frac{p}{2}\le 2$ and $s>-{d}/{p},$ we get \eqref{lem:R:lf:hf} as a consequence of  Bernstein's inequality,  Lemma 2.100 and Remark 2.102 of \cite{HajDanChe11}: 
\begin{equation*}
     \NBH{s+\frac{d}{2}-\frac{d}{p}}{2}{1}{\Dot{R}(a,b)} \lesssim \NBH{s+\frac{d}{p}}{\frac{p}{2}}{1}{\Dot{R}(a,b)}\lesssim  \NBH{-\sigma+\frac{d}{p}}{p}{1}{a} \NBH{s+\sigma}{p}{1}{b}.\end{equation*}
   Finally, putting \eqref{lem:T:lf:hf} (for the pairs $(a,b)$ and $(b,a)$)
   and   \eqref{lem:R:lf:hf} together yields \eqref{lem:pro:lf:hf}.
\end{proof}

\section{Littlewood-Paley decomposition and Besov spaces}
\label{appendix:LP:loc}
In this section we present some results on the Littlewood-Paley decomposition and Besov spaces, the source of which can be found in \cite[Chap. 2]{HajDanChe11}. 
\subsection{Definition of Littlewood-Paley decomposition and Besov spaces}
To define the Littlewood-Paley decomposition, we  fix some smooth radial non increasing function $\chi$ with $ \mathrm{Supp}\chi \subset B(0,\frac{4}{3})$  and $\chi \equiv 1$  on $ B(0,\frac{3}{4})$ , then set $\varphi(\xi)=\chi(\frac{\xi}{2})-\chi(\xi)$ so that
\begin{align*}
    \chi+\sum_{j\ge 0} \varphi(2^{-j}\cdot)=1 \ \text{on} \ \Rd \esp{and} \sum_{j\in \mathbb{Z}} \varphi(2^{-j}\cdot) =1  \ \text{on}\ \Rd\backslash\{0\}. 
\end{align*}
In order to pursue our definition, we agree that for $S$ a function, $z$ a distribution and $a$ a real constant,  
the operator $S(aD)$ is defined by
\begin{align}
    \label{def:S(aD):loc}
    S(aD) z\defn \mathcal{F}^{-1} S(a\cdot) \mathcal{F}z,
\end{align}
when the right-hand side  makes sense. Let us now state 
\begin{align}
    \label{def:dot:S:j:loc}
    \begin{split}
   & \Dot{S}_j =\chi(2^{-j}D) \;\; \text{for all}\; j\in \mathbb{Z}\esp{and} S_j\defn\Dot{S}_j \;\; \text{for all}\; j\ge 0,\;\; S_{j}=0 \;\; \text{for all}\; j\le  -1.
     \end{split}
\end{align}
We define the homogeneous dyadic block $\Dj$ as 
\begin{align}
    \label{def:DDj:Dj:loc}
    \begin{split}
    &\Dj \defn \varphi(2^{-j}D) \;\; \text{for all}\; j\in \mathbb{Z} \\
    & \DDj=\Dj  \;\; \text{for all}\; j\ge 0,\;\; \Delta_{-1}= \Dot{S}_0 \esp{and} \DDj=0  \;\; \text{for all}\; j< -1
    \end{split}
\end{align}
and we define $\mathcal{S}^{\prime}_h$ to be the set of tempered distributions $z$ such that 
\begin{align}
    \label{cnd:on S'h:loc}
    \lim_{j \rightarrow -\infty } \Dot{S}_j z=0.
\end{align}
Following \cite{HajDanChe11}, we introduce the homogeneous Besov semi-norms~:
\begin{align}
    \label{def:NBH:NB:loc}
    \NBH{s}{p}{r}{z} \defn \lVert (2^{js} \Vert\DDj  z\Vert_{L^p})_{j\in \mathbb{Z}}\rVert_{l^r} 
\end{align}
Then, for any $s\in \mathbb{R}$ and $(p,r)\in [1,\infty]$ we define the homogeneous Besov spaces  $\BH{s}{p}{r}$  to be the subset of those $z$ in $\mathcal{S}^{\prime}_h$  such that $\NBH{s}{p}{r}{z} $  is finite.
\medbreak

Any $z\in \mathcal{S}^{\prime}$  can be decomposed in terms of high and low frequencies parts, 
as follows:
\begin{align}
    \label{def:LF:HF:}
    z^l\defn \sum_{j\le 0} \Dj z \esp{and} z^h \defn \sum_{j> 0} \Dj z.
\end{align}
Throughout this thesis, we constantly used the following Besov semi-norms for low and high frequencies:
\begin{align}
    \label{def:NBH:LF:HF}
    \NBH{s}{2}{1}{z}^l\defn   \sum_{j\le 0} 2^{js}\normede{\Dj z} \esp{and}  \NBH{s}{2}{1}{z}^h\defn   \sum_{j> 0} 2^{js}\normede{\Dj z}.
\end{align}

The study of non-stationary PDEs requires spaces of type $L^{a}(0,T;X)$ (endowed with the norm $\Vert z\Vert_{L^{a}_T(X)}= \Vert \Vert z(t,\cdot)\Vert_{X} \Vert_{L^{a}_T}$) for appropriate
Banach spaces $X$.  Here  we expect $X$ to be a Besov space.  
However, for technical reasons, it is sometimes more suitable 
to use the Chemin-Lerner spaces that are defined below: 
\begin{dfn}
    \label{def:Chemin lener}
    Let  ${a}$ in $[1,\infty]$ and time $T \in [0, \infty]$. We set
    \begin{align*}
        \TLpNBH{s}{p}{r}{z}{{a}}\defn \Vert 2^{js}\Vert \Dj z\Vert_{L^{a}_T(L^p)}\Vert_{l^r}\esp{with}  \Vert z \Vert_{L^{a}_T(L^p)} \defn \Vert z\Vert_{L^{a}(0,T;L^p)}.
    \end{align*}
    We then define the space $\widetilde{L}^{a}_T(\BH{s}{p}{r})$
 to be the set of tempered distributions $z$ over
$[0,T]\times \Rd$ such that $z \in \mathcal{S}^{\prime} $ and $ \TLpNBH{s}{p}{r}{z}{{a}}<\infty.$

    \end{dfn}
    \medbreak
 We set $\widetilde{C}([0,T];\BH{s}{p}{r})= \widetilde{L}^{\infty}_T(\BH{s}{p}{r})\cap C([0,T];\BH{s}{p}{r}) $.
 Let us emphasize that, according
to the Minkowski inequality, we have:
\begin{align}
    \label{link:CL:Besov:space:chap2}
    \TLpNBH{s}{p}{r}{z}{{a}}\le  \LpNBH{s}{p}{r}{z}{{a}},\; \text{if}\; r\ge {a} \esp{and}
     \LpNBH{s}{p}{r}{z}{{a}} \le  \TLpNBH{s}{p}{r}{z}{{a}}, \; \text{if}\; r\le {a}.
\end{align}
Even though most of the functions we shall consider here will have range in the set of
vectors or  matrices, we shall keep the same notation for Besov spaces pertaining to this case. 

One of the main motivations for using  Littlewood-Paley decomposition when dealing with PDEs is that the derivatives act almost as dilations on distributions with Fourier transform supported in a ball or an annulus, as regards $L^p$ norms. This is exactly what is stated in the following proposition:
\begin{propo}[Bernstein inequality] Let $0<r<R$.
    \label{propo:Berstein:loc}
    \begin{itemize}
        \item   There exists a constant $C$ so that, for any $k\in \mathbb{N}$, pair $(q_1,q_2)$ in $[1,\infty]^2$ with $q_2\ge q_1\ge 1$ and function $u$ of $L^{q_1}$ with $\widehat{u}$ supported in the ball $B(0,\lambda R)$ of $\Rd$  for some $\lambda>0$, we have $D^k u\in L^{q_2}$ and
    \begin{align}
    \label{bern:ine}
        \normes{D^k u}{{q_2}}\le C^{k+1} \lambda^{k+d(\frac{1}{q_1}-\frac{1}{q_2})}\normes{u}{{q_1}}.
    \end{align}
    \item 
    For any  smooth function $M$ on $\mathbb{R}^d\backslash\{0\}$ with homogeneity $\gamma$, there exists a constant $C$ such that for any $\lambda > 0$ and any function $u$  with Fourier transform $\widehat{u}$ supported in annulus  $ \{ \xi \in \Rd \;\backslash \;r\lambda\le |\xi|\le R \lambda \}$ of $\Rd$, we have
     \begin{align}
     \label{bern:ine:M}
        \normes{M(D) u}{{q_2}}\le C^{k+1} \lambda^{\gamma+d(\frac{1}{q_1}-\frac{1}{q_2})}\normes{u} {{q_1}}.
    \end{align}
    \item There exists a constant $C$ so that for any $k\in \mathbb{N}$, $q\in [1,\infty]$ and function $u$ of $L^q$ with  ${\rm Supp }\;  \widehat{u}\subset \{ \xi \in \Rd \;\backslash \;r\lambda\le |\xi|\le R \lambda \}$ for some $\lambda>0$, W
we have 
    \begin{align*}
        \lambda^k \normes{u}{q}\le C^{k+1} \normes{D^k u}{q}.
    \end{align*}
    \end{itemize}
   \end{propo}

\subsection{Paradifferential calculus and applications}
 In this paragraph, we present some results related to paradifferential calculus. We invited the reader to see \cite{HajDanChe11} for more information.
For simplicity, we present the paradifferential calculus theory in nonhomogeneous case.  The paraproduct between $u$ and
$v$, two tempered distributions on $\mathbb{R}^d$ is defined by
\begin{align}
    \label{def:parapr}
    \Dot{T}_u v =\sum_{q\in \mathbb{Z}} \Dot{S}_{q-1}u \Dot{\Delta}_q v \esp{with} \Dot{S}_{q}=\sum_{l\le q-1}\Delta_q.
\end{align}
We thus have the following formal decomposition
\begin{align*}
    u v= \Dot{T}_u v+ \Dot{T}_v u+ \Dot{R}(u, v) \esp{with}  \Dot{R}(u, v)=\sum_{|q-q^{\prime}|\le 1} \Delta_q u \Delta_{q ^{\prime}} v.
\end{align*}
In the proposition below, we give estimates for the paraproduct and
remainder in $L^{a}(\Dot{B}^s_{p,1})$.
\begin{propo}
    \label{prop:parpro:reminder}
    Let $1\le p\le \infty$ and $1\le {a},{a}_1,{a}_2\le \infty$ such that $1/{a}=1/{a}_1+1/{a}_2$.\\ For all $t\in \mathbb{R}$ and $s\le d/p$, we have
    \begin{align}
        \label{est:paraproduct}
        \TLpNBH{s+t-d/p}{p}{1}{  \Dot{T}_u v }{{a}}\le C \TLpNBH{s}{p}{1}{u}{{a}_1}\TLpNBH{t}{p}{1}{v}{{a}_2}.
    \end{align}
    In general for all $s\in \mathbb{R}$, $t\ge 0$, we have
     \begin{align}
        \label{est:paraproduct:gen}
        \begin{split}
       & \TLpNBH{s}{p}{1}{  \Dot{T}_u v }{{a}}\le C \|u \|_{L^{{a}_1}_T(L^\infty)}  \TLpNBH{s}{p}{1}{v}{{a}_2},\\  &\TLpNBH{s-t}{p}{1}{  \Dot{T}_u v }{{a}}\le C \TLpNBH{-t}{\infty}{1}{u}{{a}_1}\TLpNBH{s}{p}{1}{v}{{a}_2}.
       \end{split}
    \end{align}
     If $p_1\in [1,\infty]$, $\ee\ge 0$ and $s>-\min(d/p_1,d/p^{\prime})$, then
    \begin{align}
        \label{est:Reminder:gen}
        \TLpNBH{s}{p}{1}{  \Dot{R}(u ,v) }{{a}}\le C \TLpNBH{\frac{d}{p_1}-\ee}{p_1}{\infty}{u}{{a}_1}\TLpNBH{s+\ee}{p}{1}{v}{{a}_2}.
    \end{align}
    In particular for $s,t\in \mathbb{R}$ satisfying $s+t>d/p-\min(d/p,d/p^{\prime})$, we have
   \begin{align}
        \label{est:Reminder}
        \TLpNBH{s+t-d/p}{p}{1}{  \Dot{R}(u ,v) }{{a}}\le C \TLpNBH{s}{p}{1}{u}{{a}_1}\TLpNBH{t}{p}{1}{v}{{a}_2}.
    \end{align}
\end{propo}

 In order to bound the commutator  terms, we use the following results:

\begin{propo}
	 \label{propo_commutator-BH:loc}
	 The following inequalities hold true:
   If $\sigma>0$
     \begin{align}
       \label{comm:est:a:b:sig>0}
        \normede{[a,\Dj] b} \le 2^{-q\sigma}Cc_j (\normeinf{\n a} \lVert b \rVert_{\Dot{B}_{2,1}^{\sigma-1} }+ \normeinf{b} \lVert \n a \rVert_{\Dot{B}_{2,1}^{\sigma-1} }  ) \esp{with} \sum_j c_j =1.
   \end{align}
   
   In particular, if $\sigma\ge \cd+1$, then it holds
    \begin{align}
       \label{comm:est:a:b:sig>cd}
        \normede{[a,\Dj] b} \le 2^{-q\sigma}Cc_j  \NBH{\sigma-1}{2}{1}{\n  a}\NBH{\sigma-1}{2}{1}{b} \esp{with} \sum_{\substack{j\ge -1}} c_j =1.
   \end{align}
	 Moreover if $-\frac d2 < \sigma \le \cd+1 $, 
   then
   \begin{align}
       \label{comm:est:a:b:loc}
        \normede{[a,\Dj] b} \le 2^{-q\sigma}Cc_j \lVert \n a \rVert_{\Dot{B}_{2,\infty}^{\cd} \cap L^\infty}  \lVert b \rVert_{\Dot{B}_{2,1}^{\sigma-1} }  \esp{with} \sum_j c_j =1.
   \end{align}
   We have also
   \begin{align}
       \label{comm:est:a:b:inf:loc}
        \normede{[a,\Dj] b} \le 2^{-q\sigma}C \Vert\n a\Vert_{\Dot{B}^{\cd}_{2,\infty}\cap L^\infty} \Vert b\Vert_{\Dot{B}^{\sigma-1}_{2,\infty}}\quad\hbox{for all }
        \ -d/2\le \sigma< d/2+1.
   \end{align}
 Finally If $\sigma>0$ and $k=1,\cdots, d$,
     \begin{align}
       \label{comm:est:less:gene:sig>0}
        \normede{[a,\Dj]\partial_k b} \le 2^{-q\sigma}Cc_j (\normeinf{\n a} \lVert  b \rVert_{\Dot{B}_{2,1}^{\sigma} }+ \normeinf{b} \lVert \n a \rVert_{\Dot{B}_{2,1}^{\sigma} }  ) \esp{with} \sum_j c_j =1.
   \end{align}
  
   Similar results hold  true if we replace $L^2$ by $L^{a}_T(L^2)$ in the l.h.s and  the spaces $L^{a}_T(\Dot{B}^{s}_{2,r})$ (or $\widetilde{L}^{a}_T(\Dot{B}^{s}_{2,r})$) in the r.h.s. 
	 \end{propo}
	 \medbreak
The following inequalities generalised those of Proposition \ref{propo_commutator-BH:loc}.  
	  \begin{propo}
	 \label{propo_commutator-BH}
	 Let $i \in [1,d]$ and $1\le p,\widetilde{p},q,r\le\infty $  . The following inequalities hold true for some nonnegative sequence $(c_j)_{j\in \mathbb{Z}}$ such that $\Vert (c_j)_{j\in \mathbb{Z}}\Vert_{l^r} \le 1$:
  \begin{enumerate}
      \item For all $\theta$ a $C^1$ function with $(1+|\cdot|)\widehat{\theta}\in L^1$
     \begin{align}
      \label{comm:est:1}
      \normepp{[\theta(\lambda^{-1}D),a](b)}\le C\lambda^{-1}\normess{\nabla a}{r}\normess{b}{q} \esp{with} \frac{1}{p}=\frac{1}{q}+\frac{1}{r}\cdotp
  \end{align} 
      \item  If $-\min(d/\widetilde{p}, d/p^{\prime}) < \sigma $, then
      \begin{align}
          \label{comm:est:p}
          \normepp{[a,\Dj]\partial_i  b} \le 2^{-j\sigma}Cc_j \left( \Vert \n a\Vert_{\BH{\frac{d}{\widetilde{p}}}{\widetilde{p}}{\infty}\cap L^\infty}\NBH{\sigma}{p}{r}{b}+ \normeinf{\n b}\NBH{\sigma}{p}{r}{a} \right),
      \end{align}
      and the second term is not needed if in addition $\sigma<\cdp{\widetilde{p}}+1$ or ($\sigma=\cdp{\widetilde{p}}+1$ and $r=1$).
    \smallbreak
    \item If $-\min(d/\widetilde{p}, d/p^{\prime}) \le  \sigma< \sigma<\cdp{\widetilde{p}}+1$,  then
    \begin{align}
    \label{comm:est:a:b:inf:glo}
        \normede{[a,\Dj] b} \le C2^{-j\sigma} \Vert\n a\Vert_{\Dot{B}^{\cd}_{2,\infty}\cap L^\infty} \Vert b\Vert_{\Dot{B}^{\sigma-1}_{2,\infty}}\quad\hbox{if}
        \ -d/2\le \sigma< 1+d/2.
          \end{align}
        \smallbreak
\item     Finally, if $\frac{1}{p}=\frac{1}{q}+\frac{1}{\widetilde{p}}$ then
      \begin{equation}
          \label{comm:est:0}
          \normepp{[a,\Dj]\partial_i  b} \le 2^{-j\sigma}Cc_j \left(\Vert \n a\Vert_{\BH{\frac{d}{\widetilde{p}}}{\widetilde{p}}{\infty}\cap L^\infty}\NBH{\sigma}{p}{r}{b}+ \NBH{\sigma-1}{{\widetilde p}}{r}{\n b}\NBH{\frac{d}{q}+1}{q}{\infty}{a} \right)
          \end{equation}
          provided that $-\min(\frac{d}{\widetilde{p}},\frac{d}{p^{\prime}})<\sigma<\frac{d}{q}+1$ and  $r\ge 1,$ or $\sigma=\frac{d}{q}+1$ and $r=1.$
 \end{enumerate}
  \end{propo}
\begin{proof}
We invite the reader to consult Lemma 2.97 of \cite{HajDanChe11} for the proof of \eqref{comm:est:1}. The other two inequalities may be proved by taking advantage of Bony's decomposition. Indeed we have (with
the summation convention over repeated indices):
\begin{align*}
    [a,\Dj]\partial_i b=[\Dot{T}_{a},\Dj]\partial_i  b +\Dot{T}_{\partial_i \Dj b}a+\Dot{R}(a,\partial_i \Dj b)-\Dj\Dot{R}(a,\partial_i b)-\Dj \Dot{T}_{\partial_i b} a.
\end{align*}
According to Lemma 2.100 and Remark 2.102 of \cite{HajDanChe11}, the $L^p$ norm of each term of left-hand side (except the last one) is bounded by 
\begin{align}
    \label{ter:a:uses}
     2^{-j\sigma}Cc_j  \Vert \n a\Vert_{\BH{\frac{d}{\widetilde{p}}}{\widetilde{p}}{\infty}\cap L^\infty}\NBH{\sigma}{p}{r}{b}.
\end{align}

It remains to treat the term $ \Dj \Dot{T}_{\partial_i b} a$. To do this, one can simply use the fact that 
\[  \Dot{T}:L^\infty\times \BH{\sigma}{p}{r} \longrightarrow \BH{\sigma}{p}{r},\]  which gives 
the second term of \eqref{comm:est:p}.
This term is not needed if $-\min(d/\Tilde{p}, d/p^{\prime})<\sigma<\cdp{\widetilde{p}}+1$ or ($\sigma=\cdp{\widetilde{p}}+1$ and $r=1$),
since one can argue as in the proof of \cite[Lemma 2.100 and Remark 2.102]{HajDanChe11}.

Finally, in the case $\sigma-\frac{d}{q}<1$ (or $ \sigma-\frac{d}{q}\le 1 $ and $r=1$), combining 

\[  \Dot{T}:\BH{\sigma-\frac{d}{q}-1}{\widetilde{p}}{r}\times \BH{\frac{d}{q}+1}{q}{\infty} \longrightarrow \BH{\sigma}{p}{r} \esp{and} \BH{\sigma-1}{p}{r} \hookrightarrow  \BH{\sigma-\frac{d}{q}-1}{\widetilde{p}}{r} \]
with $\frac{1}{p}=\frac{1}{q}+\frac{1}{\widetilde{p}}$ for bounding 
$\Dj \Dot{T}_{\partial_i b} a$ yields \eqref{comm:est:0}. 
\\
\end{proof}

	  The following product laws in Besov spaces have been used several times.
	 
	 	 \begin{propo}
	 \label{propo_produc_BH}
     Let $(s,p,r)\in ]0,\infty[\times[1,\infty]^2$. Then $\BH{s}{p}{r}\cap L^\infty$ is an algebra and we have 
	     \begin{equation}
	         \label{product_propo1}
	         \lVert ab \rVert_{\Dot{B}_{p,r}^{s}}
	         \le C(\normeinf{a} \lVert b \rVert_{\Dot{B}_{p,r}^{s}} + \normeinf{b} \lVert a \rVert_{\Dot{B}_{p,r}^{s}} ).
	     \end{equation}
	      Furthermore, if $-\min(d/p,d/p^{\prime}) < s \le d/p$, then the following inequality holds:
	      \begin{equation}
	         \label{product_propo2}
	          \lVert ab \rVert_{\Dot{B}_{p,1}^{s}}
	         \le C \lVert a \rVert_{\Dot{B}_{p,1}^{\frac{d}{p}}} \lVert b \rVert_{\Dot{B}_{p,r}^{s}},
	     \end{equation}
      and, if $-\min(d/p,d/p^{\prime}) \le s < d/p $ then, we have
       \begin{equation}
	         \label{product_propo:inf:p}
	          \lVert ab \rVert_{\Dot{B}_{p,\infty}^{s}}
	         \le C \lVert a \rVert_{\Dot{B}_{p,1}^{\frac{d}{p}}} \lVert b \rVert_{\Dot{B}_{p,\infty}^{s}}.
	     \end{equation}
	In the case $ s \ge d/p $,
    the following inequality also holds true:
      \begin{equation}
       \label{product_propo3}
        \lVert ab \rVert_{\Dot{B}_{p,r}^{s}}
	         \le C( \lVert a \rVert_{\Dot{B}_{p,1}^{\frac{d}{p}}} \lVert b \rVert_{\Dot{B}_{p,r}^{s}}+ \lVert a \rVert_{\Dot{B}_{p,r}^{s}} \lVert b \rVert_{\Dot{B}_{p,1}^{\frac{d}{p}}} ).
	     \end{equation}
	  In the nonhomogeneous case, if $ s > d/p $ (or $s= d/p $ and $r=1$), 
	     \begin{equation}
	         \label{product_propo3:chap2}
	         \NB{s}{p}{r}{ab} \le C\NB{s}{p}{r}{a}\NB{s}{p}{r}{b}.
	     \end{equation}
       Finally, if $(p_1,p_2)\in [1,\infty]^2$ with  $-\frac{d}{p_2}<s\le \frac{d}{p_1} \esp{and} 2\le p_1\le p_2$, then, for $i=1,2$
      \begin{equation}
	          \label{product_propo:p1:p2}
	           \lVert ab \rVert_{\Dot{B}_{p_i,r}^{s}}
	         \le C( \lVert a \rVert_{\Dot{B}_{p_2,1}^{\frac{d}{p_2}}}\ \lVert b \rVert_{\Dot{B}_{p_i,r}^{s}}+  \lVert a \rVert_{\Dot{B}_{p_i,r}^{\frac{d}{p_i}}}\lVert a \rVert_{\Dot{B}_{p_2,1}^{s-\frac{d}{p_i}+\frac{d}{p_2}}}.
	     \end{equation}
	 \end{propo}
     \begin{proof}
       We provide the proof of these inequalities only in the homogeneous case. The proof of inequalities \eqref{product_propo1}, \eqref{product_propo2}, \eqref{product_propo:inf:p} and \eqref{product_propo3} is a direct consequence of paradifferential Calculus and Bony's decomposition 
         \begin{align}
             \label{bony:decom:app}
           ab=   \Dot{R}(a, b)+\Dot{T}_{b} a+\Dot{T}_{a} b\cdot
         \end{align}
         Hence, we will concentrate on the proof of Inequality \eqref{product_propo:p1:p2}. 
         Obviously, for  $i=2,$ it is just \eqref{product_propo3}
         with $p=p_2$ and $r=1.$
          Next, since $s+\frac{d}{p_2}>0,$ setting $1/\lambda=1/p_1+1/p_2$ then using 
          $$\Dot{R}:\BH{s}{p_1}{1}\times \BH{\frac{d}{p_2}}{p_2}{1} \longrightarrow \BH{s+\cdp{p_2}}{\lambda}{1} \esp{and} \BH{s+\frac d{p_2}}{\lambda}{1} \hookrightarrow  \BH{s}{p_1}{1}$$ allows us to get the desired bound  for the first term of \eqref{bony:decom:app}.
          
          To bound the second term, since $s\le \frac{d}{p_1}$, it suffices to combine 
          $$\Dot{T}:\BH{s-\frac{d}{p_1}}{\infty}{1}\times \BH{\frac{d}{p_1}}{p_1}{1} \longrightarrow \BH{s}{p_1}{1} \esp{and} \BH{s-\cdp{p_1}+\cdp{p_2}}{p_2}{1} \hookrightarrow  \BH{s-\cdp{p_1}}{\infty}{1}.$$ The last term of \eqref{bony:decom:app}  can be bounded in $\BH{s}{p_1}{1}$ with the help of 
          \begin{equation*}
          \Dot{T}:L^\infty\times  \BH{s}{p_1}{1}\longrightarrow \BH{s}{p_1}{1} \esp{and} \BH{\frac{d}{p_2}}{p_2}{1}\hookrightarrow L^\infty.\qedhere
          \end{equation*}
     \end{proof}
     
     \begin{rmq}
      We have similar results for the spaces  $L^{a}_T(\BH{s}{p}{r})$ and $\widetilde{L}^{a}_T(\BH{s}{p}{r})$ see \cite{Dan01loc,Haspot11}.
     \end{rmq}
	 
  Among the results necessary to prove our Theorems, we have the following  one.
  \begin{propo}
\label{propo:comp}
     Let $f$ be a smooth function on $\mathbb{R}$ with $f(0)=0$. Let $(p,r)\in \times[1,\infty)\times [1,\infty) $  and
     $s\in\mathbb R$ satisfy:
    \begin{equation}
        \label{cond:on:p:r:s}
        s<\frac{d}{p}, \esp{or } s=\frac{d}{p}\esp{and} r=1.
    \end{equation}
    If in addition $s>0$ then, for any real function $u$ in $L^\infty \cap \Dot{B}^{s}_{p,r} $ we have $f\circ u$ in $L^\infty\cap \Dot{B}^{s}_{p,r}$ and 
    \begin{align}
          \label{comp:u:propo:p}
        \lVert f\circ u \rVert_{\Dot{B}^{s}_{p,r}}   \le C(f^{\prime},\normeinf{u})  \lVert u \rVert_{\Dot{B}^{s}_{p,r}}.
      \end{align}
       More generally, 
      \begin{align}
          \label{comp:uv:propo:11}
           \lVert f\circ u- f\circ v \rVert_{\Dot{B}^{s}_{p,r}}
           \le C(f^{\prime}, \normeinf{(u,v)}) (1+  \lVert( u,v) \rVert_{\Dot{B}^{\frac{d}{p}}_{p,1}} )  \lVert u-v \rVert_{\Dot{B}^{s}_{p,r}} .
      \end{align}
   Next, under conditions $s>-\min(d/p,d/p^{\prime}) $ and \eqref{cond:on:p:r:s}, we have
      \begin{align}
          \label{comp:uv:propo:1}
         \lVert f\circ u- f\circ v \rVert_{\Dot{B}^{s}_{p,r}}    \le C(f^{\prime}, \normeinf{(u,v)}) (1+  \lVert( u,v) \rVert_{\Dot{B}^{\frac{d}{p}}_{p,1}} )  \lVert u-v \rVert_{\Dot{B}^{s}_{p,r}}  .
      \end{align}
      Furthermore, if $-\min(d/p,d/p^{\prime}) \le s < \frac{d}{p}$ then the last  inequality remains valid for $r=\infty$:
         \begin{align}
          \label{comp:uv:propo:inf}
         \lVert f\circ u- f\circ v \rVert_{\Dot{B}^{s}_{p,\infty}}    \le C(f^{\prime}, \normeinf{(u,v)}) (1+  \lVert( u,v) \rVert_{\Dot{B}^{\frac{d}{p}}_{p,1}} )  \lVert u-v \rVert_{\Dot{B}^{s}_{p,\infty}}.
      \end{align}
       Similar results holds for  Chemin-Lerner spaces.
\end{propo}

\begin{proof}
    See \cite{DanADOlocal}.
\end{proof}
The following composition estimates
    enable us to handle some $0-$th order terms.
  \begin{propo}
	 \label{propo_Compo_u_v_BH:loc}
  Let $0\le n_1\le n$ and $m$ be three integers. 	  Let $f: (X,Y)\in \mathbb{R}^{n_1}\times\mathbb{R}^{n-n_1} \mapsto f(X,Y)\in\mathbb{R}^m $ be a smooth function on $\mathbb{R}^n$. Assume that $f$ is at most linear with respect to $ Y $ (that is, the second derivative $D^2_Y f$ of $f$ with respect to the component $ Y$ vanishes).\\
  If $f$  vanishes at $0_{\mathbb{R}^n}$, then for any $0<s\le \cd $  the following inequality holds true
  \begin{align}
      \NBH{s}{2}{1}{f(u,v)}\le   C(f^{\prime}, \normeinf{u})( \NBH{s}{2}{1}{ v}(1+\NBH{\cd}{2}{1}{u})+\NBH{s}{2}{1}{ u} ) . 
  \end{align}
  Furthermore if $-\cd< s\le \cd,$ then we have
    for some  $C= C(f^{\prime}, \normeinf{u_1,u_2})$: 
    \begin{multline}
      \label{comp:u_v_1:2:BH:ine:prop:1:loc}
      \NBH{s}{2}{1}{f(u_1,v_1)-f(u_2, v_2)}\le  C  \NBH{s}{2}{1}{ v_2-v_1}(1+\NBH{\cd}{2}{1}{u_2})\\
      +C(1+\NBH{\cd}{2}{1}{u_1}+\NBH{\cd}{2}{1}{u_2})\left( \NBH{\cd}{2}{1}{u_2-u_1}\NBH{s}{2}{1}{ v_1}+ \NBH{s}{2}{1}{ u_1-u_2}\right)\cdotp
  \end{multline}
    Finally  if   $ -\cd\le  s< \frac{d}{2}$ then we have
  \begin{multline}
      \label{comp:u_v_1:2:BH:ine:prop:inf:loc}
      \NBH{s}{2}{\infty}{f(u_1,v_1)-f(u_2, v_2)}\le  C  \NBH{s}{2}{\infty}{ v_2-v_1}(1+\NBH{\cd}{2}{1}{u_2})\\
      +C(1+\NBH{\cd}{2}{1}{u_1}+\NBH{\cd}{2}{1}{u_2}) \left(\NBH{s}{2}{\infty}{u_2-u_1}\NBH{\cd}{2}{1}{ v_1}+ \NBH{s}{2}{\infty}{ u_1-u_2}\right),
      \end{multline}
      where $C= C(f^{\prime}, \normeinf{u_1,u_2}) $.
	 \end{propo}
	\begin{proof}
	    The Taylor formula and the fact  $  D^2_Y f=0   $ guarantee the existence of two smooth functions $\Lambda$ and $\Gamma$ defined on $\mathbb{R}^{n_1}$ such that
     \[    f(u,v) = \Lambda(u) v+ \Gamma(u) \esp{(with} \Gamma(0_{\mathbb{R}^{n_1}})=0_{\mathbb{R}^m}\;\; \text{if}\;\; f(0_{\mathbb{R}^{n_1}},0_{\mathbb{R}^{n-n_1}})=0_{\mathbb{R}^{m}} .   \]


If  $ 0<s\le \frac{d}{2}$ then, applying the inequalities \eqref{product_propo2} and \eqref{comp:u:propo:p} 
to the term $ \Lambda(u) v $ yields the first term of the right-hand side of  inequality \eqref{comp:u:propo:p}.
Next, using Proposition \ref{propo:comp} (recall that $s>0$) for the term $  \Lambda(u) $ gives the second term of \eqref{comp:u:propo:p}.   

 To prove \eqref{comp:u_v_1:2:BH:ine:prop:1:loc}  (resp.\eqref{comp:u_v_1:2:BH:ine:prop:inf:loc}), we use  the above decomposition to get
     \[  f(u_2,v_2)- f(u_1,v_1)= \Lambda(u_2)(v_2-v_1) +(\Lambda(u_2)-\Lambda(u_1))v_1 +(\Gamma(u_2)-\Gamma(u_1)).   \]
  Having this decomposition in hand, the first two terms of the last equality may be handled by  the  inequalities \eqref{product_propo2}  (resp. \eqref{product_propo:inf:p})  and  \eqref{comp:u:propo:p} (resp. \eqref{comp:uv:propo:inf}). Concerning the last one, we use Inequality \eqref{comp:uv:propo:1} (resp. \eqref{comp:uv:propo:inf}) .
	\end{proof}

	  In homogeneous case, the Proposition \ref{propo_Compo_u_v_BH:loc} states.
  \begin{propo}
	 \label{propo_Compo_u_v_BH:glo}
  Let $0\le n_1\le n$ be two integers. 	     Let $f: (X,Y)\in \mathbb{R}^{n_1}\times\mathbb{R}^{n-n_1} \mapsto f(X,Y) $ be a smooth function on $\mathbb{R}^n$ vanishing at $0$. Assume that $f$ is at most linear with respect to $ Y $ (that is, the second derivative $D^2_Y f$ of $f$ with respect to the component $ Y$ vanishes).
 
  Then, for any $0<s\le \cd $  the following inequalities hold true:
  \begin{align}
      \label{comp:u_v:BH:ine:prop:1:glo}
      \NB{s}{2}{1}{f(u,v)}\le   C(f^{\prime}, \normeinf{u})( \NB{s}{2}{1}{ v}(1+\NB{\cd}{2}{1}{u})+\NB{s}{2}{1}{ u} ) . 
  \end{align}
  Furthermore if $-\cd< s\le \cd$ we have
     \begin{multline}
      \label{comp:u_v_1:2:BH:ine:prop:1:glo}
    \NB{s}{2}{1}{f(u_1,v_1)-f(u_2, v_2)}\le  C  \NB{s}{2}{1}{ v_2-v_1}(1+\NB{\cd}{2}{1}{u_2})\\
      +C(1+\NB{\cd}{2}{1}{u_1}+\NB{\cd}{2}{1}{u_2})\Bigl( \NB{\cd}{2}{1}{u_2-u_1}\NB{s}{2}{1}{ v_1}+ \NB{s}{2}{1}{ u_1-u_2}\Bigr)\cdotp
  \end{multline}
    Finally,  if   $ -\cd\le  s< \frac{d}{2}$ then we have
  \begin{multline}
       \label{comp:u_v_1:2:BH:ine:prop:inf:glo}
      \NB{s}{2}{\infty}{f(u_1,v_1)-f(u_2, v_2)}\le  C  \NB{s}{2}{\infty}{ v_2-v_1}(1+\NB{\cd}{2}{1}{u_2})\\
      +C(1+\NB{\cd}{2}{1}{u_1}+\NB{\cd}{2}{1}{u_2}) \Bigl(\NB{s}{2}{\infty}{u_2-u_1}\NB{\cd}{2}{1}{ v_1}+ \NB{s}{2}{\infty}{ u_1-u_2}\Bigr)\cdotp
         \end{multline}
  In the above two inequalities, $C= C(f^{\prime}, \normeinf{u_1,u_2}) $.
	 \end{propo}
\bibliographystyle{plain} 
\bibliography{biblio}

\begin{thebibliography}{10}

\bibitem{Adogbo25}
J.-P. Adogbo.
\newblock {\em Around the problem of global existence for compressible and viscous fluid mechanics systems with or without capillarity}.
\newblock PhD thesis, Université Paris Dauphine-PSL, 2025.

\bibitem{DanADOglob}
J.-P. Adogbo and R.~Danchin.
\newblock Global existence for multi-dimensional partially diffusive systems.
\newblock {\em Journal of Differential Equations}, 445:113596, 2025.

\bibitem{DanADOlocal}
J.-P. Adogbo and R.~Danchin.
\newblock Local well-posedness in the critical regularity setting for hyperbolic systems with partial diffusion.
\newblock {\em Adv. Differential Equations}, 31((3/4)):161 -- 218, March/April 2026.

\bibitem{HajDanChe11}
H.~Bahouri, J.-Y. Chemin, and R.~Danchin.
\newblock {\em Fourier Analysis and Nonlinear Partial Differential Equations}, volume 343.
\newblock Springer, 2011.

\bibitem{Beau11}
K.~Beauchard and E.~Zuazua.
\newblock Large time asymptotics for partially dissipative hyperbolic system.
\newblock {\em Arch. Rational Mech. Anal}, 199:177--227, 2011.

\bibitem{Bren66}
Ph. Brenner.
\newblock The {C}auchy problem for symmetric hyperbolic systems in {$L\sb{p}$}.
\newblock {\em Math. Scand.}, 19:27--37, 1966.

\bibitem{BurCrinTan23}
C.~Burtea, T.~Crin-Barat, and J.~Tan.
\newblock Pressure-relaxation limit for a one-velocity {B}aer-{N}unziato model to a {K}apila model.
\newblock {\em Math. Models Methods Appl. Sci.}, 33(4):687--753, 2023.

\bibitem{DanCha10}
F.~Charve and R.~Danchin.
\newblock A global existence result for the compressible navier-stokes equations in the critical ${L}^p$ framework.
\newblock {\em Arch. Rational Mech. Anal.}, 198:233--271, 2010.

\bibitem{BaratDan22M}
T.~Crin-Barat and R.~Danchin.
\newblock Partially dissipative hyperbolic systems in the critical regularity setting: the multi-dimensional case.
\newblock {\em J. Math. Pures Appl. (9)}, 165:1--41, 2022.

\bibitem{BaratDan22D1}
T.~Crin-Barat and R.~Danchin.
\newblock Partially dissipative one-dimensional hyperbolic systems in the critical regularity setting, and applications.
\newblock {\em Pure and Applied Analysis}, 4(1):85--125, 2022.

\bibitem{BaratDan23}
T.~Crin-Barat and R.~Danchin.
\newblock Global existence for partially dissipative hyperbolic systems in the {${L}^p$} framework, and relaxation limit.
\newblock {\em Math. Ann.}, 386(3-4):2159--2206, 2023.

\bibitem{Dan01loc}
R.~Danchin.
\newblock Local theory in critical spaces for compressible viscous and heat conductive gases.
\newblock {\em Communications in Partial Differential Equations}, 26:1183--1233, 2001.

\bibitem{Danc22}
R.~Danchin.
\newblock Partially dissipative systems in the critical regularity setting, and strong relaxation limit.
\newblock {\em EMS Surv. Math. Sci.}, 9(1):135--192, 2022.

\bibitem{DanMucha24}
R.~Danchin and P.~B. Mucha.
\newblock The compressible {E}uler system with nonlocal pressure: global existence and relaxation.
\newblock {\em Calculus of Variations and Partial Differential Equations}, 63(6):148, 2024.

\bibitem{Gues06}
O.~Gues and J.~Rauch.
\newblock Hyperbolic {${L}^p$} multipliers are translations.
\newblock {\em Comm. Partial Differential Equations}, 31(1-3):431--443, 2006.

\bibitem{Haspot11}
B.~Haspot.
\newblock Existence of global strong solutions in critical spaces for barotropic viscous fluids.
\newblock {\em Archive for Rational Mechanics and Analysis}, 202(2):427--460, 2011.

\bibitem{hoff2006uniqueness}
D.~Hoff.
\newblock Uniqueness of weak solutions of the {N}avier--{S}tokes equations of multidimensional, compressible flow.
\newblock {\em SIAM journal on mathematical analysis}, 37(6):1742--1760, 2006.

\bibitem{Kawashima83}
S.~Kawashima.
\newblock {\em Systems of a hyperbolic parabolic type with applications to the equations of magnetohydrodynamics.}
\newblock PhD thesis, Kyoto University, 1983.

\bibitem{lemarie2023parabolic}
V.~Lemari{\'e}.
\newblock Parabolic-elliptic {K}eller--{S}egel system.
\newblock {\em Tunisian Journal of Mathematics}, 6(3):503--542, 2024.

\bibitem{MatsmuraNishida80}
A.~Matsumura and T.~Nishida.
\newblock The initial value problem for the equations of motion of viscous and heat-conductive gases.
\newblock {\em J. Math. Kyoto Univ.}, 20:67--104, 1980.

\bibitem{ChenMia10}
C.~Miao Q.~Chen and Z.~Zhang.
\newblock Global well-posedness for compressible navier-stokes equations with highly oscillating initial velocity.
\newblock {\em Comm. Pure Appl. Math.}, 63(9):1173--1224, 2010.

\bibitem{Serre97}
D.~Serre.
\newblock {\em Syst\`emes de lois de conservation. {I}}.
\newblock Fondations. Diderot Editeur, Paris, 1996.
\newblock Hyperbolicit\'{e}, entropies, ondes de choc.

\bibitem{Serr10}
D.~Serre.
\newblock Local existence for viscous system of conservation laws: {$H^s$}-data with {$s>1+d/2$}.
\newblock In {\em Nonlinear partial differential equations and hyperbolic wave phenomena}, volume 526 of {\em Contemp. Math.}, pages 339--358. Amer. Math. Soc., Providence, RI, 2010.

\bibitem{Serre09}
D.~Serre.
\newblock The structure of dissipative viscous system of conservation laws.
\newblock {\em Phys. D}, 239(15):1381--1386, 2010.

\bibitem{SK85}
S.~Shizuta and S.~Kawashima.
\newblock Systems of equations of hyperbolic-parabolic type with applications to the discrete boltzmann equation.
\newblock {\em Hokkaido Math J.}, 14:249--275, 1985.

\bibitem{Yong04}
W-A Yong.
\newblock Entropy and global existence for hyperbolic balance laws.
\newblock {\em Arch. Rational Mech. Anal}, 172:47--266, 2004.

\end{thebibliography}

\end{document}